\documentclass[11pt,reqno]{amsart}
\usepackage{tikz}
\usepackage{epstopdf}
\usepackage{epsfig}
\usepackage{caption}
\usepackage{subcaption}
\usepackage{math}
\usepackage{tikz}
\usepackage{tikz-cd}
\usetikzlibrary{shapes,arrows}
\usepackage{cleveref}
\usepackage{algorithm}
\usepackage{algorithmic}
\tikzstyle{decision} = [diamond, draw, fill=blue!20, 
    text width=4.5em, text badly centered, node distance=3cm, inner sep=0pt]
\tikzstyle{block} = [rectangle, draw, fill=blue!20, 
    text width=5em, text centered, rounded corners, minimum height=4em]
\tikzstyle{line} = [draw, -latex']
\tikzstyle{cloud} = [draw, ellipse,fill=red!20, node distance=3cm,
    minimum height=2em]   
\usetikzlibrary{positioning}
\tikzset{main node/.style={circle,fill=blue!20,draw,minimum size=1cm,inner sep=0pt},  }

\newcommand{\ts}{\mathsf{T}}
\newcommand{\NT}{\mathrm{W}}
\newcommand{\JX}[1]{{\color{blue}{$^{\text{JX}}$[#1]}}}

\newcommand{\xz}[1]{{\color{red}{$^{\text{XZ}}$[#1]}}}

\begin{document}
\title[Numerical analysis on Neural projected dynamics]{Numerical analysis on Neural network projected schemes for approximating one dimensional Wasserstein Gradient flows}
\author[Zuo]{Xinzhe Zuo}
\email{zxz@math.ucla.edu}
\address{Department of Mathematics, University of California, Los Angeles, CA, 90095.}
\author[Zhao]{Jiaxi Zhao}
\email{jiaxi.zhao@u.nus.edu}
\address{Department of Mathematics, National University of Singapore.}
\author[Liu]{Shu Liu}
\email{shuliu@math.ucla.edu}
\address{Department of Mathematics, University of California, Los Angeles, CA, 90095.}
\author[Osher]{Stanley Osher}
\email{sjo@math.ucla.edu}
\address{Department of Mathematics, University of California, Los Angeles, CA, 90095.}
\author[Li]{Wuchen Li}
\email{wuchen@mailbox.sc.edu}
\address{Department of Mathematics, University of South Carolina, Columbia, SC, 29208.}
\begin{abstract}
We provide a numerical analysis and computation of neural network projected schemes for approximating one dimensional Wasserstein gradient flows. We approximate the Lagrangian mapping functions of gradient flows by the class of two-layer neural network functions with ReLU (rectified linear unit) activation functions. The numerical scheme is based on a projected gradient method, namely the Wasserstein natural gradient, where the projection is constructed from the $L^2$ mapping spaces onto the neural network parameterized mapping space. We establish theoretical guarantees for the performance of the neural projected dynamics. We derive a closed-form update for the scheme with well-posedness and explicit consistency guarantee for a particular choice of network structure. General truncation error analysis is also established on the basis of the projective nature of the dynamics. Numerical examples, including gradient drift Fokker-Planck equations, porous medium equations, and Keller-Segel models, verify the accuracy and effectiveness of the proposed neural projected algorithm. 
\end{abstract}
\thanks{ Xinzhe Zuo and Jiaxi Zhao contributed equally. Jiaxi Zhao, Xinzhe Zuo, and Shu Liu are partially supported by AFOSR YIP award No. FA9550-23-1-008; Xinzhe Zuo, Shu Liu, and Stanley Osher are partially funded by AFOSR MURI FA9550-18-502 and ONR N00014-20-1-2787; Wuchen Li’s work is partially supported by AFOSR YIP award No. FA9550-23-1-008, NSF DMS-2245097, and NSF RTG: 2038080.}
\maketitle

\section{Introduction}
Simulating gradient flows of free energies is a central problem in the computational physics of complex systems \cite{carrillo2021lagrangian} and data science \cite{amari1998natural,amari2016information}. In physics, gradient flows often arise from first-order principles, such as the Onsager principle \cite{Ons}. The Onsager gradient flows are widely used in phase fields, chemistry, and biology modeling. In recent years, a particular type of Onsager gradient flow, known as Wasserstein gradient flow, has been widely studied in optimal transport communities \cite{ambrosio2005gradient,otto2001,vil2008}. It studies an infinite-dimensional pseudo-Riemannian metric in the probability distribution space known as the density manifold. The gradient flow in the Wasserstein space naturally captures the free energy dissipation properties. Depending on the choices of free energies, the Wasserstein gradient flow contains a vast class of differential equations, such as gradient drift Fokker-Planck equations, porous medium equations, and Keller-Segel models. These models are widely used in population dynamics and sampling-related optimization problems. 

In recent years, machine learning has brought a class of new methods in computational physics, where free energies are identified with the loss functions \cite{chu2020equivalence, nusken2023geometry}. Meanwhile, computing Wasserstein gradient flows of loss functions in terms of samples also finds their various applications, such as generative artificial intelligence \cite{pmlr-v70-arjovsky17a} and transport map-based sampling methods \cite{DBLP:conf/iclr/0011SKKEP21}. In these applications, one often relies on the Lagrangian mapping functions to describe the Wasserstein gradient flows and deep neural networks to approximate the mapping functions due to their high expressivity and adaptivity from the compositional structure. While empirical successes of this framework have been observed in various applications \cite {DBLP:conf/iclr/0011SKKEP21, pmlr-v70-arjovsky17a}, very few theoretical results exist to explain the underlying mechanism.

Moreover, projected dynamics in neural network space are widely used to approximate Wasserstein gradient flows \cite{gaby2023neural,doi:10.1137/20M1344986}. These dynamics restrict the space of probabilities onto a finite-dimensional subspace parameterized by neural network mapping functions. For this reason, we call it the neural projected gradient dynamics. This approach originates from the natural gradient method in information geometry \cite{amari1998natural} and extends the framework set by \cite{li2018natural}. Some basic questions about its accuracy and efficiency remain: \textit{Even in one-dimensional space, how well do the neural projected dynamics approximate the Wasserstein gradient flow? What is the accuracy of the neural network approximation in Lagrangian mapping functions?} 

In this paper, we study the numerical analysis and computational neural network projected schemes for one-dimensional Wasserstein gradient flows. 
The main result is sketched below. By formulating gradient flows in Lagrangian coordinates, the proposed numerical scheme takes the form of a `preconditioned' gradient descent, where the preconditioner is the metric tensor of the statistical manifold of the parameter space. Theoretically, we first provide the derivation of the analytic solution for the inverse neural mapping metric. It is based on a special class of the ReLU network in \cref{thm:ReLU-inverse-metric}. We use the analytic form of the projected gradient flow formula to prove the consistency of the numerical scheme. Then, we prove in \cref{thm:consistent} that the numerical schemes derived from the neural projected dynamics are of first or second-order consistency for the general Wasserstein gradient directions. These include cases of the heat flow and the Fokker-Planck equation. Furthermore, viewing our neural network model as a moving mesh method, we show in \cref{prop:mesh-quality-analysis} that the mesh will not degenerate during the simulation. 

In numerics, the advantages of the proposed method are twofold. First, using a two-layer neural network as our basis function, the proposed method can be regarded as a `moving-mesh' method in Lagrangian coordinate, which demonstrates very promising performance even when the number of parameters of the neural network is very limited. In particular, our numerical examples can achieve an accuracy of $10^{-3}$ with less than 100 neurons. Second, using the Wasserstein gradient flow formulation, the proposed method is very easy to implement since it can make use of the automatic differentiation feature from popular machine learning libraries such as PyTorch.

Nowadays, the computation of Wasserstein gradient flows (WGFs) has attracted great interests from researchers in various communities such as mathematics, physics, statistics, and machine learning.  Classical numerical methods \cite{chang1970practical} have been introduced to directly evaluate the probability density function. Recently, algorithms that approximate the Lagrangian mapping functions associated with WGFs have been invented. We refer the readers to \cite{carrillo2021lagrangian} and references therein for related discussions. These treatments automatically preserve non-negativity and total mass. Together with the fast-developing deep learning techniques, they inspire a series of research on composing scalable, sampling-friendly computational methods for WGFs in higher-dimensional spaces \cite{doi:10.1137/20M1344986, mokrov2021large, fan2021variational, hu2022energetic, lee2023deep}. Recently, deep learning-based algorithms for computing the Lagrangian coordinates of the Wasserstein Hamiltonian flows, or more generally mean field control problems, have also been introduced in \cite{wu2023parameterized, neklyudov2023computational,ruthotto2020machine}.

Our treatment of projecting the WGFs onto the parameter space is also known as the natural gradient method, which are first introduced in \cite{amari1998natural} (w.r.t. Fisher-Rao metric) and \cite{chen2020optimal} (w.r.t. Wasserstein metric). Here the projected matrix is often named information matrix, namely Fisher information matrix and Wasserstein information matrix, depending on the usage of metrics in probability space. This method recently finds its application in large-scale optimization problems \cite{nurbekyan2023efficient}. In recent research \cite{du2021evolutional, bruna2024neural,gaby2023neural}, the authors aim to calculate general evolution equations by directly leveraging the neural network representation of the time-dependent solution. They endow the evolution of the equation in the functional space into the parameter space of the neural network to obtain a finite-dimensional ordinary differential equation, which can be readily integrated via the Runge-Kutta solvers. Numerical properties of the ReLU neural network families have been investigated in \cite{sci_2020}. 

Compared to previous studies, we study the numerical analysis of neural network projected dynamics for approximating WGFs. In one-dimensional space, we provide the error analysis for the neural projected dynamics with a two-layer neural network. We numerically verify the proposed error analysis. In particular, we formulate a class of explicit schemes from the neural network projected dynamics. This study continues the study of the Wasserstein information matrix on neural network models; see related discussions in \cite{li2023wasserstein, li2023scaling, doi:10.1137/20M1344986}.

The paper is organized as follows. In Section \ref{sec2}, we briefly review the formulation of Wasserstein gradient flows of free energies in both Eulerian and Lagrangian coordinates. We formulate the projected Wasserstein gradient flows over neural network models in Section \ref{sec3}. In Section \ref{sec4}, we conduct the numerical analysis of the proposed neural projected dynamics in two-layer neural network functions. In Section \ref{sec5}, we verify the accuracy of the proposed algorithm with numerical examples in Fokker-Planck equations, porous medium equations, and Keller-Segel models. 

\section{Review of Wasserstein gradient flows and Lagrangian coordinates}\label{sec2}
In this section, we prepare the theoretical foundations of Wasserstein gradient flows with a focus on Lagrangian description (diffeomorphism mapping functions) and the associated microscopic particle dynamics. See details in \cite{ambrosio2005gradient,vil2008}. 

\subsection{Wasserstein gradient flows}\label{subsection: intro WGF }
Suppose $\Omega$ is a domain in the Euclidean space $\mathbb{R}^d$. Denote the probability space
\[ \mathcal{P}(\Omega) = \left\{p(\cdot)\in C^{\infty}:~ \int_\Omega p(x)dx = 1, \quad p(\cdot)\geq 0\right\}.\]
Given an energy functional $\mathcal{F}(\cdot):\Omega\rightarrow \mathbb{R}$, we consider the following evolution equation associated with $\mathcal{F}(\cdot)$, 
\begin{equation}
\partial_t p(t,x)=\nabla_x\cdot(p(t,x)\nabla_x\frac{\delta}{\delta p}\mathcal{F}(p)), \quad p(\cdot, 0) = p_0,  \label{WGF}
\end{equation}
with Neumann boundary condition $p(t, x)\nabla_x\frac{\delta}{\delta p}\mathcal F(p)\cdot \boldsymbol{n} = 0$ where $\boldsymbol{n}$ is the outward pointing vector on boundary $\partial \Omega$. $\frac{\delta}{\delta p}$ is the $L^2$ first variation operator w.r.t. density variable $p$. The mass of $p(t, \cdot)$ is conserved and always equals $1$.
An important fact about \eqref{WGF} is that this equation can be treated as the gradient flow of $\mathcal{F}$ on $\mathcal{P}(\Omega)$. To be more specific, by endowing the probability space $\mathcal P(\Omega)$ with the $L^2$ Wasserstein metric $g_W$, we can view $(\mathcal P(\Omega), g_W)$ as a Riemannian manifold, and \eqref{WGF} is the gradient flow on such manifold with respect to $g_W$. 

Let us briefly review several facts. We first define the metric $g_W$ at arbitrary $p\in\mathcal P(\Omega)$, which is identified via the continuity equation (that is, tangent vectors) whose driving vector field belongs to the closure of all gradient fields $\nabla_x\psi:\Omega\rightarrow \mathbb{R}^d$ with $\psi\in C^\infty(\Omega)$ in $L^2(p)$-norm. Consider a smooth curve $\{p_i(t,\cdot)\}_{t\in (-\epsilon, \epsilon)}$ ($i=1,2$) passing through $p$ at $t=0$ on $\mathcal P(\Omega)$. Suppose the probability evolution $p_i(t, \cdot)$ is driven by the gradient field $\nabla_x\psi_i(\cdot)$ at $t=0$, i.e., $\psi_i(\cdot)$ solves
\[ \partial_t p_i (0, x) + \nabla_x\cdot(p_i(0, x)\nabla\psi_i(x)) = 0, \quad i=1,2.\]
We define the $L^2$ Wasserstein metric $g_W(\cdot, \cdot)$ at $p$ as a symmetric, positive-definite bilinear form, 
\[ g_W(\partial_t p_1(0,\cdot), \partial_t p_2(0,\cdot)) = \int_\Omega \nabla_x\psi(x) \cdot \nabla_x \psi_2(x) p(x)~dx. \]

Recall the definition of the gradient of a smooth function $f$ on a Riemannian manifold $(M, g)$ as
\[ g(\mathrm{grad} f(x), \dot x(0)) = \frac{d}{dt} f(x(t)), \]
for any smooth curves $\{x(t)\}_{}t\in(-\epsilon, \epsilon)$ passing through $x$ at $t=0$. Switching back to our case, for the functional $\mathcal F$ defined on $(\mathcal P(\Omega), g_W)$, we define the gradient of $\mathcal F$ w.r.t. Wasserstein metric $g_W$ at $p$ as
\[ g_W(\textrm{grad}_W\mathcal F(p), \partial_tp(0, \cdot)) = \frac{d}{dt}\mathcal F(p(t, \cdot))\Bigg|_{t=0}. \]
Here $\{ p(t, \cdot)\}_{t\in(-\epsilon, \epsilon)}$ is arbitrary curve on $\mathcal P(\Omega)$ with $p(0, \cdot)=p(\cdot)$. Suppose $ p(t, \cdot)$ is guided by the gradient field $\nabla_x\psi$ at time $t=0,$ Then the right-hand side can be computed as
\begin{align*} 
  \frac{d}{dt} \mathcal F(p(t, \cdot)) = \int_\Omega  \frac{\delta \mathcal{F}( p(0, \cdot))}{\delta p}(x) \partial_t p(0, x)~dx & = \int_\Omega \frac{\delta \mathcal{F}(p)}{\delta p}(x) (-\nabla_x\cdot(p(x)\nabla_x\psi(x))) ~ dx \\
  & = \int_\Omega \nabla \frac{\delta \mathcal{F}(p)}{\delta p}(x) \cdot \nabla_x \psi(x) p(x)dx. 
\end{align*}
Recall the definition of the metric $g_W$, it is not difficult to verify that the gradient field associated with $\mathrm{grad}_W \mathcal F(p)$ is $\nabla_x \frac{\delta}{\delta p}\mathcal F(p)$. Thus, 
\[ \mathrm{grad}_W \mathcal F(p) = -\nabla_x \cdot(p(t, x)\nabla_x \frac{\delta \mathcal F(p)}{\delta p}(x)), \]
and the Wasserstein gradient flow $\partial_t p = -\mathrm{grad}_W \mathcal F(p)$ can be formulated as equation \eqref{WGF}.

We provide several examples of WGFs. In these examples, we assume $\Omega = \mathbb{R}^d$.
\begin{itemize}[leftmargin=*]
\item (Fokker-Planck equation) Consider
\begin{align*}
  \mathcal F(p) = \int_\Omega V(x) p(x)dx + \gamma \int_\Omega  p(x)\log p(x)dx.
\end{align*}
Then the Wasserstein gradient of $\mathcal F$ equals
\begin{equation*}
\begin{split}
     \mathrm{grad}_W \mathcal F(p) =& -\nabla_x \cdot(p(x)\nabla_x (V(x) + \gamma(\log p(x)+1)))\\ 
     =& -\nabla\cdot(p(x)\nabla_x V(x)) - \gamma \Delta_x p(x). 
\end{split}
     \end{equation*} 
The corresponding WGF is the Fokker-Planck equation
\begin{equation}
  \partial_t p(t,x) = \nabla_x \cdot(p(t, x) \nabla_x V(x)) + \gamma\Delta_x  p(t, x).\label{Fokker-Planck equ }
\end{equation}

\item (Porous medium equation) Consider
\begin{equation*}
  \mathcal F(p) = \frac{p^m}{m-1}.
\end{equation*}
One computes 
\[ \mathrm{grad}_W \mathcal F(p) = -\nabla_x \cdot (p(t, x) \nabla_x ( \frac{m}{m-1} p(x)^{m-1} )) = -\nabla_x\cdot (\nabla_x(p(x)^m)) = -\Delta_x  p(x)^m. \]
Thus, the corresponding WGF yields the porous medium equation
\begin{equation}
    \partial_t p(t, x) = \Delta_x p(t, x)^m.  \label{Porous medium equ }
\end{equation}

\item (Keller-Segel equation) Another well-known WGF is by choosing $\mathcal{F}$ as the sum of the internal energy and the interaction energy
\begin{equation*}
  \mathcal F(p) = \int_\Omega U(p(x))~dx + \frac12\iint_{\Omega \times \Omega } W(|x-y|)p(x)p(y)~dxdy,
\end{equation*}
where $U$ is a certain smooth function defined on $\mathbb{R}_{+}$, and $W(\cdot)\in C(\mathbb{R}_+;\mathbb{R})$ is a kernel function. 

We calculate
\[ \mathrm{grad}_W\mathcal F(p) = -\nabla_x \cdot(p(x)\nabla_x( U'(p(x)) + W*p(x) )), \]
where we denote the convolution $W*p(x) = \int_{\Omega} W(|x-y|)p(y)~dy$. The WGF associated with this functional is the Keller-Segel equation
\begin{equation}
  \partial_t p(t, x) = \nabla_x\cdot(p(t, x) \nabla_x U'(p(t, x))) + \nabla_x\cdot(p(t, x)\nabla_x(W*p_t(x))).    \label{Keller-Segel equ }
\end{equation}

\end{itemize}

\subsection{Lagrangian coordinates \& Particle dynamics}
Consider a mapping function $T\colon Z\rightarrow \Omega$. Here $z\in Z$ is an input space, $\Omega\subset\mathbb{R}^d$ is the domain on which WGF is defined. To alleviate our discussion, we assume $Z=\Omega$. Let us further assume $T\in C^\infty(Z,\Omega)$, and the Jacobian matrix $D_zT(z)$ is non-singular for all $z\in Z$, i.e., $\mathrm{det}(D_zT(z))\neq 0$ on $Z$. This also guarantees that $T$ is injective. Given a smooth reference probability density $p_{\mathrm{r}}\in \mathcal{P}(Z)$, we denote the pushforwarded probability density of $p_r$ by $T$ as 
\begin{equation*}
p=T_\# p_r, 
\end{equation*}
where $T_\#:\mathcal P(Z) \rightarrow \mathcal P(\Omega)$ is the pushforward operator defined as
\[ \int_{\Omega} f(x) T_\# p_r(x) ~dx = \int_{  Z  } f(T(z)) p_r(z)~dz, \quad \textrm{for all } f\circ T \in L^1(p_r). \]
The density function of $p$ satisfies
\begin{equation}\label{MA}
  p(T(z))\mathrm{det}(D_zT(z))=p_{\mathrm{r}}(z). \quad \forall~ z\in Z. \quad \textrm{i.e.,} ~~ p(x) = \frac{p_r}{\mathrm{det}(D_z T)}\circ T^{-1}(x) \quad \forall ~ x\in\Omega. 
\end{equation}
Such pushforward map $T$ used for constructing probability distribution $p$ is usually called the \textit{Lagrangian coordinate}. We now imitate the derivation of the WGF to help formulate its counterpart under the Lagrangian coordinate.

We denote $\mathcal O$ as the space of smooth, $L^2(p_r)$ integrable pushforward maps with non-zeros Jacobian, i.e.,
\begin{equation*}
  \mathcal O = \left\{T\in C^\infty(Z, \Omega) ~ : ~ \mathrm{det}(D_zT)\neq 0, ~~ \int_Z |T(z)|^2p_z(z)~dz<\infty \right\}.
\end{equation*}
Then the pushforward operation $\#:\mathcal O \rightarrow \mathcal P(\Omega)$ introduces a submersion from the space of pushforward maps (diffeomorphisms) to the space of probability densities. 

In order to derive the Wasserstein gradient flows (WGFs) on the space $\mathcal O$ of pushforward maps instead of the probability space $\mathcal P(\Omega)$, we first build up certain metric $\langle\cdot,\cdot\rangle$ on $\mathcal{O}$ that corresponds to the Wasserstein metric $g_W$. As illustrated in \cite{otto2001}, $g_W$ is obtained by pulling back the $L^2(p_r)$ norm on $\mathcal O$ via submersion $\#$. Thus, a way of choosing the metric is
\begin{equation*} 
  \langle \mathbf{u}_1, \mathbf{u}_2 \rangle = \int_Z \mathbf{u}_1(z)\cdot\mathbf{u}_2(z)p_r(z)~dz, \quad \forall ~ \mathbf{u}_1, \mathbf{u}_2\in L^2(p_r)~\bigcap ~ C^\infty(Z , \Omega).
\end{equation*}

Now for any smooth functional $\mathcal F:\mathcal P(\Omega)\rightarrow\mathbb{R}$, the composition $\mathcal F^{\#} 
\triangleq \mathcal F\circ \# :\mathcal O \rightarrow \mathbb R$ defines its corresponding functional on $\mathcal O$. Follow similar arguments presented in \ref{subsection: intro WGF }, we compute the gradient of $\mathcal F^\#$ with respect to the metric $\langle \cdot ,  \cdot \rangle$ as
\[ \mathrm{grad}_{\langle\cdot, \cdot\rangle}\mathcal F^\#(T) = \frac{1}{p_{\mathrm{r}}(\cdot)}\frac{\delta \mathcal{F}^\#(T)}{\delta T}(\cdot). \]
Here, $\frac{\delta}{\delta T}$ is the $L^2(m)$ ($m$ denotes the Lebesgue measure) first variational w.r.t. the pushforward map $T$.

Thus, the gradient flow of $\mathcal{F}^\#$ on $\mathcal O$ is formulated as
\begin{equation*}
  \partial_t T(t, \cdot) = -\mathrm{grad}_{\langle\cdot, \cdot\rangle}\mathcal F^\#(T(t, \cdot)) = - \frac{1}{p_{\mathrm{r}}(\cdot)}\frac{\delta \mathcal{F}^\#(T(t, \cdot))}{\delta T}(\cdot).
\end{equation*}
The variation $\frac{\delta}{\delta T}$ is calculated as
\[ \frac{\delta \mathcal{F}^\#( T )}{\delta T}(z) = \left( \nabla_x \frac{\delta \mathcal{F}(T_\# p_r)}{\delta p}\right)\circ T(z) p_r(z). \]
The above equation can also be written as
\begin{equation}
  \partial_t T(t, z) = - \left( \nabla_x \frac{\delta \mathcal{F}(T(t,\cdot)_\# p_r)}{\delta p}\right)\circ T(t, z).  \label{Lagrangian GF}
\end{equation}
If we denote $p(t, \cdot) = T(t, \cdot)_\# p_r$, one can verify that $p(t, \cdot)$ exactly solves equation \eqref{WGF} for WGF with $p_0=T(0, \cdot)_\# p_r$, which justifies the equivalence between the gradient flow \eqref{Lagrangian GF} in Lagrangian coordinates (i.e., the map $T(t, \cdot)$) and the WGF \eqref{WGF} expressed by using Eulerian coordinate (i.e., the density function $p(t, \cdot)$).

Such gradient flow \eqref{Lagrangian GF} on the space of diffeomorphisms also forms a microscopic picture of particle dynamics of the WGF \eqref{WGF}. For any random reference sample $z\sim p_r$, by setting $\mathbf{x}_t = T(t, z)$, it is not hard to verify that $\mathbf{x}_t$ evolves w.r.t. the dynamic
\begin{equation} 
  \frac{d\mathbf{x}_t}{dt} = -\left(\nabla_x \frac{\delta}{\delta p} \mathcal F( p_t)\right)(\mathbf{x}_t), \quad \mathbf{x}_0=T(0, z), \quad z\sim p_r.  \label{particle dynamic WGF}   
\end{equation}
Here we denote $ p_t = T(t, \cdot)_\# p_r$. $ p_t$ can be equivalently treated as the probability density of the random particle $\mathbf{x}_t$. In this dynamic, the movement of a single agent $\mathbf{x}_t$ is determined by the instant population density $ p_t$ evaluated at $\mathbf{x}_t$. Such an approach offers a microscopic and deterministic interpretation of various diffusive processes possessing WGF structures.

The aforementioned examples of WGF can be formulated as the gradient flows under Lagrangian coordinates \eqref{Lagrangian GF} as well as the particle dynamics \eqref{particle dynamic WGF}. We summarize this in the following Table \ref{tab: Lagrangian coordinates GF particle dynamic }. We assume $T(0, \cdot)_\#p_r = p_0$ as the initial condition for \eqref{Lagrangian GF}, and $\mathbf{x}_0\sim p_0$ as the initial distribution of the random particle $\mathbf{x}_t$ in \eqref{particle dynamic WGF}. We denote $ p_t=T(t, \cdot)_\#p_r$ in equation \eqref{Lagrangian GF}. Accordingly, we denote $ p_t$ as the probability density of the stochastic particle $\mathbf{x}_t$ in the dynamic \eqref{particle dynamic WGF}.

\begingroup

\setlength{\tabcolsep}{10pt} 
\renewcommand{\arraystretch}{1.5} 
\begin{table}[htb!]
    \centering
    \begin{tabular}{c|c}
    \hline
       WGF  & \begin{tabular}{@{}c@{}}
       Gradient flow in Lagrangian coordinates \\  
       Particle dynamic  
       \end{tabular}  \\
         \hline\hline
       Fokker-Planck \eqref{Fokker-Planck equ }  & 
       \begin{tabular}{@{}c@{}}
       $\partial_t T(t, z) = - \nabla_x (V + \gamma \log  p_t)\circ T(t, z)$  \\  
       $\frac{d\mathbf{x}_t}{dt} = - \nabla_x V(\mathbf{x}_t) -\gamma \nabla_x\log  p_t(\mathbf{x}_t)$ 
       \end{tabular} \\
       \hline
       Porous-medium \eqref{Porous medium equ }  & 
       \begin{tabular}{@{}c@{}}
       $\partial_t T(t, z) = -\frac{m}{m-1}  p_t(T(t, z))^{m-1}\nabla_x p_t\circ T(t, z)$  \\  
       $\frac{d\mathbf{x}_t}{dt}  =  -  \frac{m}{m-1}  p_t(\mathbf{x}_t)^{m-1}\nabla p_t(\mathbf{x}_t)$ 
       \end{tabular} \\
       \hline
       Keller-Segel \eqref{Keller-Segel equ }  & 
       \begin{tabular}{@{}c@{}}
       $\partial_t T(t, z) = -  \nabla_x( U'( p_t) + W* p_t)\circ T(t, z)$   \\  
       $\frac{d\mathbf{x}_t}{dt} = -\nabla_x U'( p_t(\mathbf{x}_t))- \nabla_x W* p_t(\mathbf{x}_t)$
       \end{tabular} \\
       \hline
    \end{tabular}
    \caption{Gradient flows under Lagrangian coordinates \& Particle dynamics associated with the WGFs.}
    \label{tab: Lagrangian coordinates GF particle dynamic }
\end{table}
\endgroup

\section{Neural projected Wassersetin gradient flows and their algorithms}\label{sec3}
As discussed in Section \ref{sec2}, instead of the direct evaluation of the density function of the Wasserstein gradient flow, it suffices to compute the time-dependent Lagrangian mapping $T(t, \cdot)$. In this research, we approximate $T(t, \cdot)$ via neural networks parametrized by time-dependent parameter $\{\theta_t\}$. The evolution of $\theta_t$ is obtained by projecting the gradient flow \eqref{Lagrangian GF} onto the parameter space $\Theta$. In this section, we briefly review the basic definitions of neural network mapping functions. We next study a metric space for neural mapping functions and formulate several neural mapping dynamics for {$\{\theta_t\}$}. 
\subsection{Neural network activation functions}
We first provide the definition of a neural network mapping function. Consider a mapping function 
\begin{equation*}
f\colon Z\times \Theta\rightarrow \Omega,
\end{equation*}
where $Z\subset\mathbb{R}^{l}$ is the latent space, $\Omega\subset\mathbb{R}^d$ is the sample space and $\Theta\subset\mathbb{R}^D$ is the parameter space. In this paper, we consider the following network structure
\begin{equation*}
f(\theta, z)=\frac{1}{N}\sum_{i=1}^Na_i\sigma\Big(z-b_i\Big),
\end{equation*}
where $\theta=(a_i, b_i)\in\mathbb{R}^{D}$, $D=(l+1)N$. Here $N$ is the number of hidden units (neurons). $a_i\in \mathbb{R}$ is the weight of unit $i$. $b_i\in\mathbb{R}^l$ is an offset (location variable). $\sigma\colon\mathbb{R}\rightarrow\mathbb{R}$ is an activation function, which satisfies 
$\sigma(0)=0$, $1\in \partial \sigma(0)$. From now on, we assume that $f$ is invertible, monotone, and is continuous w.r.t. both $z$ and $\theta$ variables. 

\noindent For example, let $N=d=1$, and $b_1=0$. Define a two layer neural network by 
\begin{center}
\begin{tikzpicture}[->,shorten >=1pt,auto,node distance=2cm,
        thick,main node/.style={circle,fill=blue!20,draw,minimum size=0.5cm,inner sep=0pt]} ]
   \node[main node] (1) {$z$};
    \node[main node] (2) [right of=1]  {$\sigma$};
    \node[main node](3)[right of=2]{$x$};
    \path[-]
    (1) edge node {} (2)
    (2) edge node{} (3);
\end{tikzpicture}
\end{center}
The following neural network mapping functions have been widely used. 
\begin{example}[Linear]
Denote $\sigma(x)=x$. Consider
\begin{equation*}
f(\theta, z)=\theta z, \quad \theta\in\mathbb{R}_+.
\end{equation*}
\end{example}
\begin{example}[ReLU]
Denote $\sigma(x)=\max\{x, 0\}$. Consider
\begin{equation*}
f(\theta, z)=\theta\max\{ z,0\}, \quad \theta\in\mathbb{R}_+.
\end{equation*}
\end{example}
\begin{example}[Sigmoid]
Denote $\sigma(x)=\frac{1}{1+e^{-2x}}$. Consider
\begin{equation*}
f(\theta, z)=\frac{\theta}{1+e^{-2 z}},\quad \theta\in\mathbb{R}_+. 
\end{equation*}
\end{example} 
In Section \ref{sec4} (theoretical results) and Section \ref{sec5} (numerical examples), we focus mainly on the case where $l=d=1$, $D=2N$. And $\sigma(\cdot)$ is the ReLU activation function.  
\subsection{Neural mapping models and energies}
In this subsection, we consider the following probability density functions generated by neural network mapping functions. We call them the neural mapping models. 
\begin{definition}[Neural mapping models]
Let us define a fixed input reference probability density $p_{\mathrm{r}}\in\mathcal{P}(Z)=\Big\{p(z)\in C^{\infty}(Z)\colon \int_Zp_{\mathrm{r}}(z)dz=1,~p(z)\geq 0\Big\}$. 
Denote a probability density generated by a neural network mapping function by the pushforward operator:
\begin{equation*}
p={f_\theta}_\#p_{\mathrm{r}}\in\mathcal{P}(\Omega),
\end{equation*}
 In other words, $p$ satisfies the following Monge-Amp\`{e}re equation by 
\begin{equation}\label{equ:monge-ampere}
p(f(\theta,z))\mathrm{det}(D_zf(\theta,z))=p_{\mathrm{r}}(z)\,,
\end{equation}
where $D_zf(\theta,z)$ is the Jacobian of the mapping function $f(\theta, z)$ w.r.t. variable $z$. 
\end{definition}
\begin{definition}[Neural mapping energies]
Given an energy functional $\mathcal{F}\colon \mathcal{P}(\Omega)\rightarrow\mathbb{R}$,  
we can construct a neural mapping energy $F\colon \Theta\rightarrow \mathbb{R}$ by
\begin{equation*}
F(\theta)=\mathcal{F}({f_{\theta}}_{\#}p_{\mathrm{r}}).
\end{equation*}
\end{definition}
Many applications in machine learning and scientific computing can be cast into the following optimization problem
\begin{equation*}
\min_{\theta\in \Theta}F(\theta).
\end{equation*}
Here, $F$ often measures the closeness between the neural mapping model and the target or data density distribution. 
Several concrete examples of neural mapping energies $F$ are given below. 
For simplicity of presentation, we often write the integration operator w.r.t. density $p_{\mathrm{r}}$ over domain $Z$ by the expectation operator $\mathbb{E}_{z\sim p_{\mathrm{r}}}$. Later in Section \ref{sec:neural-projected-Wflow}, we provide several examples of the energy functional $\mathcal F$ including the potential, the interaction (E.g. maximum mean discrepancy ) 
and the internal (information entropy/divergence) functionals. They are commonly used in machine learning and optimal transport communities; see details in \cite[Section 9]{ambrosio2005gradient}. 

To summarize, the neural mapping energies are functionals $\mathcal{F}$ written in terms of the mapping functions $f(\theta, z)$. This allows us to perform optimization on the finite dimensional space $\Theta$ instead of the infinite dimensional space $\mathcal P(\Omega)$. 
\subsection{Neural mapping metric space}
We next consider a mapping space parameterized by a neural mapping function $f(\theta, \cdot)$. 
We can measure the difference between two neural mapping functions by the $L^2$ distance thanks to the following definition. 
\begin{definition}[Neural mapping distance]
Define a distance function $\mathrm{Dist}_{\mathrm{W}}\colon \Theta\times\Theta\rightarrow\mathbb{R}$ as  

\begin{equation*}
\begin{split}
\mathrm{Dist}_{\mathrm{W}}({f_{\theta^0}}_\#p_{\mathrm{r}}, {f_{\theta^1}}_\#p_{\mathrm{r}})^2=&\int_Z\|f(\theta^0, z)-f(\theta^1, z)\|^2p_{\mathrm{r}}(z)dz\\
=&\sum_{m=1}^d\mathbb{E}_{z\sim p_{\mathrm{r}}}\Big[\|f_m(\theta^0, z)-f_m(\theta^1, z)\|^2\Big],
\end{split}
\end{equation*}
where $\theta^0$, $\theta^1\in\Theta$ are two sets of neural network parameters and $\|\cdot\|$ is the Euclidean norm in $\mathbb{R}^d$. 
\end{definition}
In the above definition, $\mathrm{Dist}_{\mathrm{W}}$ represents a distance function for two given neural mapping functions $f(\theta^0,\cdot)$ and $f(\theta^1,\cdot)$. 
In fact, the $L^2$ distance between neural mapping functions induces a metric on neural network parameters. Similar Riemannian geometry for feed-forward neural networks is also studied in \cite{ollivier2015riemannian}.

We next consider the Taylor expansion of the distance function. Let $\Delta\theta\in\mathbb{R}^D$, 
\begin{equation*}
\begin{split}
&\mathrm{Dist}_{\mathrm{W}}({f_{\theta+\Delta\theta}}_\#p_{\mathrm{r}}, {f_{\theta}}_\#p_{\mathrm{r}})^2\\=&\sum_{m=1}^d\mathbb{E}_{z\sim p_{\mathrm{r}}}\Big[\|f_m(\theta+\Delta\theta, z)-f_m(\theta, z)\|^2\Big]\\
=&\sum_{m=1}^d \sum_{i=1}^D\sum_{j=1}^D\mathbb{E}_{z\sim p_{\mathrm{r}}}\Big[\partial_{\theta_i}f_m(\theta, z)\partial_{\theta_j}f_m(\theta, z)\Big]\Delta\theta_i\Delta\theta_j+o(\|\Delta\theta\|^2)\\
=&\Delta\theta^{\ts}G_{\mathrm{W}}(\theta)\Delta\theta+o(\|\Delta\theta\|^2). 
\end{split}
\end{equation*}
Here $G_{\NT}$ is a Gram-type matrix function. We summarize its definition below.  
\begin{definition}[Neural mapping metric]\label{def:neural-metric}
Define a matrix function  $G_{\NT}\colon \Theta\rightarrow\mathbb{R}^{D\times D}$. Denote
$G_\NT(\theta)=(G_{\NT}(\theta)_{ij})_{1\leq i,j\leq D}$, such that 
\begin{equation*}
G_{\NT}(\theta)_{ij}=\sum_{m=1}^d\mathbb{E}_{z\sim p_{\mathrm{r}}}\Big[\partial_{\theta_i}f_m(\theta, z)\partial_{\theta_j}f_m(\theta, z)\Big].
\end{equation*}
We also write 
\begin{equation*}
G_\NT(\theta)=\mathbb{E}_{z\sim p_{\mathrm{r}}}\Big[\nabla_\theta f(\theta, z) \nabla_\theta f(\theta, z)^{\ts}\Big],
\end{equation*}
where we denote $\nabla_\theta f(\theta, z)=(\partial_{\theta_i}f_m(\theta, z))_{1\leq i\leq D, 1\leq m\leq d }\in\mathbb{R}^{D\times d}$.  
\end{definition}
From now on, we call $(\Theta, G_{\NT})$ the {\em neural mapping metric space}. Here we always assume that $G_{\NT}(\theta)$ is a positive definite matrix in $\mathbb{R}^{D\times D}$.

\subsection{Neural mapping dynamics}
In this subsection, we derive some analogies of Wasserstein gradient flows in the neural mapping metric space $(\Theta, G_{\mathrm{W}})$. 
Shortly, we apply them to define the neural mapping dynamics and compare them with their counterparts in $L^2$ mapping metric space and $L^2$ Wasserstein metric probability space. 
From now on, we assume that $f$ is smooth w.r.t. parameter $\theta$. This is not true for the ReLU activation function, which will be studied in detail in later sections.  

The next proposition provides gradient operators of a function $F\in C^{2}(\Theta; \mathbb{R})$ in the neural mapping metric space $(\Theta, G_{\mathrm{W}})$. 

\begin{proposition}[Neural mapping gradient operators]\label{prop35}
The gradient operator of $F$ in $(\Theta, G_{\mathrm{W}})$, $\mathrm{grad}_{\mathrm{W}}F(\theta)=(\mathrm{grad}_{\mathrm{W}} F(\theta)_k)_{k=1}^D$, is given by 
\begin{equation*}
\mathrm{grad}_{\mathrm{W}} F(\theta)_k=\sum_{i=1}^DG^{-1}_{\mathrm{W}}(\theta)_{ki}\partial_{\theta_i} F(\theta).
\end{equation*}
\end{proposition}
\begin{proof}
We briefly derive the gradient operator of $F$ in $(\Theta, G_\NT)$ below. Suppose $\theta(t)=\theta_t$ is a smooth curve passing through the point $\theta(0)=\theta$.  
Consider a Taylor expansion of $F(\theta_t)$ at $t=0$ by 
\begin{equation}\label{lq}
\begin{split}
F(\theta_t)=&F(\theta)+t\cdot \frac{d}{dt}F(\theta_t)|_{t=0}+o(t)\\
=&F(\theta)+t\cdot (G_{\mathrm{W}}(\theta)\cdot \mathrm{grad}_{\mathrm{W}} F(\theta), \dot\theta)+o(t), 
\end{split}
\end{equation}
where we denote $\frac{d}{dt}\theta_t|_{t=0}=\dot\theta$. Comparing linear terms of $t$ in \eqref{lq}, we have 
\begin{equation*}
\begin{split}
(G_{\mathrm{W}}(\theta)\cdot\mathrm{grad}_{\mathrm{W}} F(\theta), \dot\theta)=&\frac{d}{dt}F(\theta_t)|_{t=0}\\=&(\nabla_\theta F(\theta), \dot\theta),
\end{split}
\end{equation*}
for any $\dot \theta\in T_\theta\Theta=\mathbb{R}^d$.
Thus 
\begin{equation*}
\mathrm{grad}_{\mathrm{W}} F(\theta)=G^{-1}_{\mathrm{W}}(\theta)\nabla_\theta F(\theta).
\end{equation*}
\end{proof}
We are ready to present the neural mapping gradient flow, which will be used for our first-order algorithm in neural mapping optimization problems.  
\begin{proposition}[Neural mapping gradient flows]\label{prop:NGF}
Consider an energy functional $\mathcal{F}\colon \mathcal{P}(\Omega)\rightarrow\mathbb{R}$. Then the gradient flow of function $F(\theta)=\mathcal{F}({f_{\theta}}_{\#}p_{\mathrm{r}})$ in $(\Theta, G_{\mathrm{W}})$ is given by 
\begin{equation}\label{NGD}
\frac{d\theta}{dt}=-\mathrm{grad}_{\mathrm{W}}F(\theta).
\end{equation}
In particular, 
\begin{equation*}
\begin{split}
\frac{d\theta_i}{dt}=&-\sum_{j=1}^D\sum_{m=1}^d\Big(\mathbb{E}_{z\sim p_{\mathrm{r}}}\Big[\nabla_\theta f(\theta, z)\nabla_\theta f(\theta, z)^{\ts}\Big]\Big)_{ij}^{-1}\cdot\\
&\hspace{2cm}\mathbb{E}_{\tilde z\sim p_{\mathrm{r}}}\Big[\nabla_{x_m}\frac{\delta}{\delta p}\mathcal{F}(p)(f(\theta,\tilde z))\cdot \partial_{\theta_j}f_m(\theta,\tilde z)\Big],
\end{split}
\end{equation*}
where $\frac{\delta}{\delta p(x)}$ is the $L^2$--first variation w.r.t. variable $p(x)$, $x=f(\theta,z)$.
\end{proposition}
\begin{proof}
As the neural mapping metric is given in \cref{def:neural-metric}, it suffices to calculate the formula for the Euclidean gradient $\partial_{\theta_j}F(\theta)$ as follows:

\begin{equation*}
    \begin{aligned}
        \partial_{\theta_j}F(\theta) = & \ \int_{\Omega}\partial_{\theta_j}\rho_\theta(x)  \frac{\delta}{\delta p}\mathcal{F}(\rho_\theta)(x)dx      \\
        = & \ \int_{\Omega} - \nabla_x \cdot \left[\rho_\theta(x) \partial_{\theta_j}f(\theta, f(\theta, \cdot)^{-1}(x)))\right]\frac{\delta}{\delta p}\mathcal{F}(\rho_\theta)(x) dx   \\
        = & \ \int_{\Omega} \partial_{\theta_j}f(\theta, f(\theta, \cdot)^{-1}(x)) \cdot \nabla_x \left(\frac{\delta}{\delta p}\mathcal{F}(\rho_\theta)\right)(x) \rho_\theta(x) dx      \\
        = & \ \mathbb{E}_{z\sim p_{\mathrm{r}}}\Big[ \partial_{\theta_j}f(\theta,z) \cdot \nabla_x\left(\frac{\delta}{\delta p}\mathcal{F}(p)\right)(f(\theta,z)) \Big]\,.
    \end{aligned}
\end{equation*}
Here we denote $\rho_\theta = f_{\theta\#}p_{\mathrm{r}}$. 
\end{proof}

\subsection{Neural projected Wasserstein flows}\label{sec:neural-projected-Wflow}
The dynamics in parameter space can be formulated in terms of mappings and probability densities. For simplicity of discussion, we demonstrate that the neural mapping gradient flow is a projected Wasserstein gradient flow. Here the projection is from the full mapping space into a neural parameterized mapping space. Concretely, we present the following reformulations of equation \eqref{NGD}, which are in terms of mapping functions and probability density functions. The proof is based on the gradient flow equation in \cref{prop:NGF} and the application of the chain rule.

\begin{proposition}[Neural projected Wasserstein gradient flows]
Dynamic \eqref{NGD} in term of mapping functions $f(\theta,z)=(f_m(\theta,z))_{m=1}^d$ leads to   
\begin{equation*}
\begin{split}
\frac{\partial}{\partial t}f_m(\theta(t), z)=&-\sum_{i=1}^D\sum_{j=1}^D\sum_{n=1}^d\partial_{\theta_i}f_m(\theta,z)\Big(\mathbb{E}_{\tilde z\sim p_{\mathrm{r}}}\Big[\nabla_\theta f(\theta, \tilde z)\nabla_\theta f(\theta, \tilde z)^{\ts}\Big]\Big)_{ij}^{-1}\cdot\\
&\hspace{2.5cm}\mathbb{E}_{\tilde z\sim p_{\mathrm{r}}}\Big[\nabla_{x_n}\frac{\delta}{\delta p(x)}\mathcal{F}(p)(f(\theta,\tilde z))\cdot \partial_{\theta_j}f_n(\theta,\tilde z)\Big].
\end{split}
\end{equation*}

\end{proposition}
We present several examples of neural mapping Wasserstein gradient flows from proposition \ref{prop:NGF}. 
\begin{example}[Neural projected linear transport equation]\label{ex4} 
Consider a linear energy given by 
\begin{equation*}
\mathcal{F}(p)=\int_\Omega V(x)p(x)dx. 
\end{equation*}
In this case, the neural projected gradient flow satisfies 
\begin{equation}\label{eq:gf_linear}
\frac{d\theta}{dt}=-G_{\mathrm{W}}^{-1}(\theta)\cdot\mathbb{E}_{\tilde z\sim p_{\mathrm{r}}}\Big[\nabla_\theta V(f(\theta,\tilde z))\Big].
\end{equation}
In details, 
\begin{equation*}
\frac{d\theta_i}{dt}=-\sum_{j=1}^D\Big(\mathbb{E}_{z\sim p_{\mathrm{r}}}\Big[\nabla_\theta f(\theta, z)\nabla_\theta f(\theta, z)^{\ts}\Big]\Big)_{ij}^{-1}\cdot\mathbb{E}_{\tilde z\sim p_{\mathrm{r}}}\Big[\nabla_xV(f(\theta,\tilde z))\cdot \partial_{\theta_j}f(\theta, \tilde z)\Big]. 
\end{equation*}
\end{example}
\begin{example}[Neural projected interaction transport equation]\label{ex5}
Consider an interaction energy given by
\begin{equation*}
\mathcal{F}(p)=\frac{1}{2}\int_\Omega\int_\Omega W(x_1,x_2)p(x_1)p(x_2)dx_1dx_2.
\end{equation*}
In this case, the neural mapping gradient flow satisfies 
\begin{equation}\label{eq:gf_interaction}
\frac{d\theta}{dt}=-\frac{1}{2}G_{\mathrm{W}}^{-1}(\theta)\cdot\mathbb{E}_{(z_1,z_2)\sim p_{\mathrm{r}}\times p_{\mathrm{r}}}\Big[\nabla_\theta W(f(\theta,z_1), f(\theta, z_2))\Big].
\end{equation}
In details, 
\begin{equation*}
\begin{split}
\frac{d\theta_i}{dt}=&-\sum_{j=1}^D\Big(\mathbb{E}_{z\sim p_{\mathrm{r}}}\Big[\nabla_\theta f(\theta, z)\nabla_\theta f(\theta, z)^{\ts}\Big]\Big)_{ij}^{-1}\cdot\\
&\hspace{1.2cm}\mathbb{E}_{(z_1,z_2)\sim p_{\mathrm{r}}\times p_{\mathrm{r}}}\Big[\nabla_{x_1}W(f(\theta,z_1), f(\theta,z_2))\cdot \partial_{\theta_j}f(\theta, z_1)\Big]. 
\end{split}
\end{equation*}
\end{example}
\begin{example}[Neural projected negative entropy]\label{ex6}
Consider a negative entropy functional given by
\begin{equation*}
\mathcal{F}(p)=\int_\Omega U(p(x))dx.
\end{equation*}
In this case, the neural mapping gradient flow satisfies 
\begin{equation}\label{eq:gf_entropy}
\frac{d\theta}{dt}=-G_{\mathrm{W}}^{-1}(\theta)\cdot\mathbb{E}_{z\sim p_{\mathrm{r}}}\Big[\nabla_\theta \hat U(\frac{p_{\mathrm{r}}(z)}{\mathrm{det}(D_zf(\theta,z))})\Big]\,,
\end{equation}
where $\hat U(p) = U(p)/p$. This is because: 
\begin{equation*}
\begin{split}
\mathcal{F}({f_{\theta}}_\#p_{\mathrm{r}})=&\int_\Omega U(p(f(\theta,z)))df(\theta,z)\\
=&\int_Z U(\frac{p_{\mathrm{r}}(z)}{\mathrm{det}(D_zf(\theta,z))})\frac{\mathrm{det}(D_zf(\theta,z))}{p_{\mathrm{r}}(z)}p_{\mathrm{r}}(z)dz\\
=&\mathbb{E}_{z\sim p_{\mathrm{r}}}\Big[\hat U(\frac{p_{\mathrm{r}}(z)}{\mathrm{det}(D_zf(\theta,z))})\Big]\,.
\end{split}
\end{equation*}
The choice $U(p) = p\log (p)$ and $\hat U(p)= \log(p)$ corresponds to the negative entropy. This belongs to the family of internal energy.
In details, 
\begin{equation*}
\begin{split}
\frac{d\theta_i}{dt}=&-\sum_{j=1}^D\Big(\mathbb{E}_{z\sim p_{\mathrm{r}}}\Big[\nabla_\theta f(\theta, z)\nabla_\theta f(\theta, z)^{\ts}\Big]\Big)_{ij}^{-1}\cdot\\
&\hspace{1.2cm}\mathbb{E}_{z\sim p_{\mathrm{r}}}\Big[-\mathrm{tr}\Big(D_zf(\theta,z)^{-1}\colon\partial_{\theta_j}D_zf(\theta, z)\Big)\hat U'(\frac{p_{\mathrm{r}}(z)}{\mathrm{det}(D_zf(\theta,z))})\frac{p_{\mathrm{r}}(z)}{\mathrm{det}(D_zf(\theta,z))}\Big]. 
\end{split}
\end{equation*}
Here we denote $\mathrm{tr}(A\colon B)=\mathrm{tr}(AB)$, for matrices $A$, $B\in\mathbb{R}^{d\times d}$. 
\end{example}
The above examples are projected Wasserstein gradient flows in neural mapping metric space. In particular, Examples \ref{ex4}, \ref{ex5},  \ref{ex6} correspond to the following classical PDEs, respectively.  
\begin{align}
\partial_tp(t,x)=&\nabla_x\cdot\Big(p(t,x)\nabla_x V(x)\Big),\label{eq:transport_pde}\\
\partial_tp(t,x)=&\nabla_x\cdot\Big(p(t,x)\int_\Omega \nabla_xW(x,y)p(t,y)dy\Big),\label{eq:interaction_pde}\\
\partial_tp(t,x)=&\nabla_x\cdot\Big(p(t,x)\nabla_xU'(p(t,x))\Big).
\label{eq:negative_entropy_pde}\,
\end{align}
The above dynamics include potential transport, interaction transport, and porous medium equations. The Fokker-Planck equation is a combination of the above first and third equations.
\subsection{Algorithm}
In this section, we discuss the implementations of gradient flows projected onto the parameter space. We apply the forward Euler discretization of the natural gradient flow \eqref{NGD}. 
Let $h>0$ be the step size. Then the update is given by 
\begin{equation}\label{eq:euler}
\theta^{k+1}=\theta^k-h\Big(\tilde G_{\mathrm{W}}(\theta^k)\Big)^{-1}\nabla_\theta \tilde F(\theta^k)\,,
\end{equation}
where $\tilde G_{\mathrm{W}}(\theta)=(\tilde G_{\mathrm{W}}(\theta)_{ij})_{1\leq i,j\leq D}\in\mathbb{R}^{D\times D}$, $\nabla_\theta\tilde F(\theta)$ are empirical estimates of the matrix $G_{\mathrm{W}}$ and the gradient $\nabla F(\theta)=\{\partial_{\theta_j}F(\theta)\}_{j=1}^D$, respectively. In details, if $(z_i)_{l=1}^M\sim p_{\mathrm{r}}$, where $M$ is the number of empirical samples, then
\begin{equation*}
\tilde G_{\mathrm{W}}(\theta)_{ij}=\frac{1}{M}\sum_{l=1}^{M}\sum_{m=1}^d\partial_{\theta_i}f_m(z_l,\theta)\partial_{\theta_j}f_m(z_l,\theta)\,.
\end{equation*}
In practice, the condition number of $\tilde G_{\mathrm{W}}(\theta)$ could be very large and it is more stable to use instead the pseudoinverse of $\tilde G_{\mathrm{W}}(\theta)$ in \eqref{eq:euler}. 
Therefore, the update is  
\begin{equation*}
\theta^{k+1}=\theta^k-h\tilde G_{\mathrm{W}}(\theta)^\dagger\nabla_\theta \tilde F(\theta^k)\,.
\end{equation*}
When the reference measure is a one-dimensional standard Gaussian distribution, $G_{\mathrm{W}}(\theta)$ can be explicitly computed for our choice of neural network. In this case, we have 
\begin{equation*}
\theta^{k+1}=\theta^k-h G_{\mathrm{W}}(\theta)^\dagger\nabla_\theta \tilde F(\theta^k)\,.
\end{equation*}
We summarize the above explicitly update formulas below. 
	\begin{algorithm}[H] \label{alg}
		\begin{algorithmic}
			\caption{Projected Wasserstein gradient flows}
			\STATE{\textbf{Input:} Initial parameters $\theta\in\mathbb{R}^D$; stepsize $h>0$, total number of steps $L$, samples $\{z_i\}_{i=1}^M\sim p_{\mathrm{r}}$ for estimating $\tilde G_{\mathrm{W}}(\theta)$ and $\nabla_\theta \tilde F(\theta)$.}
			\STATE
			\FOR{$k=1,2,\ldots, L$ }
\STATE{\begin{equation*}
\theta^{k+1}=\theta^k-h\tilde G_{\mathrm{W}}(\theta)^\dagger\nabla_\theta \tilde F(\theta^k); \textrm{\quad (when $G_{\mathrm{W}}(\theta) $ is unknown)} 
\end{equation*}}
\OR 
\STATE{\begin{equation*}
\theta^{k+1}=\theta^k-h G_{\mathrm{W}}(\theta)^\dagger\nabla_\theta \tilde F(\theta^k); \textrm{\quad (when $G_{\mathrm{W}}(\theta) $ is known)} 
\end{equation*}}
			\ENDFOR
		\end{algorithmic}
	\end{algorithm}

\section{Numerical analysis on neural network projected gradient flows}\label{sec4}
In this section, we establish theoretical guarantees for the performance of the neural projected dynamics. We start by deriving an analytic formula for the inverse of the neural mapping metric of a special ReLU family in \cref{sec:analytic-inv}. Based on the closed-form projected dynamics equations, we can establish the truncated error analysis for the projected dynamics in \cref{sec:special-truncated}. The analysis of truncated error for general dynamics is presented in \cref{sec:general-truncated}.

\subsection{Analytic formula for the inverse of neural mapping metric}\label{sec:analytic-inv}
In this section, we consider the following special case of the ReLU model in 1D. We first rewrite the neural network mapping function into the following form:
\begin{equation}\label{equ:ReLU-network}
f(\theta, z)=\frac{1}{N}\sum_{i=1}^Na_i\sigma(z-b_i), \quad \sigma(z) = \left\{\begin{aligned}
    & 0, \quad z < 0,       \\
    & z, \quad z \geq 0.
\end{aligned}\right.
\end{equation}
In particular, we combine $a_i, b_i$ into one parameter in the 1D case. Under this reparameterization, $a_i$s represent the slopes of each ReLU component and $b_i$s are the intercepts. To make the last assumption on this ReLU network mapping function which facilitates the analytic formula of the neural mapping metric, we require all the slope parameters to stay non-negative, i.e. $a_i \geq 0$. Although this is an artificial assumption to enforce analyticity, it is natural in the sense that positive slope parameters induce monotone mapping function. Meanwhile, solutions of the Monge problems in 1D are known to be monotone. In \cref{fig:ReLU}, we plot a typical ReLU mapping function.

\begin{figure}[ht]
	\centering
        \label{fig:ReLU}
  	\includegraphics[width=0.6\linewidth]{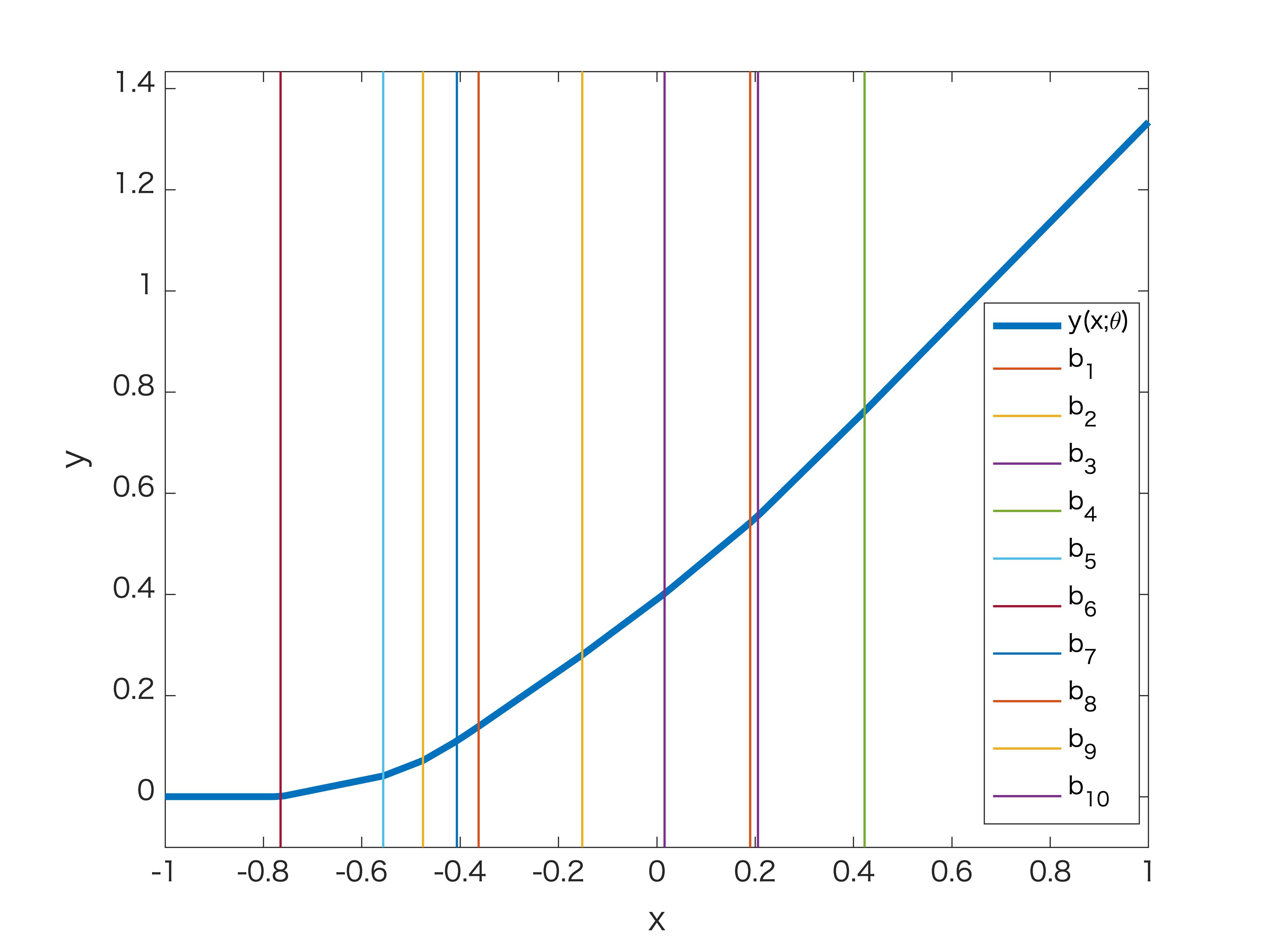}
  	\caption{ReLU network mapping function considered in this section. The figure plots a typical monotone map parameterized by the ReLU network where the parameter $a_i$ is required to be positive.}
  	\label{ReLU_2}
	\end{figure}
We start with the analytic formula for the neural mapping metric, assuming the reference measure is given by $p_{\mathrm{r}}(\cdot)$ with associated cumulative distribution function $\mathfrak{F}_0(\cdot)$. 

\begin{proposition}[Neural mapping metric of two-layer ReLU network]\label{prop:metric-2layer}
	The neural mapping metric of the two-layer ReLU network with reference measure $p_{\mathrm{r}}(\cdot)$ is given as
	\bequ\label{WIM-2lay}
		\begin{aligned}
		&G_{\mathrm{W}} = \frac{1}{N^2}\begin{pmatrix}
			G_{\mathrm{W}}^{bb} & G_{\mathrm{W}}^{bw} 		\\
			\lp G_{\mathrm{W}}^{bw} \right)^T & G_{\mathrm{W}}^{ww}
		\end{pmatrix},		\\
		&G_{\mathrm{W}}^{bb} = \begin{pmatrix}	
			a_1^2\lp 1 - \mathfrak{F}_0\lp b_1 \right)\right) & a_1a_2\lp 1 - \mathfrak{F}_0\lp b_2 \right)\right) & \cdots & a_1a_N\lp 1 - \mathfrak{F}_0\lp b_N \right)\right)		\\
			a_1a_2\lp 1 - \mathfrak{F}_0\lp b_2 \right)\right) & a_2^2\lp 1 - \mathfrak{F}_0\lp b_2 \right)\right) &  \cdots & a_2a_N\lp 1 - \mathfrak{F}_0\lp b_N \right)\right)		\\
			 \vdots & \vdots & \ddots & \vdots 		\\
			 a_1a_{N - 1}\lp 1 - \mathfrak{F}_0\lp b_{N - 1} \right)\right) & a_2a_{N - 1}\lp 1 - \mathfrak{F}_0\lp b_{N - 1} \right)\right) &  \cdots & a_Na_{N - 1}\lp 1 - \mathfrak{F}_0\lp b_N \right)\right)	\\
			 \\
			a_1a_N\lp 1 - \mathfrak{F}_0\lp b_N \right)\right) & a_2a_N\lp 1 - \mathfrak{F}_0\lp b_N \right)\right) &  \cdots & a_N^2\lp 1 - \mathfrak{F}_0\lp b_N \right)\right)
		\end{pmatrix},		\\
		\\
		&G_{\mathrm{W}}^{ba} = -\begin{pmatrix}	
			a_1\int_{b_1}^{\infty} \lp z - b_1 \right) p_{\mathrm{r}}(z) dz & a_1\int_{b_2}^{\infty} \lp z - b_2 \right) p_{\mathrm{r}}(z) dz & \cdots & a_1\int_{b_N}^{\infty} \lp z - b_N \right) p_{\mathrm{r}}(z) dz		\\
			\\
			a_2\int_{b_2}^{\infty} \lp z - b_1 \right) p_{\mathrm{r}}(z) dz & a_2\int_{b_2}^{\infty} \lp z - b_2 \right) p_{\mathrm{r}}(z) dz & \cdots & a_2\int_{b_N}^{\infty} \lp z - b_N \right) p_{\mathrm{r}}(z) dz		\\
			 \vdots & \vdots & \ddots & \vdots 		\\
			 
			a_N\int_{b_N}^{\infty} \lp z - b_1 \right) p_{\mathrm{r}}(z) dz & a_N\int_{b_N}^{\infty} \lp z - b_2 \right) p_{\mathrm{r}}(z) dz & \cdots & a_N\int_{b_N}^{\infty} \lp z - b_N \right) p_{\mathrm{r}}(z) dz
		\end{pmatrix},		\\
				&\lp G_{\mathrm{W}}^{aa} \right)_{ij} =\int_{\max\{b_i, b_j\}}^{\infty} \lp z - b_j \right) \lp z - b_i \right) p_{\mathrm{r}}(z) dz.
		\end{aligned}
	\eequ
\end{proposition}
\begin{proof}
	We first calculate the derivative of the neural network map $f(\theta, z)$ w.r.t. network parameters $\theta$
	\bequ\label{equ:ReLU-network-derivative}
		\begin{aligned}
		& \partial_{b_i} f(\theta, z) = \lbb \begin{aligned}
			& \ 0, \quad z < b_i, 		\\
			& -\frac{a_i}{N}, \quad z > b_i,
		\end{aligned}\right.	\quad \partial_{a_i} f(\theta, z) = \frac{1}{N}\sigma\lp z - b_i \right),
		\end{aligned}
	\eequ
	while the value at the singular point $b_i$ does not exist and can be omitted from the measure-theoretical perspective. According to \cref{def:neural-metric}, one can evaluate different blocks of the metric tensor as the following integral
	\bequn	
		\begin{aligned}
			\lp G_{\mathrm{W}}^{bb} \right)_{ij} & = \int_{\mbR} \partial_{b_i} f(\theta, z) \partial_{b_j} f(\theta, z) p_{\mathrm{r}}(z)dz = \frac{a_ia_j}{N^2}\lp 1 - \mathfrak{F}_0\lp \max\{b_i, b_j\} \right)\right),		\\
			\lp G_{\mathrm{W}}^{ba} \right)_{ij} & = \int_{\mbR} \partial_{b_i} f(\theta, z) \partial_{a_j} f(\theta, z) p_{\mathrm{r}}(z)dz = -\frac{a_i}{N^2}\int_{\max\{b_i, b_j\}}^{\infty} \lp z - b_j \right) p_{\mathrm{r}}(z) dz, 		\\
			\lp G_{\mathrm{W}}^{aa} \right)_{ij} & = \int_{\mbR} \partial_{a_i} f(\theta, z) \partial_{a_j} f(\theta, z) p_{\mathrm{r}}(z)dz = \frac{1}{N^2}\int_{\max\{b_i, b_j\}}^{\infty} \lp z - b_j \right) \lp z - b_i \right) p_{\mathrm{r}}(z) dz.
		\end{aligned}
	\eequn
\end{proof}

For general reference measure $p_{\mathrm{r}}(\cdot)$, the matrix elements of the $G_{\mathrm{W}}^{ba}, G_{\mathrm{W}}^{aa}$ relate to the first and second moments of the measure which may not have an analytic formula. Here, we consider a special neural mapping model with a Gaussian reference measure, thus rendering the metric with analytic elements.

\begin{corollary}\label{cor:metric-2layerGaussian}
    With the same setting as \cref{prop:metric-2layer} and Gaussian reference measure, the matrix element of the neural mapping metric can be written analytically as 
    \begin{equation}
        \begin{aligned}
            &\lp G_{\mathrm{W}}^{ba} \right)_{ij} = p_{\mathrm{r}}(b_i) -  b_j(1-\mathfrak{F}_0(b_i)), \quad b_i > b_j. \\
            &\lp G_{\mathrm{W}}^{aa} \right)_{ij} = b_ib_j(1-\mathfrak{F}_0(b_i)) - b_jp_{\mathrm{r}}(b_i)+ (1-\mathfrak{F}_0(b_i)), \quad b_i > b_j.
        \end{aligned}
    \end{equation}
    The other half of the elements can be obtained via switching $b_i, b_j$.
\end{corollary}
\begin{proof}
    The proof is obtained by elementary integration calculation
    \begin{equation}
        \begin{aligned}
            & \ \int_{b_i}^{\infty} \lp z - b_j \right)  p_{\mathrm{r}}(z) dz      \\
            = & \ p_{\mathrm{r}}(z)\Big|_{\infty}^{b_i} - b_j(1-\mathfrak{F}_0(b_i)) = p_{\mathrm{r}}(b_i) -  b_j(1-\mathfrak{F}_0(b_i)),     \\
            & \ \int_{b_i}^{\infty} \lp z - b_j \right) \lp z - b_i \right) p_{\mathrm{r}}(z) dz      \\
            = & \ b_ib_j(1-\mathfrak{F}_0(b_i)) - (b_i+b_j)p_{\mathrm{r}}(b_i) + b_ip_{\mathrm{r}}(b_i) + (1-\mathfrak{F}_0(b_i)) \\
            = & \ b_ib_j(1-\mathfrak{F}_0(b_i)) - b_jp_{\mathrm{r}}(b_i)+ (1-\mathfrak{F}_0(b_i)).
        \end{aligned}
    \end{equation}
\end{proof}
Now, we focus on the upper right corner $G_{\mathrm{W}}^{bb}$ of the neural mapping metric. We will establish an analytical formula for the inverse of this matrix.

\begin{theorem}[Analytic inverse of the neural mapping metric]\label{thm:ReLU-inverse-metric}
    The inverse matrix of the $G_{\mathrm{W}}^{bb}$ block in \cref{prop:metric-2layer} can be written analytically as 
    \begin{equation}
        \begin{aligned}
            \frac{1}{N^2}\left( G_{\mathrm{W}}^{-1}(\mathbf{b}) \right)_{ij} = \left\{
            \begin{aligned}
                \frac{1}{a_i^2}\lp \frac{1}{\mathfrak{F}_0(b_{i})-\mathfrak{F}_0(b_{i-1})} + \frac{1}{\mathfrak{F}_0(b_{i+1})-\mathfrak{F}_0(b_i)}\right), \quad i = j \neq 1, N, \\
                \frac{1}{a_i^2}\lp \frac{1}{\mathfrak{F}_0(b_{N})-\mathfrak{F}_0(b_{N-1})} + \frac{1}{1-\mathfrak{F}_0(b_N)}\right), \quad i = j = N, \\
                \frac{1}{a_i^2}\frac{1}{\mathfrak{F}_0(b_{2})-\mathfrak{F}_0(b_{1})}, \quad i = j = 1,\\
                -\frac{1}{a_ia_{i-1}}\frac{1}{\mathfrak{F}_0(b_i) - \mathfrak{F}_0(b_{i-1})}, \quad j = i-1,      \\
                -\frac{1}{a_ia_{i+1}}\frac{1}{\mathfrak{F}_0(b_{i+1}) - \mathfrak{F}_0(b_{i})}, \quad j = i+1,      \\
                0, \qquad \qquad o.w.
            \end{aligned}\right.
        \end{aligned}
    \end{equation}
\end{theorem}
\begin{proof}
    First, we decompose the neural mapping metric into the following matrix product
    \begin{equation}
        G_{\mathrm{W}} = \frac{1}{N^2}D\begin{pmatrix}
            1-\mathfrak{F}_0(b_1) & 1-\mathfrak{F}_0(b_2) & \cdots & 1-\mathfrak{F}_0(b_N)    \\
            1-\mathfrak{F}_0(b_2) & 1-\mathfrak{F}_0(b_2) & \cdots & 1-\mathfrak{F}_0(b_N)    \\
            \vdots & \vdots & \ddots & \vdots \\
            1-\mathfrak{F}_0(b_N) & 1-\mathfrak{F}_0(b_N) & \cdots & 1-\mathfrak{F}_0(b_N)    \\
        \end{pmatrix} D,
    \end{equation}
    where $D = \diag(a_1, a_2, \cdots, a_N)$ is a diagonal matrix. Then, it is direct to check that the middle matrix has the following tri-diagonal analytic inverse below: 
    \begin{equation}
        \begin{pmatrix}
            \frac{1}{\mathfrak{F}_0(b_2)-\mathfrak{F}_0(b_1)} & -\frac{1}{\mathfrak{F}_0(b_2)-\mathfrak{F}_0(b_1)} & 0 & \cdots & 0     \\
            -\frac{1}{\mathfrak{F}_0(b_2)-\mathfrak{F}_0(b_1)} & \frac{1}{\mathfrak{F}_0(b_2)-\mathfrak{F}_0(b_1)} + \frac{1}{\mathfrak{F}_0(b_3)-\mathfrak{F}_0(b_2)} & -\frac{1}{\mathfrak{F}_0(b_3)-F(b_2)} & \cdots & 0     \\
            \vdots & \vdots & \vdots & \ddots & \vdots \\
            0 & 0 & 0 & \cdots & \frac{1}{F(b_N) - F(b_{N-1})} + \frac{1}{1-F(b_N)}
        \end{pmatrix}.
    \end{equation}
    Multiplying this matrix with the inverse of the diagonal matrix $D$ on both sides concludes this proof.
\end{proof}
This analytic form of the inverse metric will be used intensively in the next subsection to prove the consistency of the numerical scheme based on the ReLU neural network.

\subsection{Truncated error analysis of the neural projected Wasserstein gradient flows based on analytic formula}\label{sec:special-truncated}
In this section, we perform the numerical analysis of the neural mapping projected Wasserstein flows introduced in \cref{sec:neural-projected-Wflow} based on the analytic formula in \cref{sec:analytic-inv}. Because of the analytic inverse of the neural mapping metric, the right-hand side of the Wasserstein projected gradient flow can be calculated explicitly, and one can thus talk about its consistency and order of accuracy following the same spirit as classical numerical analysis. We perform this derivation for the Wasserstein projected gradient flows of the potential functional explicitly.

Let us first recall that the formula for neural projected Wasserstein gradient flow is given by 
\begin{equation}\label{equ:projected-gradient-flow}
\frac{d\theta}{dt}=-G_{\mathrm{W}}^{-1}(\theta)\cdot\mathbb{E}_{\tilde z\sim p_{\mathrm{r}}}\Big[\nabla_\theta V(f(\theta,\tilde z))\Big].
\end{equation}
We have the following analytic formula for the projected gradient flow in the ReLU network model that we introduced in \cref{sec:analytic-inv}.
\begin{proposition}[Wasserstein gradient flow of potential functionals in ReLU network]\label{prop:potential-flow-formula}
    The projected potential flow in the ReLU network model \cref{equ:ReLU-network} has the following form:
    \begin{equation}
        \begin{aligned}
            \dot{b}_i & = \frac{N}{a_i}\lb \frac{\mbE_{z\sim p_{\mathrm{r}}}[V'(f(b, z))\mathbf{1}_{[b_{i}, b_{i+1}]}]}{\mathfrak{F}_0(b_{i+1}) - \mathfrak{F}_0(b_{i})} -  \frac{\mbE_{z\sim p_{\mathrm{r}}}[V'(f(b, z))\mathbf{1}_{[b_{i-1}, b_{i}]}]}{\mathfrak{F}_0(b_{i}) - \mathfrak{F}_0(b_{i-1})}\rb,       \quad i \neq 1, N,      \\
            \dot{b}_N & = \frac{N}{a_N}\lb \frac{\mbE_{z\sim p_{\mathrm{r}}}[V'(f(b, z))\mathbf{1}_{[b_{N}, \infty)}]}{1 - \mathfrak{F}_0(b_{N})} -  \frac{\mbE_{z\sim p_{\mathrm{r}}}[V'(f(b, z))\mathbf{1}_{[b_{N-1}, b_{N}]}]}{\mathfrak{F}_0(b_{N}) - \mathfrak{F}_0(b_{N-1})}\rb ,        \\
            \dot{b}_1 & = \frac{N}{a_1}\frac{\mbE_{z\sim p_{\mathrm{r}}}[V'(f(b, z))\mathbf{1}_{[b_{1}, b_{2}]}]}{\mathfrak{F}_0(b_{2}) - \mathfrak{F}_0(b_{1})}.   \\
        \end{aligned}
    \end{equation}
    Using the trapezoid rule to calculate the integration gives the following spatial discretization, which can be used to simulate the projected gradient flow:
    \begin{equation}\label{equ:potential-flow-discretization}
        \dot{b}_i = \frac{N}{2a_i}\lp V'(f(b, b_{i+1})) - V'(f(b, b_{i-1})) \right).
    \end{equation}
\end{proposition}
\begin{proof}
    It suffices to calculate the gradient of the linear potential functional in this model. Let us start with the calculation of the functional form of the potential energy in the ReLU network mapping model as follows
    \begin{equation}
        \begin{aligned}
            \mbE_{x \sim f_{b \#}p_{\mathrm{r}}}[V(x)] = & \ \mbE_{z\sim p_{\mathrm{r}}}[V(f(b, z))],
        \end{aligned}
    \end{equation}
    where we use the change of the integration variable above. Therefore, the gradient of this functional w.r.t. $b$ can be simplified to
    \begin{equation}
        \begin{aligned}
            \partial_{b_i}\mbE_{z\sim p_{\mathrm{r}}}[V(f(b, z))] = & \ \mbE_{z\sim p_{\mathrm{r}}}[\partial_{b_i}V(f(b, z))] = -\frac{a_i}{N}\mbE_{z\sim p_{\mathrm{r}}}[V'(f(b, z))\mathbf{1}_{[b_i, \infty)}(z)],
        \end{aligned}
    \end{equation}
    where we use $\mathbf{1}_A$ to denote the characteristic function on the interval $A$. Now, plugging this result into the projected gradient flow \cref{equ:projected-gradient-flow} with the analytical formula for the inverse matrix $G_{\mathrm{W}}^{-1}$ in \cref{thm:ReLU-inverse-metric}, we obtain
    \begin{equation}
        \begin{aligned}
            \dot{b_i} = & \ \frac{N^2}{a_i^2}\lp \frac{1}{\mathfrak{F}_0(b_{i})-\mathfrak{F}_0(b_{i-1})} + \frac{1}{\mathfrak{F}_0(b_{i+1})-\mathfrak{F}_0(b_i)}\right)\frac{a_i}{N}\mbE_{z\sim p_{\mathrm{r}}}[V'(f(b, z))\mathbf{1}_{[b_i, \infty)}(z)]     \\
            & \ - \frac{N^2}{a_ia_{i-1}}\frac{1}{\mathfrak{F}_0(b_i) - \mathfrak{F}_0(b_{i-1})}\frac{a_{i-1}}{N}\mbE_{z\sim p_{\mathrm{r}}}[V'(f(b, z))\mathbf{1}_{[b_{i-1}, \infty)}(z)] \\
            & \ - \frac{N^2}{a_ia_{i+1}}\frac{1}{\mathfrak{F}_0(b_{i+1}) - \mathfrak{F}_0(b_{i})}\frac{a_{i+1}}{N}\mbE_{z\sim p_{\mathrm{r}}}[V'(f(b, z))\mathbf{1}_{[b_{i+1}, \infty)}(z)]       \\
            = & \ \frac{N}{a_i}\lb \frac{\mbE_{z\sim p_{\mathrm{r}}}[V'(f(b, z))\mathbf{1}_{[b_{i}, b_{i+1}]}(z)]}{\mathfrak{F}_0(b_{i+1}) - \mathfrak{F}_0(b_{i})} -  \frac{\mbE_{z\sim p_{\mathrm{r}}}[V'(f(b, z))\mathbf{1}_{[b_{i}, b_{i+1}]}(z)]}{\mathfrak{F}_0(b_{i}) - \mathfrak{F}_0(b_{i-1})} \rb.
        \end{aligned}
    \end{equation}
    Taking a close look at the terms inside the brackets, one finds that they are calculating the average value of $V'$ inside the intervals $[b_{i-1}, b_i], [b_i, b_{i+1}]$ weighted by the base distribution $p_{\mathrm{r}}(\cdot)$. Lastly, in order to complete the spatial discretization, one needs to choose a quadrature rule to calculate the integration in the above formula. One example is the trapezoid rule:
    \begin{equation*}
        \mbE_{z\sim p_{\mathrm{r}}}[V'(f(b, z))\mathbf{1}_{[b_{i}, b_{i+1}]}(z)] \approx (\mathfrak{F}_0(b_{i+1}) - \mathfrak{F}_0(b_{i}))\frac{V'(f(b, b_i)) + V'(f(b, b_{i+1}))}{2},
    \end{equation*}
    which provides the desired discretization. Special attention should be paid to the boundary node $b_1, b_N$ to obtain their corresponding evolution equation and discretization.
\end{proof}
Given this spatial discretization, we can analyze the order of consistency of it, which is treated in the following proposition.
\begin{proposition}[Consistency of the projected gradient flow]\label{prop:potential-flow-consistency}
	Assume potential functional satisfies $\norml V'' \normr_{\infty} < \infty$. The spatial discretization \cref{equ:potential-flow-discretization} is of first-order accuracy both in the mapping and the density coordinates.
\end{proposition}
\begin{proof}
    We prove this statement from two directions, i.e. consistency in the space of mapping distribution and consistency in the space of mapping function. We have
    \begin{equation}
        \begin{aligned}
            \partial_t f(b(t), z) = & \ \dot{b}^T \partial_b f(b, z)        \\
            = & \ -\sum_{i=1}^N\frac{N}{2a_i}\lp V'(f(b, b_{i+1})) - V'(f(b, b_{i-1})) \right) \frac{a_i}{N}\mathbf{1}_{[b_i, \infty)}(z)      \\
            = & \ -\sum_{i=1}^N\frac{V'(f(b, b_{i+1})) - V'(f(b, b_{i-1}))}{2}\mathbf{1}_{[b_i, \infty)}(z) \\
            = & \ -\frac{V'(f(b, b_{i+1}))+V'(f(b, b_{i}))}{2}, \quad z \in [b_i, b_{i+1}].
        \end{aligned}
    \end{equation}
    In the above derivation, we slightly cheat in the derivation so we can use the consistent formula for the evolution equations for all the nodes $b_i$. It is easy to conclude that our discretization corresponds to the evolution of the mapping function $f$ of constant speed $-\frac{V'(f(b, b_{i+1}))+V'(f(b, b_{i}))}{2}$ on each interval $[b_i, b_{i+1}]$. Now, recall that in mapping space, the Wasserstein gradient flow of the potential function $V(x)$ corresponds to the velocity field $-V'(x)$. Therefore, given that the length of each interval is of order $\Delta b$, we conclude that our spatial discretization is first order consistent on the mapping space.
    
    Next, we prove the statement for the mapping distribution. To do this, we need to derive the evolution equation for the mapping distribution according to \cref{equ:potential-flow-discretization}. We have for $x \in [f(b, b_i), f(b, b_{i+1})]$
    \begin{equation}\label{equ:potential-parametric}
        \begin{aligned}
            & \ \partial_t p(t, x) = \dot{b}^T \partial_b p(t, x)        \\
            = & \ \sum_{i=1}^N\frac{N}{2a_i}\lp V'(f(b, b_{i+1})) - V'(f(b, b_{i-1})) \right) \frac{a_i}{N}\lp \partial_x p(t, x)\mathbf{1}_{[f(b, b_i), \infty)}(x) + p(t, x)\delta_{f(b, b_i)}(x) \right)   \\
            = & \ \partial_x p(t, x)\sum_{i=1}^N\frac{V'(f(b, b_{i+1})) - V'(f(b, b_{i-1}))}{2}\mathbf{1}_{[f(b, b_i), \infty)}(x)       \\
            & \ + p(t, x)\sum_{i=1}^N\frac{V'(f(b, b_{i+1})) - V'(f(b, b_{i-1}))}{2}\delta_{f(b, b_i)}(x) \\
            = & \ \partial_x p(t, x)\frac{V'(f(b, b_{i+1})) + V'(f(b, b_{i}))}{2} - p(t, x)\frac{V'(f(b, b_{i+1})) + V'(f(b, b_{i-1}))}{2}\delta_{f(b, b_i)}(x).
        \end{aligned}
    \end{equation}
    A quick method to derive the formula of $\partial_b p(t, x)$ is to view it as a probability flow corresponds to the cotangent vector $\partial_{b}f$ and then use the Wasserstein metric to calculate via a continuity equation as in \cref{prop:NGF}. Recall that the potential gradient flow in the density manifold is given by
    \begin{equation}\label{equ:potential-global}
        \partial_t p(t, x) = \nabla \cdot (p(t, x)\nabla V(x)) = \partial_xp(t, x)\partial_x V(x) + \partial_{xx} p(t, x) V(x).
    \end{equation}
    Comparing \cref{equ:potential-parametric} and \cref{equ:potential-global}, it is not difficult to recognize that the first term in \cref{equ:potential-parametric} approximates the continuous counterpart in \cref{equ:potential-global} in the first order. The remaining parts correspond to each other: the approximation is first order not in the strong sense, but rather in the weak sense as there is Dirac measure in \cref{equ:potential-parametric}. Combining the above two parts, we finish the proof. 
\end{proof}
\subsubsection{Projected dynamics of Negative entropy gradient flow}\label{sec:gradient_neg_entropy}
The potential functional can be viewed as a linear functional whose projected gradient flow has a rather simple expression. The corresponding formula has a more complex expression for general nonlinear internal energy, such as entropy. We begin with calculating the negative entropy functional of a neural mapping measure $f_{\#}p_{\mathrm{r}}$:
\begin{equation}\label{equ:entropy}
		\begin{aligned}
			H\lp f_{\#}p_{\mathrm{r}} \right) = & \ \mbE_{x \sim cont\lp f_{\#}p_{\mathrm{r}} \right)}\lb \log  f_{\#}p_{\mathrm{r}}\lp x \right) \rb + \mathfrak{F}_0\lp b_1 \right) \log \mathfrak{F}_0\lp b_1 \right)		\\
			= & \ \mbE_{z \sim cont\lp p_{\mathrm{r}} \right)}\lb \log  f_{\#}p_{\mathrm{r}}\lp f\lp z \right) \right) \rb	+ \mathfrak{F}_0\lp b_1 \right) \log \mathfrak{F}_0\lp b_1 \right)	\\
			= & \ \mbE_{z \sim cont\lp p_{\mathrm{r}} \right)}\lb \log  \frac{p_{\mathrm{r}}\lp  z \right)}{f'\lp z \right)} \rb + \mathfrak{F}_0\lp b_1 \right) \log \mathfrak{F}_0\lp b_1 \right)		\\
			= & \ \mbE_{z \sim cont\lp p_{\mathrm{r}} \right)}\lb \log  p_{\mathrm{r}}\lp  z \right) \rb - \mbE_{z \sim cont\lp p_{\mathrm{r}} \right)}\lb {\log f'\lp z \right)} \rb + \mathfrak{F}_0\lp b_1 \right) \log \mathfrak{F}_0\lp b_1 \right),
		\end{aligned}
	\end{equation}
	where we use the Monge-Amp\`ere equation $f_{\#}p_{\mathrm{r}}\lp f\lp z \right) \right) = \frac{p_{\mathrm{r}}(z)}{f'\lp z \right)}$ in one dimension. Moreover, notice that the last term corresponds to the entropy of the discrete part of distribution $f_{\#} p_{\mathrm{r}}$ as the ReLU mapping function maps $(-\infty, b_1]$ to $0$ and $cont\lp \cdot \right)$ refers to the continuous part of a distribution. Similarly, the relative entropy functional is given by
	\bequn
		\begin{aligned}
			\KL\lp f_{\#}p_{\mathrm{r}} \big\| \nu \right) = & \ \quad\mbE_{z \sim cont\lp p_{\mathrm{r}} \right)}\lb \log  p_{\mathrm{r}}\lp  z \right) \rb - \mbE_{z \sim cont\lp p_{\mathrm{r}} \right)}\lb {\log f'\lp z \right)} \rb 	\\
			& -\ \mbE_{z \sim cont\lp p_{\mathrm{r}} \right)}\lb {\log \nu\lp f\lp z \right) \right)} \rb + \mathfrak{F}_0\lp b_1 \right) \left(\log \mathfrak{F}_0\lp b_1 \right)\right).
		\end{aligned}
	\eequn
 \begin{comment}
		The above formula has some ambiguity, i.e. the KL-divergence only makes sense when $f_{\#}p$ is absolute continuous w.r.t. distribution $\nu$, which generally fails to hold due to the discrete part of $f_{\#}p$.
\end{comment}
		Moreover, the gradient flow of the KL-divergence differs from that of negative entropy
		only by a term that appears in the derivation in the potential functional gradient flow. This
		This also manifests in calculus on the density manifold between the heat and Fokker-Planck equations.
	Now, one calculates the derivative of continuous parts w.r.t. parameter $b_i$
	\bequn
		\begin{aligned}
			\partial_{b_i}\mbE_{x \sim p_{\mathrm{r}}}\lb {\log f'\lp x \right)} \rb & \ = \lbb\begin{aligned}
				& \log \frac{\sum_{j = 1}^{i - 1} a_j}{\sum_{j = 1}^{i} a_j} p_{\mathrm{r}}\lp b_i \right), \quad i \neq 1,		\\
				& - p_{\mathrm{r}}\lp b_1 \right) \frac{\log a_1}{N} , \hspace{1cm} i = 1.
			\end{aligned}\right.		\\
			\partial_{b_i}\mbE_{x \sim p_{\mathrm{r}}}\lb {\log \nu\lp f\lp x \right) \right)} \rb & \ = \mbE_{x \sim p_{\mathrm{r}}}\lb \frac{\nu'\lp y\lp x \right) \right) \partial_{b_i}y\lp x \right)}{\nu\lp y\lp x \right) \right)} \rb.
		\end{aligned}
	\eequn
	The first derivation is based on the observation that the function $\log f'\lp x \right)$ is a step function which changes its value at $b_i$. It takes value $\log \sum_{j = 1}^{i - 1} a_j$ at interval $\lb b_{i - 1}, b_i \rb$. Hence the desired conclusion follows, where $p_{\mathrm{r}}\lp b_i \right)$ comes in since this is the expectation w.r.t. distribution $p_{\mathrm{r}}\lp x \right)$. Therefore, the derivative of the entropy and relative entropy functional reads as follows
	\begin{equation}\label{equ:dHdb}
		\begin{aligned}
			\partial_{b_i} H\lp f_{\#}p_{\mathrm{r}} \right) = & \ \lbb\begin{aligned}
				& -\log \frac{\sum_{j = 1}^{i - 1} a_j}{\sum_{j = 1}^{i} a_j} p_{\mathrm{r}}\lp b_i \right), \hspace{2.1cm} i \neq 1,		\\
				& p_{\mathrm{r}}\lp b_1 \right)\lp \log \mathfrak{F}_0\lp b_1 \right) + 1 + \log \frac{a_1}{N} \right), \quad i = 1.
			\end{aligned}\right.		\\
			\partial_{b_i} \KL\lp f_{\#}p_{\mathrm{r}} \big\| \nu \right) = & \ \mbE_{x \sim p}\lb \frac{\nu'\lp f\lp x \right) \right) \partial_{b_i}f\lp x \right)}{\nu\lp f\lp x \right) \right)} \rb - \log \frac{\sum_{j = 1}^{i} a_j}{\sum_{j = 1}^{i - 1} a_j} p_{\mathrm{r}}\lp b_i \right).
		\end{aligned}
	\end{equation}
	With all these preparations, we can write out the gradient flow equation of the entropy functional:
\bequ\label{equ:heat-scheme}
	\begin{aligned}
		\dot{b_i} = & \ \frac{1}{\mathfrak{F}_0\lp b_{i} \right) - \mathfrak{F}_0\lp b_{i - 1} \right)}\lp \frac{\log\frac{\sum_{j = 1}^{i-1} a_j}{\sum_{j = 1}^{i} a_j} p_{\mathrm{r}}\lp b_i \right)}{a_i^2} - \frac{\log\frac{\sum_{j = 1}^{i - 2} a_j}{\sum_{j = 1}^{i - 1} a_j} p_{\mathrm{r}}\lp b_{i - 1} \right)}{a_ia_{i - 1}} \right) \\
		& \ + \frac{1}{\mathfrak{F}_0\lp b_{i + 1} \right) - \mathfrak{F}_0\lp b_{i} \right)}\lp \frac{\log\frac{\sum_{j = 1}^{i-1} a_j}{\sum_{j = 1}^{i} a_j} p_{\mathrm{r}}\lp b_i \right)}{a_i^2} - \frac{\log\frac{\sum_{j = 1}^{i} a_j}{\sum_{j = 1}^{i+1} a_j} p_{\mathrm{r}}\lp b_{i +1} \right)}{a_ia_{i + 1}} \right),	\quad i = 2, \cdots, N-1,	\\
		\dot{b_1} = & \ -\frac{1}{a_1(\mathfrak{F}_0(b_2) - \mathfrak{F}_0(b_1))}\lp \frac{p_{\mathrm{r}}\lp b_1 \right)\lp \log \mathfrak{F}_0\lp b_1 \right) + 1 + \log \frac{a_1}{N} \right)}{a_1} + \frac{\log \frac{\sum_{j = 1}^{1} a_j}{\sum_{j = 1}^{2} a_j} p_{\mathrm{r}}\lp b_2 \right)}{a_2} \right),		\\
		\dot{b_N} = & \ \frac{\log\frac{\sum_{j = 1}^{N-1} a_j}{\sum_{j = 1}^{N} a_j} p_{\mathrm{r}}\lp b_N \right)}{a_N^2(1 - \mathfrak{F}_0\lp b_{N} \right))} - \frac{1}{\mathfrak{F}_0\lp b_{N} \right) - \mathfrak{F}_0\lp b_{N - 1} \right)}\lp \frac{\log\frac{\sum_{j = 1}^{N - 2} a_j}{\sum_{j = 1}^{N-1} a_j} p_{\mathrm{r}}\lp b_{N - 1} \right)}{a_Na_{N - 1}} - \frac{\log\frac{\sum_{j = 1}^{N-1} a_j}{\sum_{j = 1}^{N} a_j} p_{\mathrm{r}}\lp b_N \right)}{a_N^2} \right).
	\end{aligned}
\eequ
Similar to the proof in \cref{prop:potential-flow-consistency}, one can carefully expand the neural projected dynamics of the entropy functional and prove that it converges to the heat equation in the limit that number of neurons tends to infinity and the gap between neurons nodes tends to zero. 

\begin{comment}
    Based on this set of equations, we can calculate the evolution equation for $f_{\mfb}(x)$ \cref{equ:2-layer-ReLU}. Assume 
$x \in [b_i, b_{i+1}]$, we have 
\begin{equation}\label{equ:heat-lagrangian}
	\begin{aligned}
		N\frac{d}{dt}f_{\mfb}(x) = & \ \sum_{j = 1}^i \frac{d}{dt}\sigma_* \lp x; \bm{\theta}_j\right) = \sum_{j = 1}^i -a_i\dot{b}_j		\\
		= & \ \frac{1}{\mathfrak{F}_0\lp b_{i + 1} \right) - \mathfrak{F}_0\lp b_{i} \right)}\lp \frac{\log\frac{\sum_{j = 1}^{i} a_j}{\sum_{j = 1}^{i+1} a_j} p_{\mathrm{r}}\lp b_{i +1} \right)}{a_{i + 1}} - \frac{\log\frac{\sum_{j = 1}^{i-1} a_j}{\sum_{j = 1}^{i} a_j} p_{\mathrm{r}}\lp b_i \right)}{a_i} \right).
	\end{aligned}
\end{equation}
Consequently, the moving direction depends on the absolute value of two factors inside the bracket.
\end{comment}

\subsubsection{Analysis of the long-time existence of the neural-projected heat flow}

In general, the projected Wasserstein gradient flow does not necessarily need to be a linear dynamics even though the original gradient flow is linear, e.g., the projected gradient flow corresponding to the heat equation is highly nonlinear. This poses great difficulties in analyzing and establishing the long-time existence of the projected dynamics, as mentioned in \cite{doi:10.1137/20M1344986}. Specifically, we focus on the nonlinear projected gradient flow of the entropy, which corresponds to the Heat equation in the full space. If we view all nodes $b_i, i \in [N]$ as grid points and view the scheme as an example of the moving mesh method \cite{huang2010adaptive}, then the mesh quality is an important quantity to observe during simulation. One does not want
the mesh quality to decrease too much and even become degenerate during the simulations. Therefore, we consider the well-posedness of the non-linear ODE \cref{equ:heat-scheme}.

\begin{proposition}\label{prop:mesh-quality-analysis}
    The neural projected dynamics \cref{equ:heat-scheme} of the heat flow is well-posed, e.g. the solution extends to arbitrary time.
\end{proposition}
\begin{proof}

We consider a special scenario when two adjacent nodes $b_i, b_{i+1}$ become 
close to each other while maintaining a relatively large gap with all other nodes, i.e.
\begin{equation}
	o(1) = b_{i+1} - b_i = o(b_p - b_q), \quad \forall q \in [N]\backslash\{i,i+1\}, \quad p = i, i+1.
\end{equation}
WLOG, we assume $b_{i+1} = b_i + \Delta b > b_i$ and reduce the following term which appears both in 
their time derivative expression in \cref{equ:heat-scheme}
\begin{equation}\label{equ:common-term}
	\begin{aligned}
		& \ \frac{1}{\mathfrak{F}_0\lp b_{i} \right) - \mathfrak{F}_0\lp b_{i - 1} \right)}\lp \frac{\log\frac{\sum_{j = 1}^{i-1} a_j}{\sum_{j = 1}^{i} a_j} p_{\mathrm{r}}\lp b_i \right)}{a_i^2} - \frac{\log\frac{\sum_{j = 1}^{i - 2} a_j}{\sum_{j = 1}^{i - 1} a_j} p_{\mathrm{r}}\lp b_{i - 1} \right)}{a_ia_{i - 1}} \right)		\\
		= & \ \lp \frac{1}{p_{\mathrm{r}}(b_i)\Delta b} + O(1)\right)\lp \log\frac{i-1}{i}p_{\mathrm{r}}(b_{i}) - \log\frac{i-2}{i-1}p_{\mathrm{r}}(b_{i-1}) \right) 	\\
		= & \ \lp \frac{1}{p_{\mathrm{r}}(b_i)\Delta b} + O(1)\right)\lp \log\frac{i-1}{i}p_{\mathrm{r}}(b_{i}) - \log\frac{i-2}{i-1}\lp p_{\mathrm{r}}(b_i) + O(\Delta b)\right) \right) 	\\
		= & \ \lp \frac{1}{p_{\mathrm{r}}(b_i)\Delta b} + O(1)\right)\lp \log\frac{i^2-2i+1}{i^2-2i}p_{\mathrm{r}}(b_i) + O(\Delta b)\right) 	\\
		= & \ \frac{1}{\Delta b} \log\frac{i^2-2i+1}{i^2-2i} + O(1) \rightarrow +\infty, \quad \Delta b \rightarrow 0^+, 	\\
	\end{aligned}
\end{equation}
where we use the simplified model where all the weights $a_i$ are set to $1$ and
Taylor expansion to conclude that $\mathfrak{F}_0\lp b_{i} \right) - \mathfrak{F}_0\lp b_{i - 1} \right) = p_{\mathrm{r}}(b_i)\Delta b + O(\Delta b^2)$ and $p(b_{i-1})$
follows the same spirit. This term appears with positive sign in the RHS of 
$\dot{b}_i$ and negative sign in the RHS of $\dot{b}_{i-1}$, indicating that the left (right) node $b_{i-1}$($b_i$)
will move fast towards left (right) respectively. This repulsion behavior guarantees that 
the Lagrangian coordinates will never collide with each other and the mesh degeneracy will not appear.

Next, we analyze our scheme using the time derivative of the Lagrangian coordinate.
It is a well-known result that under the heat flow the mean of the distribution is fixed. 
Therefore, due to the diffusive nature of the heat equation, one can imagine that the position of the quantile greater than the mean should move right in the heat equation and vice versa. Suppose $x \in [b_i, b_{i+1}]$ is a quantile with $b_i$ greater than the mean $0$. As the base measure
is a standard Gaussian distribution whose probability density function decreases over $[0, \infty)$,
we conclude that 
\begin{equation}
	0 < b_i < b_{i+1} \Longrightarrow p_{\mathrm{r}}(b_i) > p_{\mathrm{r}}(b_{i+1}) \Longrightarrow  \log\frac{\sum_{j = 1}^{i-1} a_j}{\sum_{j = 1}^{i} a_j} p_{\mathrm{r}}\lp b_i \right) < \log\frac{\sum_{j = 1}^{i} a_j}{\sum_{j = 1}^{i+1} a_j} p_{\mathrm{r}}\lp b_{i+1} \right) < 0.
\end{equation}
Consequently, the Lagrangian coordinate $f_{b}(z)$ is indeed moving towards right, which matchs the intuition from the heat equation.

\begin{comment}
In our simulation, instead of solving the heat equation, we rather simulate
a Fokker-Planck equation which differs from heat by a drifting term. All the 
nodes are updated by the same scheme \cref{equ:heat-scheme} except for the leftmost node
$b_1$, where we add a drift term which drives this node towards left, i.e.
\begin{equation}
	\dot{b_1} = -\frac{1}{a_1(\mathfrak{F}_0(b_2) - \mathfrak{F}_0(b_1))}\lp \frac{p_{\mathrm{r}}\lp b_1 \right)\lp 1 + \log \frac{a_1\mathfrak{F}_0\lp b_1 \right)}{N} \right)}{a_1} + \frac{\log \frac{\sum_{j = 1}^{1} a_j}{\sum_{j = 1}^{2} a_j} p_{\mathrm{r}}\lp b_2 \right)}{a_2} \right) - N c,
\end{equation}
where $c$ is the drift velocity we introduced. This may cause one issue: as the absolute value of $b_1$ become
too large, $\mathfrak{F}_0(b_1)$ will be rounded off to $0$ due to the machine precision. This will not affact 
the denominator $\mathfrak{F}_0(b_2) - \mathfrak{F}_0(b_1)$ as it becomes $\mathfrak{F}_0(b_2)$ directly. However, the term $\mathfrak{F}_0(b_1)$ blows up
in this case. The solution we used here is to automatically set $\mathfrak{F}_0(b_1) = \epsilon = 10^{-16}$ as long as
$|b_1|$ reachs a threshold value $M = 8$. Instead of modifying the value of $\mathfrak{F}_0(b_1)$, 
if we directly shift $b_1$ back to the threshold value $-M$ every time it moves to the left
of $-M$, this will restrict the range of the pushforward map, which deteriorates the accuracy
of the numerical scheme. Therefore, we favor to modify $\mathfrak{F}_0(b_1)$ instead of $b_1$ directly in our
implementation.
\end{comment}
\end{proof}

The neural projected dynamics can be understood as a Lagrangian scheme \cite{carrillo2021lagrangian, LIU2020109566, doi:10.1137/050633019} with neural network basis. Specifically, fixing basis as ReLU components in \cref{equ:ReLU-network}, one can view $a_i$'s and $b_i$'s as the shape and location coefficients of the basis functions respectively. Updating $a_i$'s is similar to classical finite-element method with fixed basis functions, while adding the degree of freedom of $b_i$'s is similar to the moving mesh method. The Lagrangian schemes can handle the problem of the free boundary such as porous medium, e.g. in \cite{LIU2020109566}, they use finite element method to solve the mapping function of the porous medium equation with high accuracy. While most Lagrangian schemes are based on updating the $a_i$'s parameters, our methods have more flexibility and expressivity as it takes more degree of freedom into account. The primal-dual structure of the Wasserstein gradient flow also leverages a lot of usage of Lagrangian schemes \cite{carrillo2021lagrangian}.

On the other hand, our numerical algorithm and the moving mesh method. The principal ingredients of the moving mesh method include the equidistribution principle, the moving mesh equation, and the method of lines approach \cite{tang2005moving}. The moving mesh equation is solved during the simulation to ensure the adaptivity such that the mesh can resolve to the detailed structure. In many classical moving mesh methods, the mesh equations are solved separately from the governing PDE itself to guarantee the adaptivity of the numerical methods. This implies that how the mesh change will not depend explicitly on the underlying PDE. There also exist moving mesh methods such that the mesh updates take into account of the governing PDE (e.g., the arbitrary Lagrangian-Eularian methods \cite{arbitrary}). From this perspective, the projected dynamics provide a PDE-specific moving mesh equation, i.e. the mesh moved according to the PDE dynamics to simulate which is more adaptive and efficient. Moreover, through a detailed study of the simple case, we can establish a theoretical guarantee on the quality of our moving mesh method in \cref{prop:mesh-quality-analysis}.

\subsection{Truncated error analysis for general neural projected Wasserstein gradient flow}\label{sec:general-truncated}
The proof of the consistency of the numerical scheme relies on the analytic formula derived before which is restrictive. In this section, we provide another methodology to prove the consistency of the
numerical scheme we derived in this paper. Instead of calculating the evolution of the
mapping explicitly, we calculate the deviation of the projected gradient w.r.t.
the original gradient direction. Let us first state a geometric proposition where we attempt
to be as general as possible. This result is also proved in \cite{doi:10.1137/20M1344986} and we prove it here for completeness.

Let $\mcX$ be a manifold (possibly infinite-dimensional) with a Riemannian metric $g_{\mcX}$, 
which provides an inner product on the tangent space $T_x\mcX$ (possibly infinite-dimensional Hilbert space) 
for each $x \in \mcX$. Let $\mcY \subset \mcX$ be its submanifold with induced metric denoted by $g_{\mcY}$, i.e.
$\forall y \in \mcY$:
\begin{equation*}
	g_{\mcY}(y): T_y\mcY \times T_y\mcY \rightarrow \mbR, \quad g_{\mcY}(y)(v, w) = g_{\mcX}(y)(v, w), \quad \forall v, w \in T_y\mcY.
\end{equation*}
Furthermore, let $H: \mcX \rightarrow \mbR$ be a functional defined over $\mcX$ and we use $\wtd H: \mcY \rightarrow \mbR$ for its restriction
on $\mcY$. We have the following proposition.
\begin{proposition}\label{prop:proj-gradient}
	Let $\nabla_{g_{\mcX}}H(y) \in T_y\mcX$ $(\nabla_{g_{\mcY}}\wtd H(y) \in T_y\mcY)$ denote the gradient of the functional 
	$H$ w.r.t. the metric $g_{\mcX}$ $(g_{\mcY})$ at $y \in \mcX$ $(y \in \mcY)$. Then, we have 
	\begin{equation}\label{equ:proj}
		\nabla_{g_{\mcY}}\wtd H(y) = \Pi(y)\nabla_{g_{\mcX}}H(y),
	\end{equation}
	where $\Pi(y)$ is the orthogonal projection operator from $T_y\mcX$ to $T_y\mcY$.
\end{proposition}
\begin{proof}
	As $\mcY$ is a submanfold of $\mcX$, we have inclusion map $\mathrm{I}(y): T_y\mcY \rightarrow T_y\mcX$ and
	restriction map $\mathrm{I}^*(y): T_y^*\mcX \rightarrow T_y^*\mcY$ for each $y \in \mcY$. Both mappings are 
	linear and are adjoint to each other. Therefore, viewing the metric tensor $g_{\mcY}(y)$ as a linear 
	mapping between $T_y\mcY \rightarrow T_y^* \mcY$, we have
	\begin{equation*}
		g_{\mcY}(y) = \mathrm{I}^*(y) \circ g_{\mcX}(y) \circ \mathrm{I}(y), \quad \forall y \in \mcY.
	\end{equation*}
	Moreover, the inner product $g_{\mcX}(y)$ on the Hilbert space $T_y\mcX$ induces an orthogonal decomposition:
	\begin{equation*}
		T_y\mcX = T_y\mcY \oplus T_y\mcY^{\perp}, \quad \forall y \in \mcY,
	\end{equation*}
	along with an orthogonal projection operator $\Pi(y)$. Now, recall that the Riemannian gradient $\nabla_{g_{\mcX}}H(y)$ is defined as 
	\begin{equation*}
		g_{\mcX}(y)\nabla_{g_{\mcX}}H(y) = dH(y).
	\end{equation*}
	The differential of $H(\cdot)$ and $\wtd H(\cdot)$ is related by
	\begin{equation*}
		d\wtd H(y) = \mathrm{I}^*(y)d H(y), \quad \forall y \in \mcY.
	\end{equation*}
	Therefore, gathering all the ingredients, we have the following commutative diagram
	\begin{equation*}
		\begin{tikzcd}
			& T_y\mcY \arrow[r, "\mathrm{I}(y)", yshift=0.7ex] \arrow[d, "g_{\mcY}(y)"'] & T_y\mcX \arrow[d, "g_{\mcX}(y)"] \arrow[l, "p_{\mathrm{r}}(y)", yshift=-0.7ex] &\rotatebox[origin=c]{180}{$\in$} \nabla_{g_{\mcX}}H(y)\\
			d\wtd H(y) \in & T_y^*\mcY  & T_y^*\mcX \arrow[l, "\mathrm{I}^*(y)"'] &\rotatebox[origin=c]{180}{$\in$} dH(y)
		\end{tikzcd}
	\end{equation*} 
	As $\Pi(y)$ is the orthogonal projection, we conclude that 
	\begin{equation*}
		\nabla_{g_{\mcY}}\wtd H(y) = (\mathrm{I}^*(y) g_{\mcX}(y) \mathrm{I}(y))^{-1}\mathrm{I}^*(y)d H(y) = \Pi(y)\nabla_{g_{\mcX}}H(y).
	\end{equation*}
\end{proof}
We can prove the consistency of our numerical schemes over different PDEs with the Wasserstein gradient flow 
structures by leveraging this proposition in the case $\mcX = P_2^{\infty}(\mbR)$ is the density manifold and 
$g_{\mcX}$ is chosen to be the $W_2$ metric. To achieve this, we can rewrite \cref{equ:proj} as
\begin{equation}
	\nabla_{g_{\mcY}}\wtd H(y) = \argmin_{v \in T_y\mcY}\norml \nabla_{g_{\mcX}}H(y) - v \normr_{g_{\mcX}}.
\end{equation}
Therefore, $\forall v \in T_y\mcY$ will provide an upper bound for the truncated error of our approximation scheme.
Moreover, if we assume that the submanifold $\mcY \subset P_2^{\infty}(\mbR)$ is identical to a generative model via mapping function $f_{\theta\#}$, 
i.e. $\mcY = f_{\theta\#}p_{\mathrm{r}}$ with $\theta \in \Theta$ and $p_{\mathrm{r}}$ the base measure. Then, the projected gradient direction can also be characterized
using the metric over the mapping space, i.e.
\begin{equation}\label{equ:projected-gradient}
	\nabla_{\Theta}\wtd H(\theta) = \argmin_{v \in T_{\theta}\Theta}\int (\nabla_{\theta} H(\theta)(x) - v(x))^2 f_{\theta\#}p_{\mathrm{r}}(x)dx,
\end{equation}
where $\theta$ is mapped to point $y$ and we abuse the notion of $\nabla_{\Theta} H(\theta)$ to denote the gradient vector in the mapping coordinate. 
Moreover, we can perform truncated error analysis directly over the mapping space, which is more convenient
and clear. Let us focus on the ReLU network mapping \cref{equ:ReLU-network}. The tangent space in the mapping coordinate is spanned 
by the vectors in \cref{equ:ReLU-network-derivative}. Meanwhile, the tangent space in the density coordinate is spanned by
\begin{equation}
	\partial_{b_i}f_{\theta\#}p_{\mathrm{r}}(x) = \frac{a_i}{N}p_{\mathrm{r}}'(x)\mathbf{1}_{[f(\theta, b_i), \infty)} ,\quad \partial_{a_i}f_{\theta\#}p_{\mathrm{r}}(x) =  \frac{f_{\theta}^{-1}(x)-b_i}{N}p_{\mathrm{r}}'(x)\mathbf{1}_{[f(\theta, b_i), \infty)} ,
\end{equation}
where the notation $f_{\theta}(\cdot ) = f(\theta, \cdot)$. Here we use the fact that the mapping $f_{\theta}$ is linear with slope $\frac{\sum_{j=1}^i a_j}{N}$ over the interval $[b_i, b_{i+1}]$. If $b_i$s are fixed, the projected dynamics belongs to projection-based model reduction \cite{benner2015survey} where the basis is fixed to be neurons. While changing $b_i$s correspond to model reduction with adaptive basis.
\begin{proposition}\label{proposition:order}
	The numerical scheme based on ReLU network mapping is consistent with order $2$ using both $a, b$ parameters and
	of order $1$ with either $a$ or $b$ parameters.
\end{proposition}
\begin{proof}	
	In view of \cref{equ:ReLU-network-derivative}, we have that the approximation using only $\partial_{b_i}f_{\theta}$ is simply 
	piece-wise constant approximation. As each ingredient has the shape of a Heaviside function, it is consistent with order $1$. While the approximation using both $\partial_{b_i}f_{\theta}$ and $\partial_{a_i}f_{\theta}$
	is a piece-wise linear approximation, thereby consistent of order $2$. This is because another set of ReLU-shape functions is added to the basis.
\end{proof}

The connection between the ReLU neural network and the linear finite element space is systematically studied in \cite{sci_2020}. They theoretically establish that at least two hidden layers are needed in a ReLU neural network to represent any linear finite element functions in $\Omega \subset \mbR^d$ when $d \geq 2.$

Based on this concrete understanding of the structure of the tangent space, we can calculate the local truncation error of the projected gradient flow.
\begin{theorem}\label{thm:consistent}
	Given a tangent vector $v(x) \in T_{f_{\theta\#}p_{\mathrm{r}}}\mcP(\mbR)$ whose approximated tangent vector in projected dynamics is given by $\nabla_{\theta}H(\theta)$, the local truncation error in the ReLU network mapping is given by 
	\begin{equation}\label{equ:relu-approximation}
		\sum_{i=1}^{N} \int_{b_i}^{b_{i+1}} v^2(f_{\theta}(z))p_{\mathrm{r}}(z)dx - \frac{\lp\int_{b_i}^{b_{i+1}} v(f_{\theta}(z))(z - m_i)p_{\mathrm{r}}(z)dz\right)^2}{\int_{b_i}^{b_{i+1}}(z - m_i)^2p_{\mathrm{r}}(z)dz} - \frac{\lp\int_{b_i}^{b_{i+1}} v(f_{\theta}(z))p_{\mathrm{r}}(z)dz\right)^2}{\mathfrak{F}_0(b_{i+1}) - \mathfrak{F}_0(b_i)}
	\end{equation}
	where $m_i$ is the center of mass of $p_{\mathrm{r}}(z)$ in $[b_i, b_{i+1}]$ and $b_{N+1}$ is understood as $+\infty$. Under the assumption that
	$v$ has bounded second order derivative and $b_{i+1} - b_i < \Delta b, \forall i.$
	\begin{equation}\label{equ:relu-approximation-taylor}
		\norml v(x) - \nabla_{\theta}H(\theta) \normr_{L^2(f_{\theta\#}p_{\mathrm{r}})}^2 = \frac{1}{4}\lp\frac{\sum_{j=1}^N a_j}{N}\right)^2\norml v''\normr_{\infty} O(\Delta b^4).
	\end{equation}
\end{theorem}
\begin{proof}
	As mentioned in the above theorem, the approximation using ReLU network mapping is equivalent to
	piecewise linear approximation in the mapping coordinate. Moreover, at each node $b_i$, the slope and value of the
	The function does not need to be continuous, which is exactly the same as the linear spline interpolation. The main difference 
	is that the grid points $b_i$ are not fixed since they can evolve over time. Therefore, we rewrite the optimization 
	problem \cref{equ:projected-gradient} as 
	\begin{equation}
		\arg\min_{c_i, d_i}\quad\sum_{i=1}^{N}\int_{f_{\theta}(b_i)}^{f_{\theta}(b_{i+1})} \lp v(x) - c_ix - d_i \right)^2 f_{\theta\#}p_{\mathrm{r}}(x)dx,
	\end{equation}
	which can be further reduced to $N-1$ separated optimization problem of $c_i, d_i$ over small interval $[f_{\theta}(b_i), f_{\theta}(b_{i+1})]$.
	For each subproblem, we have
	\begin{equation*}
		\begin{aligned}
			\int_{f_{\theta}(b_i)}^{f_{\theta}(b_{i+1})} \lp v(x) - c_ix - d_i \right)^2 f_{\theta\#}p_{\mathrm{r}}(x)dx
			= \int_{b_i}^{b_{i+1}} \lp v(f_{\theta}(z)) - c_if_{\theta}(z) - d_i \right)^2 p_{\mathrm{r}}(z)dz.
		\end{aligned}
	\end{equation*}
	This is a quadratic optimization problem of $c_i, d_i$ with positive definite Hessian matrix. Taking derivative w.r.t. $c_i, d_i$, we obtain
	\begin{equation*}
		\begin{aligned}
			\int_{b_i}^{b_{i+1}} \lp v(f_{\theta}(z)) - c_if_{\theta}(z) - d_i \right) p_{\mathrm{r}}(z)dz	& = 0,		\\
			\int_{b_i}^{b_{i+1}} f_{\theta}(z)\lp v(f_{\theta}(z)) - c_if_{\theta}(z) - d_i \right) p_{\mathrm{r}}(z)dz	& = 0.
		\end{aligned}
	\end{equation*}
	Now, using the fact that $f_{\theta}(z)$ is a linear function over the interval $[b_i, b_{i+1}]$, we have
	\begin{equation}\label{equ:linear-spline}
		\begin{aligned}
			c_if_{\theta}(z) + d_i = \frac{\int_{b_i}^{b_{i+1}} v(f_{\theta}(z))(z - m_i)p_{\mathrm{r}}(z)dx}{\int_{b_i}^{b_{i+1}}(z - m_i)^2p_{\mathrm{r}}(z)dz}(z - m_i) + \frac{\int_{b_i}^{b_{i+1}} v(f_{\theta}(z))p_{\mathrm{r}}(z)dx}{\mathfrak{F}_0(b_{i+1}) - \mathfrak{F}_0(b_i)}.
		\end{aligned}
	\end{equation}
	Plugging back, we obtain the approximation error as 
	\begin{equation}
		\begin{aligned}
			& \int_{b_i}^{b_{i+1}} \lp v(f_{\theta}(z)) - c_if_{\theta}(z) - d_i \right)^2 p_{\mathrm{r}}(z)dz		\\
			= & \ \int_{b_i}^{b_{i+1}} v(f_{\theta}(z))\lp v(f_{\theta}(z)) - c_if_{\theta}(z) - d_i \right) p_{\mathrm{r}}(z)dz		\\
			= & \ \int_{b_i}^{b_{i+1}} v^2(f_{\theta}(z))p_{\mathrm{r}}(z)dz - \frac{\lp\int_{b_i}^{b_{i+1}} v(f_{\theta}(z))(z - m_i)p_{\mathrm{r}}(z)dz\right)^2}{\int_{b_i}^{b_{i+1}}(z - m_i)^2p_{\mathrm{r}}(z)dz} - \frac{\lp\int_{b_i}^{b_{i+1}} v(f_{\theta}(z))p_{\mathrm{r}}(z)dz\right)^2}{\mathfrak{F}_0(b_{i+1}) - \mathfrak{F}_0(b_i)}.
		\end{aligned}
	\end{equation}
	Next, we assume all the intervals $[b_i, b_{i+1}]$ are short (of scale $O(\Delta)$) and consider expanding the $v$ as Taylor series around $m_i$, i.e.
\begin{equation}
	\begin{aligned}
	v(f_{\theta}(z)) = & \ v(f_{\theta}(m_i)) + \frac{\sum_{j=1}^i a_j}{N}v'(f_{\theta}(m_i))(z - m_i) \\
	& \ + \frac{1}{2}\lp\frac{\sum_{j=1}^i a_j}{N}\right)^2v''(f_{\theta}(m_i))(z - m_i)^2 + O(\Delta^3).
	\end{aligned}
\end{equation}
Here, we use the fact that $f_{\theta}(z)$ is a linear function with slope $\frac{\sum_{j=1}^i a_j}{N}$ over the interval $[b_i, b_{i+1}]$. 
Plugging into \cref{equ:linear-spline}, we obtain
\begin{equation}
	\begin{aligned}
	c_if_{\theta}(z) - d_i = & \ \frac{\sum_{j=1}^i a_j}{N}v'(f_{\theta}(m_i))(z - m_i) + v(f_{\theta}(m_i)) \\
	& \ + \frac{1}{2}\lp\frac{\sum_{j=1}^i a_j}{N}\right)^2v''(f_{\theta}(m_i))\frac{\int_{b_i}^{b_{i+1}} (z-m_i)^2p_{\mathrm{r}}(z)dz}{\mathfrak{F}_0(b_{i+1}) - \mathfrak{F}_0(b_i)} \\
	& \ + \frac{1}{2}\lp\frac{\sum_{j=1}^i a_j}{N}\right)^2v''(f_{\theta}(m_i))\frac{\int_{b_i}^{b_{i+1}} (z - m_i)^3p_{\mathrm{r}}(z)dz}{\int_{b_i}^{b_{i+1}}(z - m_i)^2p_{\mathrm{r}}(z)dz}(z - m_i) + O(\Delta^3).
	\end{aligned}
\end{equation}
Notice that the first two terms are exactly the zero-th and first order term of the $v(f_{\theta}(z))$ function which is similar to classical linear function approximation by discarding all the higher order term. The appearance of residue terms is due to approximation in $L^2(p)$ sense. To calculate the 
$L^2$-approximation error, we have
\begin{equation}
	\begin{aligned}
		& \ \lb \frac{1}{2}\lp\frac{\sum_{j=1}^i a_j}{N}\right)^2v''(f_{\theta}(m_i))\rb^2 \\
		& \ \int_{b_i}^{b_{i+1}} \lp (z-m_i)^2 - \frac{\int_{b_i}^{b_{i+1}} (z-m_i)^2p_{\mathrm{r}}(z)dz}{\mathfrak{F}_0(b_{i+1}) - \mathfrak{F}_0(b_i)} - \frac{\int_{b_i}^{b_{i+1}} (z - m_i)^3p_{\mathrm{r}}(z)dz}{\int_{b_i}^{b_{i+1}}(z - m_i)^2p_{\mathrm{r}}(z)dz}(z - m_i) + O(\Delta^3)\right)^2 p_{\mathrm{r}}(z)dz		
		\\ = & \  \lb \frac{1}{2}\lp\frac{\sum_{j=1}^i a_j}{N}\right)^2v''(f_{\theta}(m_i))\rb^2 O((b_{i+1}-b_i)^5)p_{\mathrm{r}}(b_i) + O((b_{i+1}-b_i)^6)p_{\mathrm{r}}(b_i).
	\end{aligned}
\end{equation}
In summary, the $L^2$ approximation error consists of the sum over all the interval $[b_i, b_{i+1}]$, with each term depends on $a_i$  through
the factor $\frac{\sum_{j=1}^i a_j}{N}$, on $b_i$ through $(b_{i+1}-b_i)^5$ and the term $v''(f_{\theta}(m_i))$, which also contains $a_i, b_i$.
\end{proof}

Let us calculate a special case of the Fokker-Planck equation
\begin{equation*}
	\partial_t p(t,x) - \nabla \cdot (p(t,x) \nabla V(x)) - \gamma\Delta p(t,x) = 0.
\end{equation*}
Under the Wasserstein metric, the tangent vector in the mapping space is given by
\begin{equation*}
	v(x) = -V'(x) - \gamma \frac{ p'(t,x)}{ p(t,x)}.
\end{equation*}
In this case, we have that
\begin{equation*}
	v''(x) = -V^{(3)}(x) - \gamma \frac{p^{(3)}(t,x)p(t,x)^2 + 2p'(t,x)^3 - 3p(t,x)p'(t,x)p''(t,x)}{p(t,x)^3}.
\end{equation*}
The above function will determine the approximation quality of the projected dynamics.

\begin{remark}
    The high-order neural mapping function class and associated high-order projected dynamics can also be derived following a similar procedure. For example, we can add a quadratic term of the ReLU function into the network mapping function as
    \bequ\label{equ:2-layer-ReLU-highorder}
	f\lp \theta, z \right) = \frac{1}{N}\sum_{i = 1}^N a_i\sigma \lp z - b_i \right) + c_i\sigma^2 \lp z - b_i \right).
    \eequ
    Notice that adding high order ReLU term is different from increasing the layers in the ReLU neural network which corresponds to function composition. We leave the detailed analysis and numerical experiments on high-order methods in future work.
\end{remark}

\begin{comment}
Next, we design a higher order scheme by adding the power of the ReLU function into our ReLU model
\bequ\label{equ:2-layer-ReLU-highorder}
	y\lp x; \bm{\theta} \right) = \frac{1}{N}\sum_{i = 1}^N a_i\sigma \lp x - b_i \right) + u_i\sigma^2 \lp x - b_i \right).
\eequ
Under the framework of \cref{prop:proj-gradient}, this model corresponds to projected the Wassserstein gradient on to
the linear space spanned by all the quadratic polynomials. Therefore, the truncation error based on this model will 
be of order 3. We directly state the theorem that characterizes the truncation error for this higher order scheme without 
proof as it is similar to the previous one.
\begin{Thm}
	Given a tangent vector $v(x) \in v_{f_{\theta\#} p \mcP(\mbR)}$, the approximation error in the ReLU network mapping is given by 
	\begin{equation}
		\sum_{i=1}^{N} \lp\int_{b_i}^{b_{i+1}} v(f_{\theta}(z)) p(x)dx\right)^2 - \frac{\lp\int_{b_i}^{b_{i+1}} v(f_{\theta}(z))(x - m_i) p(x)dx\right)^2}{\int_{b_i}^{b_{i+1}}(x - m_i)^2 p(x)dx} - \frac{\lp\int_{b_i}^{b_{i+1}} v(f_{\theta}(z)) p(x)dx\right)^2}{\mathfrak{F}_0(b_{i+1}) - \mathfrak{F}_0(b_i)}
	\end{equation}
	where $m_i$ is the center of mass of $ p(x)$ in $[b_i, b_{i+1}]$ and $b_{N+1}$ is understood as $+\infty$.
	\begin{equation}
		\frac{1}{8}\lp\frac{\sum_{j=1}^N a_j}{N}\right)^3\norml v'''\normr_{\infty} O(\Delta b^6).
	\end{equation}
\end{Thm}
\end{comment}

\section{Numerical Examples}\label{sec5}
In this section, we provide several numerical experiments to test our algorithm and theories. We focus our attention on the linear transport equation, Fokker-Planck equation, porous medium equations, and Keller-Segel equation. They all correspond to some specific energy functionals in the probability space equipped with the Wasserstein-2 distance. 

\subsection{Neural Network structure}\label{nn struct} 
We first describe the structure of our neural network for the experiment. We focus on two-layer neural network with ReLU as activation functions.  
\begin{equation}\label{eq:nn}
    f(\theta,z) = \sum_{i = 1}^N a_i\cdot \sigma(z-b_i) + \sum_{i = N+1}^{2N} a_i \cdot\sigma(b_i-z) \,.
\end{equation}
Here $\theta \in \RR^{4N}$ represents the collection of weights $\{a_i\}_{i=1}^{2N}$ and bias $\{b_i\}_{i=1}^{2N}$. To simplify our notation, we have absorbed the $1/N$ factor into $a_i$'s. At initialization, we set $a_i = 1/N$ for $i\in\{1,\ldots,N\}$ and $a_i = -1/N$ for $i\in\{N+1,\ldots,2N\}$. To choose the $b_i$'s, we first set $\mathbf{b} = \textrm{linspace}(-B,B,N)$ for some positive constant $B$ (e.g. $B=4$ or $B=10$). We then set $b_i = \mathbf{b}[i]$ for $i=1,\ldots,N$ and $b_j= \mathbf{b}[j-N]+\varepsilon$ for $j=N+1,\ldots,2N$. Here $\varepsilon=5\times 10^{-6}$ is a small offset which will be explained later in Section \ref{sec:FPK}. Our initialization is chosen such that $f(\theta,\cdot)$ approximates the identity map at initialization. In practice, we find it beneficial to perform a rescaling of the weights $a_i$'s. We replace $a_i$ with $\overline a_i/\beta$ for some fixed constant $\beta >0$. And we initialize $\overline a_i = \beta/N$ for $i\in\{1,\ldots,N\}$ and $\overline a_i = -\beta/N$ for $i\in\{N+1,\ldots,2N\}$. This rescaling makes sure that $f(\theta,\cdot)$ still approximates the identity map at initialization. We provide a brief intuition for rescaling. Let us consider $g(a,b,z) = a\cdot \sigma(b-z)$ for $b=\mathcal{O}(1)$, $z = \mathcal{O}(1)$, $a=\mathcal{O}(1/N)$. Then $\partial_a g = \sigma(z-b) = \mathcal{O}(1)$. On the other hand, $\partial_b g = a \cdot \sigma'(b-z) = \mathcal{O}(1/N)$. This simple calculation shows that the partial gradient of $\eqref{eq:nn}$ with respect to weights and bias are of different scales. Therefore, to make them the same scale, a natural choice is choosing $\beta = \mathcal{O}(N)$. 
\begin{remark}
    The choice of neural network \eqref{eq:nn} is slightly more complicated than the one studied in Section \ref{sec4}. This symmetric structure is used in numerical experiments to overcome ReLU's drawback such that only the positive input is activated. Moreover, \eqref{eq:nn} allows us to construct an approximation to the identity map over $\RR$ easily. However, the results of Proposition \ref{proposition:order} still hold for \eqref{eq:nn}. And Theorem \ref{thm:consistent} can be generalized to neural network of the form given in \eqref{eq:nn} in a straightforward manner. The metric tensor $G_{\mathrm{W}}$ is now a $4N\times 4N$ matrix. The calculations of the individual components of $G_{\mathrm{W}}$ follow the same procedure presented in \cref{prop:metric-2layer}. 
\end{remark}
\begin{remark}
    We remind our readers that our algorithm takes the form of 
    \begin{equation*}
\theta^{k+1}=\theta^k-h G_{\mathrm{W}}(\theta)^\dagger\nabla_\theta \tilde F(\theta^k)\,.
\end{equation*}
During implementation, $\nabla_{\theta}\tilde F(\theta)$ can be obtained by backpropagating $\tilde F(\theta)$ in the case of Example \ref{ex4} and Example \ref{ex5}. However, we need to pay special attention to $\partial_{b_i} F(\theta)$ when dealing with Example \ref{ex6}. This will be elaborated further in Section \ref{sec:FPK} and Section \ref{sec:PM}. 
\end{remark}
\begin{comment}
\begin{remark}
    Our algorithm is related to the moving mesh method. Viewing each ReLU components in \cref{eq:nn} as basis functions, one can view $a_i$'s and $b_i$'s as the shape and location coefficients of the basis functions, respectively. Updating $a_i$'s is similar to the classical finite-element method with fixed basis functions, while adding the degree of freedom of $b_i$'s is similar to the moving mesh method. In many classical moving mesh methods, the mesh equations are solved separately from the governing PDE itself to guarantee the adaptivity of the numerical methods. This implies that how the mesh change will not depend explicitly on the underlying PDE. There also exist moving mesh methods such that the mesh updates take into account of the governing PDE (e.g., the arbitrary Lagrangian-Eularian methods \cite{arbitrary}). However, this is still quite different from our approach where $b_i$'s and $a_i$'s are updated together and are correlated with each other by the metric tensor $G_{\mathrm{W}}$. Moreover, our framework is Lagrangian in nature whereas traditional moving mesh methods follow mostly Eulerian approach. \JX{I am thinking if we should combine this part of review on moving mesh with previous review in \cref{rem:moving-mesh-lagrangian}} \xz{That is an option. We just need to mention that we are updating both weights and bias in the previous section.}
\end{remark}
\end{comment}

\subsection{Linear transport PDE}\label{sec:linear_pde}
We investigate the linear transport PDE given by Eq.~\eqref{eq:transport_pde} with several choices of potential $V(x)$, corresponding to the gradient flow of them under the Wasserstein metric. For a simple potential function, this example can serve as a sanity check of the projected dynamics formulation. The trajectories of the particles for Eq.~\eqref{eq:transport_pde} (i.e. Lagrangian formulation) follows the following ODE 
\begin{equation}\label{eq:transport_pde_lagrangian}
    \dot x(t) = -\nabla V(x) \,.
\end{equation}
Let us denote by $T(t,z_0)$ the solution to Eq.~\eqref{eq:transport_pde_lagrangian} with initial condition $x(0) = z_0$. In other words, $T(t,z_0)$ is the transport map at time $t$ starting from position $z_0$. We define the error at time $t$ by 
\begin{align}
    \mathrm{error} &= \int_{-\infty}^{\infty} |f(\theta_t,z_0) - T(t,z_0)| p_0(z_0) \;dz_0 \nonumber \\
    &\approx \frac{1}{N_1}\sum_{j=1}^{N_1} |f(\theta_t,z_j) - T(t,z_j)| p_0(z_j)\,,\label{eq:error_def}
\end{align}
where we discretize the integration domain by $N_1$ equally spaced points to approximate the integral. And $p_0(z_0)$ denotes the initial distribution of $z_0$. Below we test our projected dynamics under three choices of potential functions and investigate the convergence behavior of two projected dynamics: (i) fixing the bias terms $b_i$ and only updating the weights $a_i$ and (ii) updating both bias $b_i$ and weights $a_i$. Note that when the bias terms $b_i$'s are fixed, we have that $G_{\mathrm{W}} \in \RR^{2N\times 2N}$. Recall that we are essentially simulating the gradient flow on parameter $\theta_t$ given by Eq.~\eqref{eq:gf_linear}. We use $M=5\times 10^5$ particles sampled from a standard Gaussian distribution for approximating $\mathbb{E}_{\tilde z\sim p_{\mathrm{r}}}\Big[V(f(\theta,\tilde z))\Big]$. Once we have the empirical loss function 
$$
 \mathbb{E}_{\tilde z\sim p_{\mathrm{r}}}\Big[ V(f(\theta,\tilde z))\Big] \approx \frac{1}{M}\sum_{i=1}^M V(f(\theta, z_i))\,,
$$
we can backpropagate this loss to obtain 
$$
\mathbb{E}_{\tilde z\sim p_{\mathrm{r}}}\Big[ \nabla_{\theta} V(f(\theta, \tilde z))\Big] \approx \frac{1}{M}\sum_{i=1}^M \nabla_{\theta} V(f(\theta, z_i))\,,
$$
which will be used in the update of $\theta_t$ given by Eq.~\eqref{eq:gf_linear}. 

\subsubsection{Quadratic potential}
As the first example for linear transport PDE, we consider the quadratic potential $V(x)=\frac{1}{2}(x-\mu_0)^2$ as a sanity check. The stationary distribution will be the delta mass supported at $\mu_0$. Using the method of characteristics, one can show that the solution at time $t>0$ is given by 
\begin{equation}
    p(t,x) = p_0\big((x-\mu_0)e^t + \mu_0\big) e^t \,,
\end{equation}
where $p_0(x) = p(0,x)$ is the initial distribution. In Lagrangian coordinates, the transport map of a point $z_0$ at time $t$ is given by 
\begin{equation}\label{eq:analytic_map_x_power_2}
   T(t,z_0) = \mu_0 + e^{-t}(z_0-\mu_0)\,.
\end{equation}
One can check that $T(z_0,0) =z_0$ and $T(t,z_0)\to \mu_0$ as $t\to \infty$. It is worthwhile mentioning that at each $t>0$, the Lagrangian map $x_t(z_0):z_0 \mapsto T(t,z_0)$ is a \emph{linear map}. For simplicity, we take $ \delta_0 = 0$. We choose $dt=10^{-3}$ and run for 1000 steps. We compare our numerical results with Eq.~\eqref{eq:analytic_map_x_power_2}. The result is demonstrated in Fig.~\ref{fig:linear_quadratic}. In Fig.~\ref{fig:x_power_2_mapping}, we have provided a visualization of the analytic solution to the linear transport PDE in Lagrangian coordinates at $t=1$ and our computed solution. As shown in Fig.~\ref{fig:x_power_2_mapping}, the analytic transport map is linear while the neural mapping function is piecewise linear. Increasing $N$ does not necessarily give a smaller approximation error. In fact, we see in Fig.~\ref{fig:linear_quadratic_err} that larger $N$ usually gives a larger error, commonly known as overfitting in machine learning.

\begin{figure}
     \centering
     \begin{subfigure}[b]{0.45\textwidth}
         \centering
         \includegraphics[width=\textwidth]{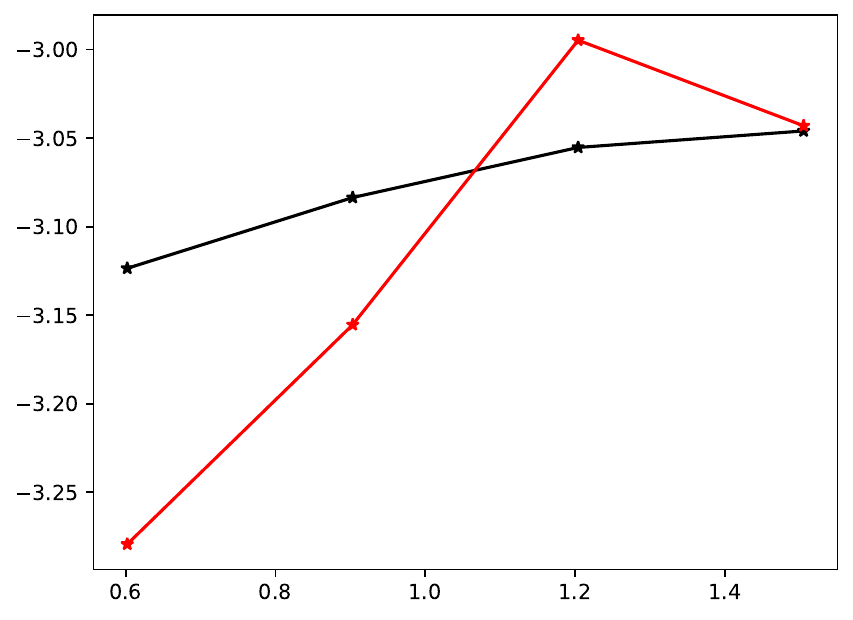}
         \caption{Error}\label{fig:linear_quadratic_err}
     \end{subfigure}
     \begin{subfigure}[b]{0.43\textwidth}
         \centering
         \includegraphics[width=\textwidth]{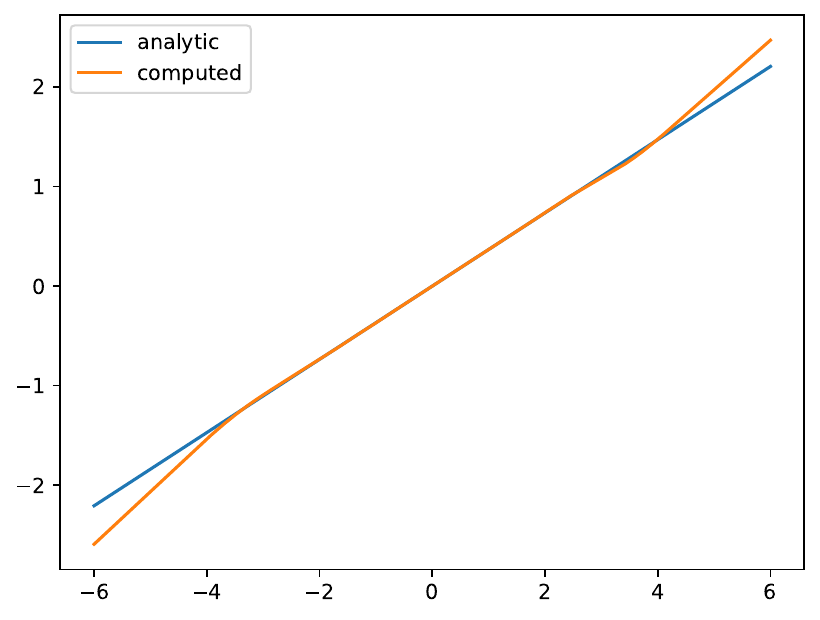}
         \caption{Mapping comparison}\label{fig:x_power_2_mapping}
     \end{subfigure}
        \caption{Left: log-log plot of linear transport PDE with a quadratic potential. The y-axis represents $\log_{10}$ error defined by \eqref{eq:error_def}. x-axis represents $\log_{10}(N)$. The bias terms $b_i$ are initialized based on Section \ref{nn struct} with $B=4$. Red line represents results when only the weights terms are updated. Black line represents results when both weights and bias are updated. Right: Mapping comparison between $T(t,z)$ given by Eq.~\eqref{eq:analytic_map_x_power_2} and our computed solution $f(\theta_t,z)$.}
        \label{fig:linear_quadratic}
\end{figure}

\subsubsection{Quartic potential } 
Let us consider $V(x)=(x-1)^4/4-(x-1)^2/2$. The analytic solution of the transport map is given by
\begin{equation}\label{eq:analytic_map_x_power_4}
    T(t,z_0) = \left\{ \begin{aligned}   
     & \mathrm{sgn}(z_0-1)\frac{e^t}{\sqrt{(z_0-1)^{-2}+e^{2t}-1}}+1, \quad z_0 \neq 1\,, \\
     & 1, \hspace{6.5cm} z_0 = 1\,.
     \end{aligned}\right.
\end{equation}
Basic settings are the same as the previous case. We choose $dt=2\times10^{-4}$ and run for 1000 steps. We compare our numerical results with Eq.~\eqref{eq:analytic_map_x_power_4}.  We present our results in Fig.~\ref{fig:linear_x_4}. In Fig.~\ref{fig:linear_x_4_err}, we observe a clear decrease in error as the number of neurons becomes larger. In Fig.~\ref{fig:x_power_4_mapping}, we visualize the analytic solution to the linear transport PDE in Lagrangian coordinates at $t=0.2$ and our computed solution. We can see that even when the optimal transport map is nonlinear, our computed solution still matches the analytic solution very accurately. 

\begin{figure}
     \centering
     \begin{subfigure}[b]{0.45\textwidth}
         \centering
         \includegraphics[width=\textwidth]{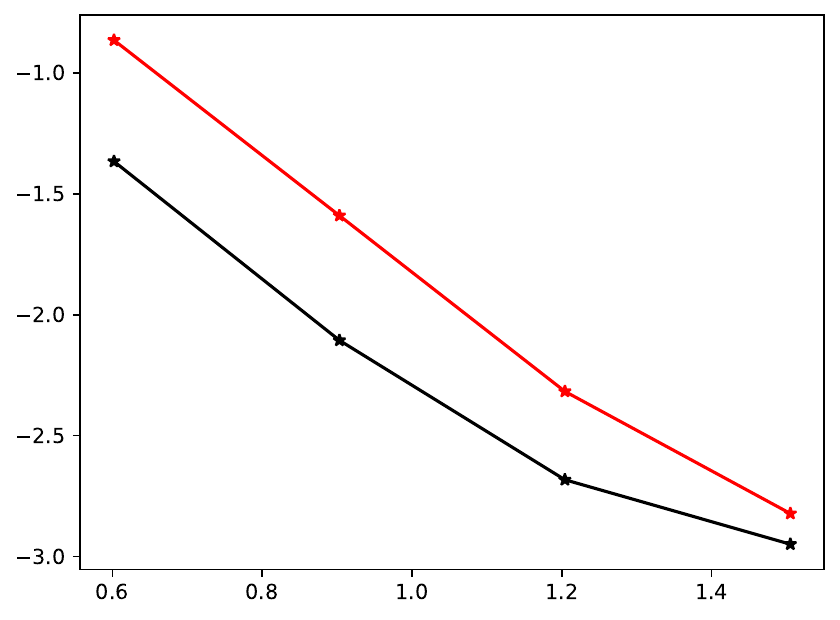}
         \caption{Error}\label{fig:linear_x_4_err}
     \end{subfigure}
     \begin{subfigure}[b]{0.43\textwidth}
         \centering
         \includegraphics[width=\textwidth]{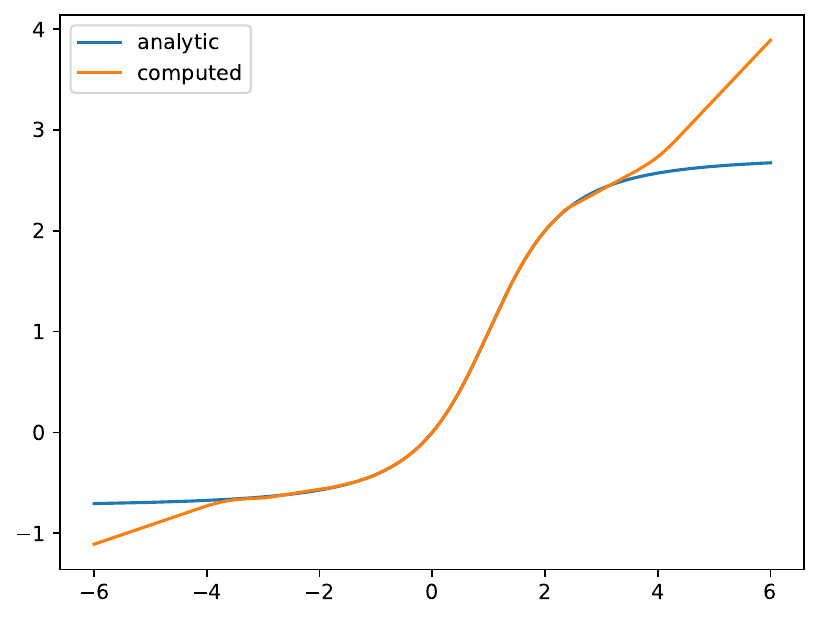}
         \caption{Mapping comparison} \label{fig:x_power_4_mapping}
     \end{subfigure}
        \caption{Left: log-log plot of linear transport PDE with quartic polynomial potential. The y-axis represents $\log_{10}$ error defined by \eqref{eq:error_def}. x-axis represents $\log_{10}(N)$. The bias terms $b_i$ are initialized based on Section \ref{nn struct} with $B=4$. Red line represents results when only the weights terms are updated. Black line represents results when both weights and bias are updated. Right: Mapping comparison between $T(t,z)$ given by Eq.~\eqref{eq:analytic_map_x_power_4} and our computed solution $f(\theta_t,z)$.}
        \label{fig:linear_x_4}
\end{figure}

\subsubsection{Sixth order polynomial potential} 
Let us consider $V(x) = (x-4)^6/6$. The analytic solution of the transport map is given by
\begin{equation}
\label{eq:analytic_map_x_power_6}
    T(t,z_0) = \left\{ \begin{aligned}
        & 4 + \mathrm{sgn}(z_0-4) \frac{1}{\sqrt{2 \sqrt{\frac{1}{4(z_0-4)^4}+t}}}, \quad z_0 \neq 4\,, \\
        & 4 ,\hspace{5.68cm} z_0 =4\,.
    \end{aligned} \right.
\end{equation}
We choose $dt=10^{-6}$ and run for 1000 steps. The reason to choose such a small step size is that the ODE \eqref{eq:transport_pde_lagrangian} is stiff when $V(x)$ is a sixth order polynomial. This can be readily seen by considering the forward Euler scheme for solving \eqref{eq:transport_pde_lagrangian}, which results in the popular gradient descent algorithm. The step size that can guarantee convergence in gradient descent is at most $2/L$ where $L$ is the Lipschitz constant of the gradient function. In our case, the gradient function $\nabla V(x)$ is not globally Lipschitz. Even if we consider a fixed interval $(-l,l)$, the Lipschitz constant is $L = 5(l+4)^4$. If we take $l=10$, then we get $L = \mathcal{O}(10^{-6})$.  We compare our numerical results with Eq.~\eqref{eq:analytic_map_x_power_6}. We have chosen $\{z_j\}_{j=1}^{N_1}$ to be a uniform mesh of size $N_1 = 4\times 10^6$ on $[-6,6]$ in Eq.~\eqref{eq:error_def} and $p_0$ is the standard Gaussian distribution. Note that $N_1$ is chosen to be much larger than the number of neurons $N$ in the network mapping function as it is used to evaluate the accuracy of our algorithm. We present our results in Fig.~\ref{fig:linear_x_6}. We can see a clear decrease in error when $N$ increases from Fig.~\ref{fig:linear_x_6_err}. It is also clear from Fig.~\ref{fig:linear_x_6_err} that updating both weights and bias tends to have a smaller error than just updating the weights, although the difference becomes smaller when $N$ increases and more mesh points become available. Comparing dashed and solid lines in Fig.~\ref{fig:linear_x_6_err}, we find that the initialization of $b_i$ also plays a role in the overall performance of our solution. The error is smaller when the initial mesh points (i.e., the $b_i$'s) are more concentrated near the center of the reference measure. In our case, the reference measure is a standard Gaussian, whose measure is ``almost'' supported on $[-4,4]$. Hence we see that the solid lines show a smaller error than the dashed lines in Fig.~\ref{fig:linear_x_6_err}. In Fig.~\ref{fig:x_power_6_mapping}, we have given a visualization of the analytic solution to the linear transport PDE in Lagrangian coordinates at $t=10^{-3}$ and our computed solution. It is worth noting from Fig.~\ref{fig:x_power_6_mapping} that our learned Lagrangian map approximates the analytic Lagrangian map well near the center of the reference distribution, which is concentrated near the origin. Even though the error of the learned Lagrangian map is larger outside of $[-4,4]$, the overall error from Eq.~\eqref{eq:error_def} is still small since the reference measure (standard Gaussian measure) on $\RR \setminus [-4,4]$ is exponentially small. 
\begin{remark}
    According to Proposition \ref{proposition:order}, updating both $a_i$ and $b_i$ is a second order method. This can be seen from Fig.~\ref{fig:linear_x_6_err} when $N$ is small. When $N$ is large, the numerical advantage of updating both $a_i$ and $b_i$ is less significant compared with updating only $a_i$. This is partially explained by the condition number of the $G_{\mathrm{W}}(\theta)$ grows too large when $\theta$ contains all of $a_i$ and $b_i$. This phenomenon is also observed in our other experiments. Using the implicit scheme or proximal scheme (without solving the linear system that involves $G_{\mathrm{W}}(\theta)$ directly) might help with this difficulty, which we leave as a future study.
\end{remark}

\begin{figure}
     \centering
     \begin{subfigure}[b]{0.45\textwidth}
         \centering
         \includegraphics[width=\textwidth]{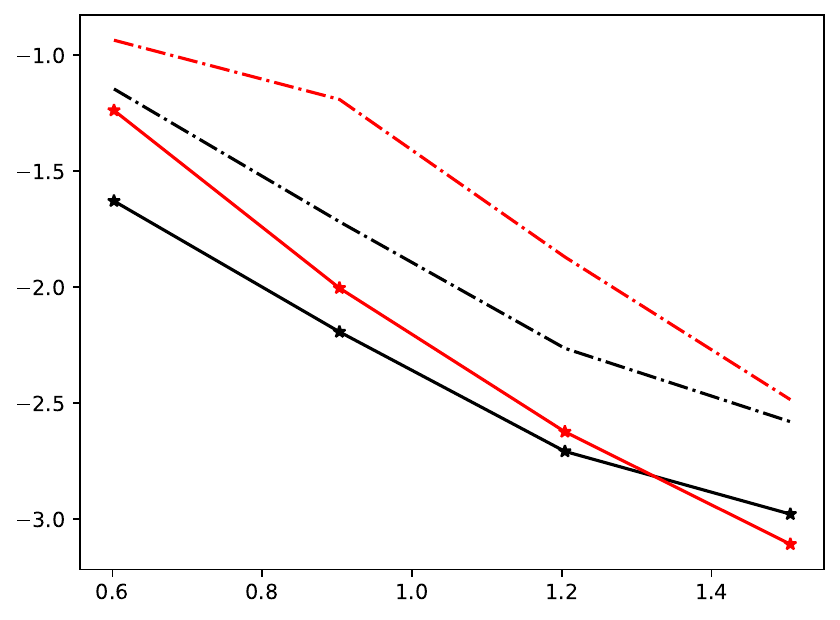}
         \caption{Error}\label{fig:linear_x_6_err}
     \end{subfigure}
     \begin{subfigure}[b]{0.43\textwidth}
         \centering
         \includegraphics[width=\textwidth]{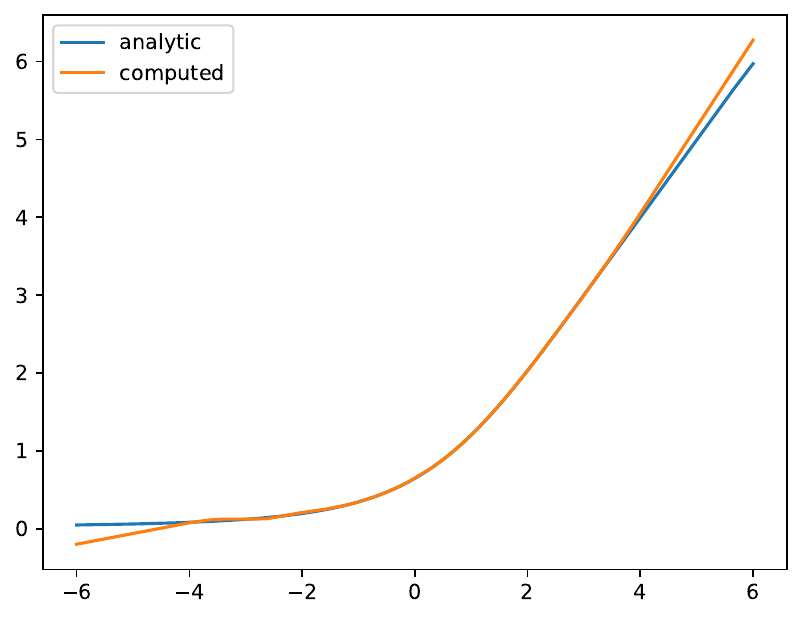}
         \caption{Mapping comparison}\label{fig:x_power_6_mapping}
     \end{subfigure}
        \caption{Left: log-log plot of linear transport PDE with sixth order polynomial potential. The y-axis represents $\log_{10}$ error defined by \eqref{eq:error_def}. x-axis represents $\log_{10}(N)$. The bias terms $b_i$ are initialized based on Section \ref{nn struct} with $B=10$ for dashed line and $B=4$ for solid line. Red lines represent results when only the weights terms are updated. Black lines represent results when both weights and bias are updated. Right: Mapping comparison between $T(t,z)$ given by Eq.~\eqref{eq:analytic_map_x_power_6} and our computed solution $f(\theta_t,z)$. }
        \label{fig:linear_x_6}
\end{figure}

\subsection{Fokker-Planck Equation}\label{sec:FPK}
We consider Fokker-Planck equations. In general, there is no closed-form solution for either the Eulerian or Lagrangian coordinate except for some special forms of potential $V$ (e.g. quadratic). We can still have an approximation of the analytic transport map by realizing that the optimal transport map of a point $z_0$ at time $t$ is given by 
\begin{equation}\label{eq:trasnport_numerical}
    T(t,z_0) = \mathfrak{F}_t\big({\mathfrak{F}_0^{-1}}(z_0)\big)
\end{equation}
where $\mathfrak{F}_t$ is the cumulative distribution function (CDF) of $p(t,x)$. $\mathfrak{F}_0$ has a closed form expression when we choose our reference measure to be a standard Gaussian. But we still need to know $p(t,x)$. Therefore, to investigate the performance of our algorithm, we need to use a numerical solver to solve for $p(t,x)$. We choose a center difference in space, implicit in time discretization as our choice of numerical solver with vanishing boundary condition. Recall that we are essentially simulating the gradient flow on parameter $\theta_t$ given by Eq.~\eqref{eq:gf_linear} and Eq.~\eqref{eq:gf_entropy}. To calculate the derivative of the energy functionals, we used $M=10^6$ particles sampled from a standard Gaussian distribution for approximating $\mathbb{E}_{ z\sim p_{\mathrm{r}}}\Big[ V(f(\theta, z))+\hat U(\frac{p_{\mathrm{r}}(z)}{D_zf(\theta,z)})\Big]$. Approximating $\mathbb{E}_{ z\sim p_{\mathrm{r}}}\Big[ \nabla_{\theta}V(f(\theta, z))\Big]$ is straightforward and has been explained in detail in Section \ref{sec:linear_pde}. On the other hand, some care needs to be taken when approximating $\mathbb{E}_{ z\sim p_{\mathrm{r}}}\Big[ \nabla_{\theta}\hat U(\frac{p_{\mathrm{r}}(z)}{D_zf(\theta,z)})\Big]$ as explained in Section \ref{sec:gradient_neg_entropy}. Suppose that all of the $\{b_k\}_{k=1}^{2N}$ are different. Take $2\leq j \leq N$. Let us also assume that the $b_k$'s are ordered so that $b_1\leq b_2\leq \cdots\leq b_N$. 
\begin{align}
    \mathbb{E}_{ z\sim p_{\mathrm{r}}}\partial_{b_j} \log(D_z f(\theta,z)) &= \mathbb{E}_{ z\sim p_{\mathrm{r}}}\partial_{b_j} \log\left(\sum_{i=1}^N a_i \mathbf{1}_{[b_i,\infty)}(z) - \sum_{i=N+1}^{2N} a_i \mathbf{1}_{(-\infty,b_i]}(z)\right) \nonumber \\
    &= p_{\mathrm{r}}(b_j)\log \left( \frac{\sum_{i=1}^{j-1} a_i  - \sum_{i=N+1}^{2N} a_i \mathbf{1}_{(-\infty,b_i]}(b_j)}{\sum_{i=1}^j a_i  - \sum_{i=N+1}^{2N} a_i \mathbf{1}_{(-\infty,b_i]}(b_j)}\right)\,.
\end{align}
And 
\begin{equation}
    \mathbb{E}_{ z\sim p_{\mathrm{r}}}\partial_{b_1} \log(D_z f(\theta,z)) 
    = p_{\mathrm{r}}(b_1)\log \left( \frac{\sum_{i=N+1}^{2N} -a_i \mathbf{1}_{(-\infty,b_i]}(b_1)}{a_1  - \sum_{i=N+1}^{2N} a_i \mathbf{1}_{(-\infty,b_i]}(b_1)}\right)\,.
\end{equation}
Similarly, if we assume that $b_{N+1}\geq b_{N+2} \geq \cdots \geq b_{2N}$ and let $N+2\leq j\leq 2N$, we have 
\begin{align}
    \mathbb{E}_{ z\sim p_{\mathrm{r}}}\partial_{b_j} \log(D_z f(\theta,z)) &= \mathbb{E}_{ z\sim p_{\mathrm{r}}}\partial_{b_j} \log\left(\sum_{i=1}^N a_i \mathbf{1}_{[b_i,\infty)}(z) - \sum_{i=N+1}^{2N} a_i \mathbf{1}_{(-\infty,b_i]}(z)\right) \nonumber \\
    &= p_{\mathrm{r}}(b_j)\log \left( \frac{\sum_{i=1}^{N} a_i \mathbf{1}_{[b_i,\infty)}(b_j) - \sum_{i=N+1}^{j} a_i }{\sum_{i=1}^{N} a_i \mathbf{1}_{[b_i,\infty)}(b_j)  - \sum_{i=N+1}^{j-1} a_i }\right)\,.
\end{align}
And 
\begin{equation}
    \mathbb{E}_{ z\sim p_{\mathrm{r}}}\partial_{b_{N+1}} \log(D_z f(\theta,z)) 
    = p_{\mathrm{r}}(b_{N+1})\log \left( \frac{\sum_{i=1}^{N} a_i \mathbf{1}_{[b_i,\infty)}(b_{N+1}) -  a_{N+1} }{\sum_{i=1}^{N} a_i \mathbf{1}_{[b_i,\infty)}(b_{N+1})}\right)\,.
\end{equation}
Note that during implementation, we do not have to order the $b_j$'s in order to evaluate the above partial derivatives. Let $0<\delta \leq \frac{1}{2} \min_{i\neq j} |b_i-b_j|$. Then by a straightforward calculation, we have 
\begin{equation}
     \mathbb{E}_{ z\sim p_{\mathrm{r}}}\partial_{b_j} \log(D_z f(\theta,z)) =\begin{cases}
      p_{\mathrm{r}}(b_j)  \log\left(\frac{D_z f(\theta,b_j -\delta )}{ D_z f(\theta,b_j)} \right)\,, \quad 1\leq j \leq N\,. \\
       p_{\mathrm{r}}(b_j)  \log\left(\frac{D_z f(\theta,b_j  )}{ D_z f(\theta,b_j+\delta)} \right)\,, \quad N+1\leq j \leq 2N\,.
      \end{cases}
\end{equation}
In our experiment, we set $\delta=\varepsilon/2$ where $\varepsilon$ is the small offset we introduced in Section \ref{nn struct} during initialization. 
\subsubsection{Quadratic potential}

As a first example for the Fokker-Planck equation, we use the quadratic potential as a sanity check. Here $V(x)$ is chosen to be a quadratic function. This is one of the rare cases where the Fokker-Planck equation has a closed-form analytic solution. In Lagrangian coordinates, the trajectories of the particles follow the following SDE, commonly known as the Ornstein-Uhlenbeck process: 
\begin{equation}\label{eq:sde_OU}
    dX_t = -\gamma_0(X_t-\mu_0)dt + \sigma_0 dW_t\,.
\end{equation}
The corresponding Langevin equation for the density $p(t,x)$ is given by 
\begin{equation}\label{eq:langevin_OU}
    \frac{\partial p}{\partial t} = \gamma_0 \frac{\partial }{\partial x}\big((x-\mu_0)p \big ) + D \frac{\partial^2 p}{\partial x^2}\,,
\end{equation}
where $D = \sigma_0^2/2$. It can be shown that the solution to \eqref{eq:langevin_OU} is given by 
\begin{equation}\label{eq:solution_OU}
    p(t,x) = \sqrt{\frac{\gamma_0}{2\pi D (1-\mathrm{e}^{-2\gamma_0 t})}}\int_{-\infty}^{\infty} \mathrm{exp}\Big(-\frac{\gamma_0}{2D}\frac{(x-\mu_0-x' \mathrm{e}^{-\gamma_0 t})^2}{1-\mathrm{e}^{-2\gamma_0 t}}\Big)p_0(x')\, \mathrm{d}x' \,,
\end{equation}
where $p_0(x) = p(0,x)$ is the initial distribution. In our experiment, $p_0(x)$ is a standard Gaussian. Then \eqref{eq:solution_OU} implies that $p(t,x)$ is also Gaussian with mean $\mu_0(1-e^{-\gamma_0 t})$ and variance $e^{-2\gamma_0 t} + \frac{D(1-e^{-2\gamma_0 t})}{\gamma_0}$. Then the transport map is given by the optimal transport map between two Gaussians, which has a closed form expression. In this example, the transport map is 
\begin{equation}\label{eq:transport_map_OU}
    T(t,z) = \mu_0(1-e^{-\gamma_0 t}) + z \sqrt{e^{-2\gamma_0 t} + D(1-e^{-2\gamma_0 t})/\gamma_0 }\,,
\end{equation}
which is always a linear map, no matter the choice of $\mu_0$, $\gamma_0$ and $D$. We use $M=10^6$ particles sampled from a standard Gaussian distribution for approximating $\mathbb{E}_{ z\sim p_{\mathrm{r}}}\Big[\nabla_\theta V(f(\theta, z))+\nabla_\theta \hat U(\frac{p_{\mathrm{r}}(z)}{D_zf(\theta,z)})\Big]$. We choose $dt=10^{-3}$ and run for 1000 steps. We used a neural network with $m=32$ and $B=4$ following the setup in Section \ref{nn struct}. We have the following two choices of parameters corresponding to different dynamics. 
\begin{itemize}
    \item \emph{Moving and widening Gaussian.} We choose $\gamma_0=1$, $\mu_0=30$, $\sigma_0 = 4$. Under this setting, the solution at time $t$ is a Gaussian distribution with mean $30(1-e^{-t})$ and variance $e^{-2t} + 8(1-e^{-2t}) $. This evolution is shown on the left panel of Fig.~\ref{fig:FPK_2_density}. 
    \item \emph{Moving and shrinking Gaussian.} We choose $\gamma_0=1$, $\mu_0=10$, $\sigma_0 = 0.01$. Under this setting, the solution at time $t$ is a Gaussian distribution with mean $10(1-e^{-t})$ and variance $e^{-2t} + 5\times 10^{-5}(1-e^{-2t}) $. This evolution is shown on the right panel of Fig.~\ref{fig:FPK_2_density}.  
\end{itemize}
Our results are demonstrated in Fig.~\ref{fig:FPK_2_density}. As shown in Fig.~\ref{fig:FPK_2_density}, the computed density closely follows the analytic density of the Fokker-Planck equation from $t=0$ to $t=1$.
\begin{figure}
     \centering
     \begin{subfigure}[b]{0.4\textwidth}
         \centering
         \includegraphics[width=\textwidth]{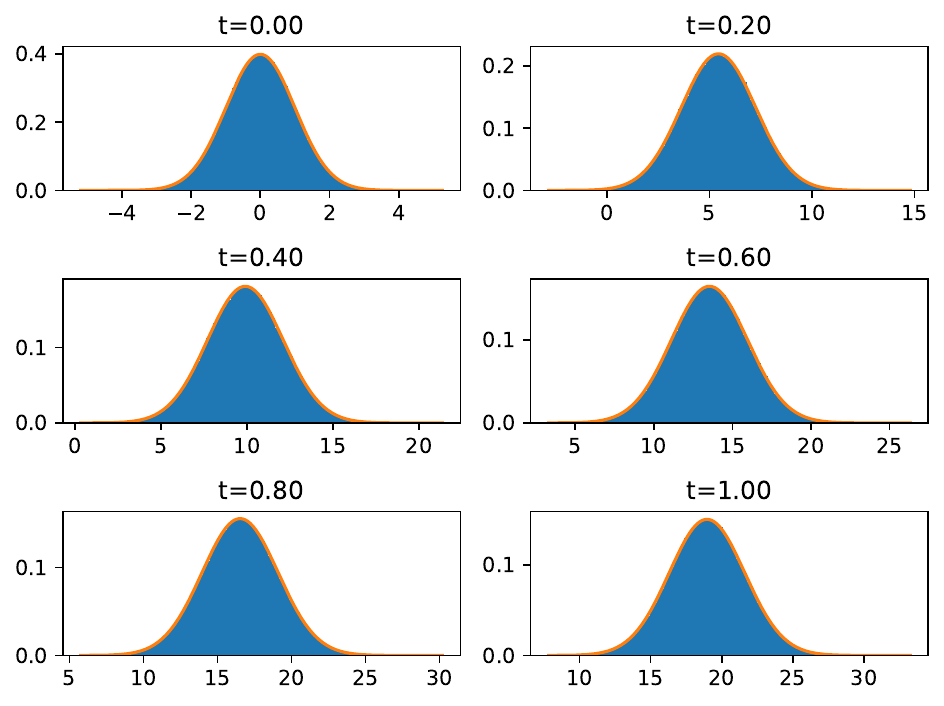}
         \caption{Moving, widening Gaussian}
     \end{subfigure}
     \begin{subfigure}[b]{0.4\textwidth}
         \centering
         \includegraphics[width=\textwidth]{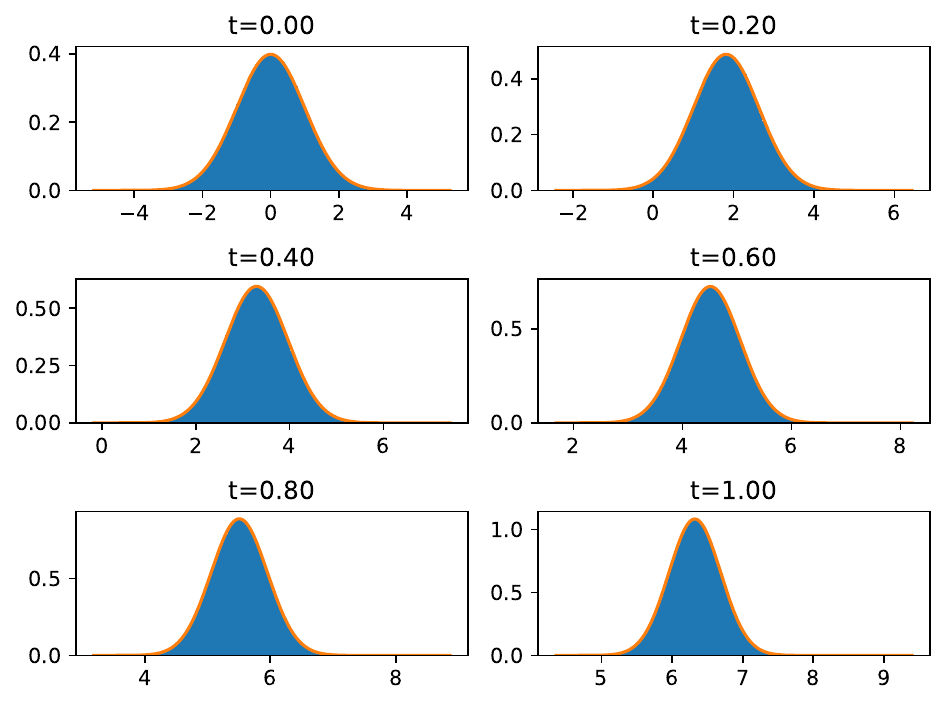}
         \caption{Moving, shrinking Gaussian}
     \end{subfigure}
        \caption{Density evolution of Eq.~\eqref{eq:langevin_OU}. Orange curve represents the solution given by Eq.~\eqref{eq:solution_OU}. Blue rectangles represent the histogram using $10^6$ particles in 100 bins from $t=0$ to $t=1$. Left panel: a Gaussian distribution shifting to the right with increasing variance. Right panel: a Gaussian distribution shifting to the right with decreasing variance. }\label{fig:FPK_2_density}
\end{figure}

\subsubsection{Quartic potential}
We consider $V(x) = (x-1)^4/4-(x-1)^2/2$. We choose $dt=2\times10^{-4}$ and run for 1000 steps. We compare our numerical results with the transport map computed from Eq.~\eqref{eq:trasnport_numerical}. The results are shown in Fig.~\ref{fig:FPK_power4}. In Fig.~\ref{fig:FPK_power4_err}, we observe a clear decrease in error when the number of neurons increases. In Fig.~\ref{fig:FPK_power4_map}, we plot a comparison between our computed Lagrangian map $f(\theta,z)$ vs the transport map computed from Eq.~\eqref{eq:trasnport_numerical} using a numerical solver. The evolution of the density is demonstrated in Fig.~\ref{fig:FPK_power4_density} from $t=0$ to $t=0.2$. 

\begin{figure}
 \centering
     \begin{subfigure}[b]{0.32\textwidth}
         \centering
         \includegraphics[width=\textwidth]{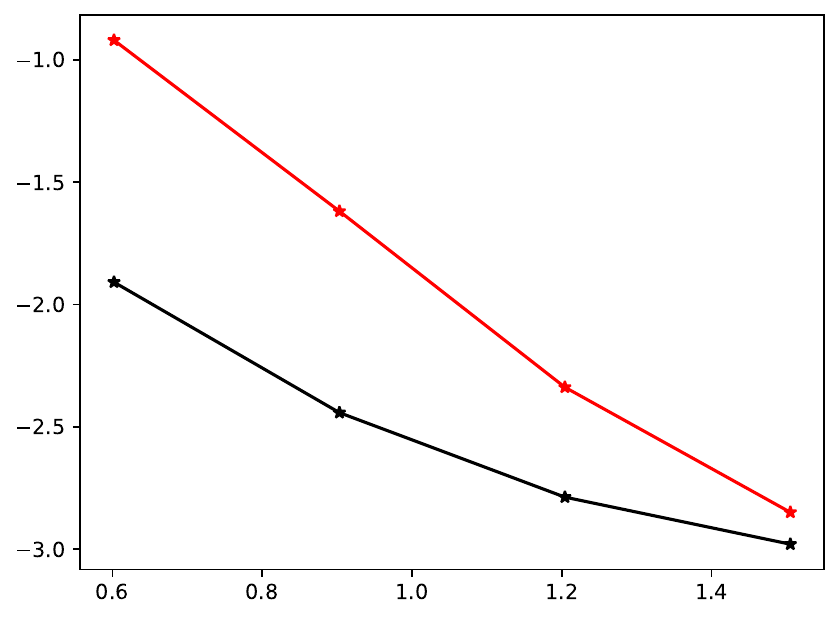}
         \caption{Error} \label{fig:FPK_power4_err}
     \end{subfigure}
     \centering
     \begin{subfigure}[b]{0.31\textwidth}
         \centering
         \includegraphics[width=\textwidth]{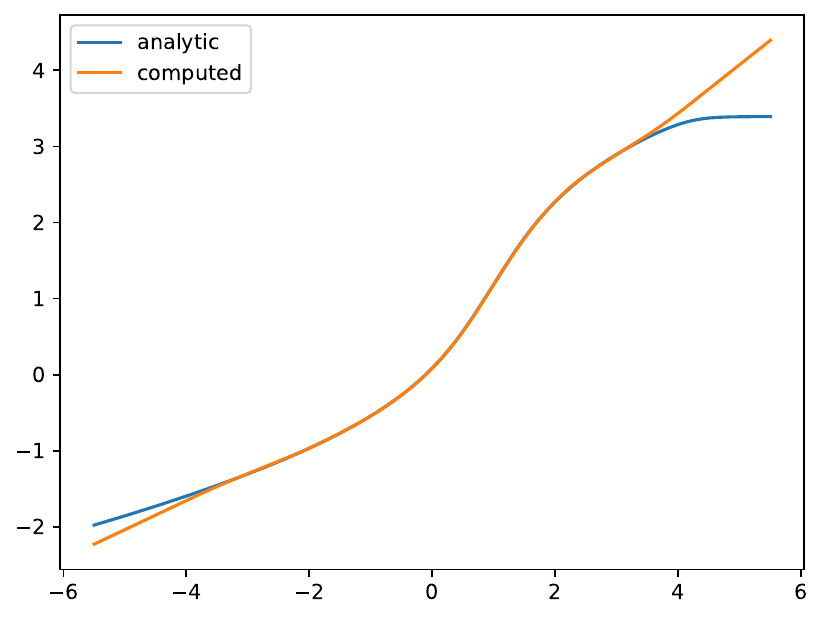}
         \caption{Mapping comparison} \label{fig:FPK_power4_map}
     \end{subfigure}
     \begin{subfigure}[b]{0.33\textwidth}
         \centering
         \includegraphics[width=\textwidth]{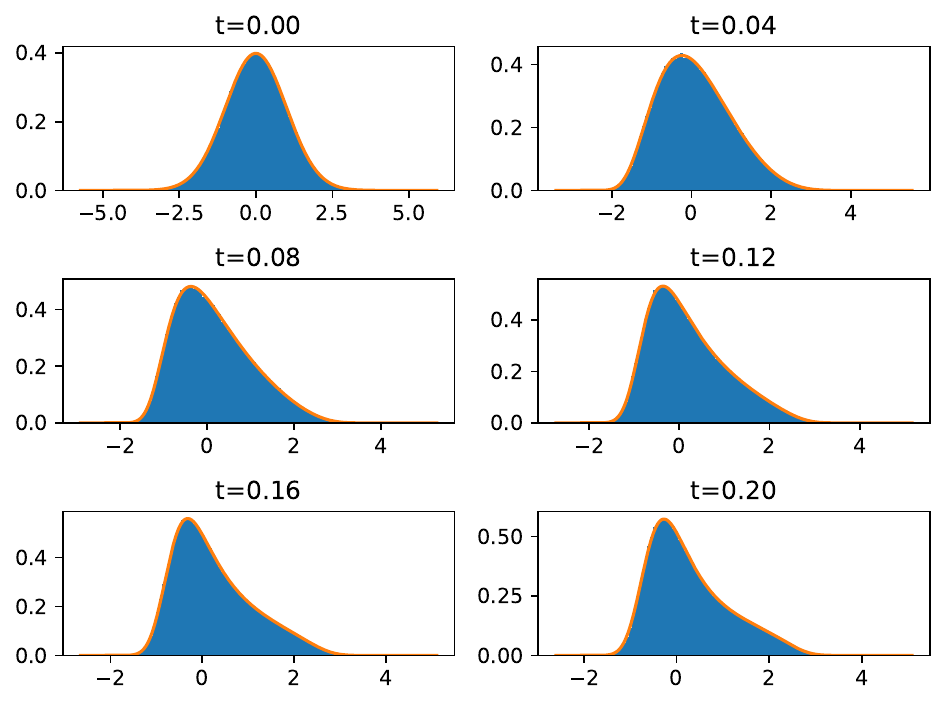}
         \caption{Density evolution} \label{fig:FPK_power4_density}
     \end{subfigure}
        \caption{Left: log-log plot of Fokker-Planck equation with a quartic polynomial potential. The y-axis represents $\log_{10}$ error defined by \eqref{eq:error_def}. x-axis represents $\log_{10}(N)$. The bias terms $b_i$ are initialized based on Section \ref{nn struct} with $B=4$. Red lines represent results when only the weights terms are updated. Black lines represent results when both weights and bias are updated. Middle: mapping comparison between $T(t,z)$ (using Eq.~\eqref{eq:trasnport_numerical}) and our computed solution $f(\theta_t,z)$. Right: density evolution of the Fokker-Planck equation with a quartic polynomial potential. Orange curve represents the density $p(t,x)$ computed by a numerical solver. Blue rectangles represent the histogram of $10^6$ particles in 100 bins from $t=0$ to $t=0.2$.  }
        \label{fig:FPK_power4}
\end{figure}

\subsubsection{Sixth order polynomial potential }
We consider $V(x)=(x-4)^6/6$. We choose $dt=10^{-6}$ and run for 1000 steps. We compare our numerical results with the transport map computed from Eq.~\eqref{eq:trasnport_numerical}. The results are shown in Fig.~\ref{fig:FPK_6}. We have observed similar behavior as in the case of linear transport PDE: the error becomes smaller when $N$ increases. Moreover, comparing dashed and solid lines in Fig.~\ref{fig:FPK_6_err} we see that as the initial mesh points (i.e. the $b_i$'s) concentrate nearer the center of our reference measure, the errors are smaller. In Fig.~\ref{fig:FPK_power6_map} we show a comparison between Lagrangian maps computed by our method and the numerical solver. We have also plotted the evolution of the density in Fig.~\ref{fig:FPK_power6_density} from $t=0$ to $t=10^{-3}$.

\begin{figure}[ht]
     \centering
     \begin{subfigure}[b]{0.32\textwidth}
         \centering
         \includegraphics[width=\textwidth]{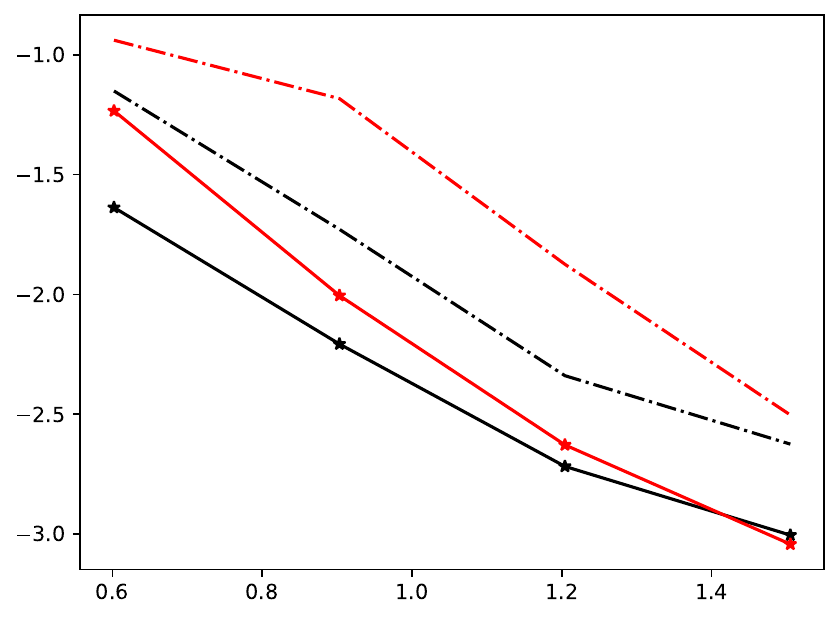}
         \caption{Error}\label{fig:FPK_6_err}
     \end{subfigure}
     \begin{subfigure}[b]{0.31\textwidth}
         \centering
         \includegraphics[width=\textwidth]{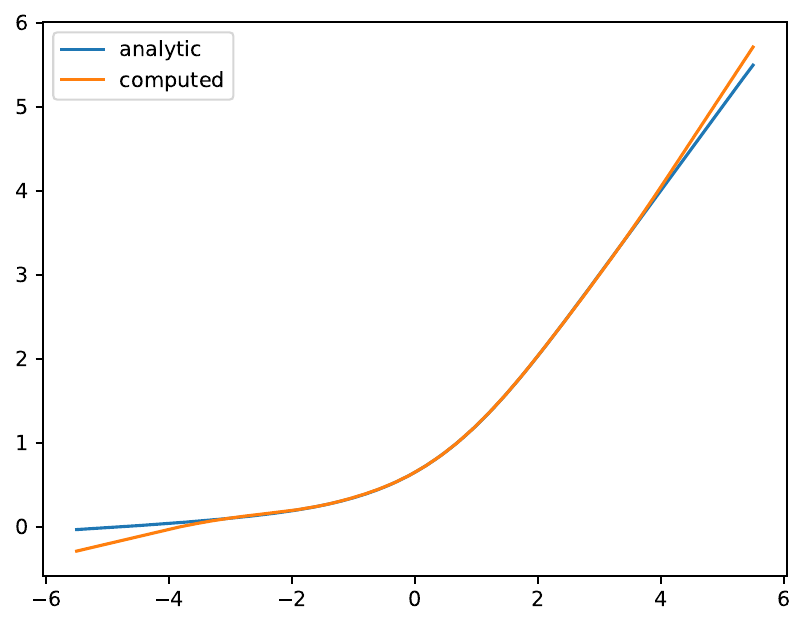}
         \caption{Mapping comparison} \label{fig:FPK_power6_map}
     \end{subfigure}
      \begin{subfigure}[b]{0.33\textwidth}
         \centering
         \includegraphics[width=\textwidth]{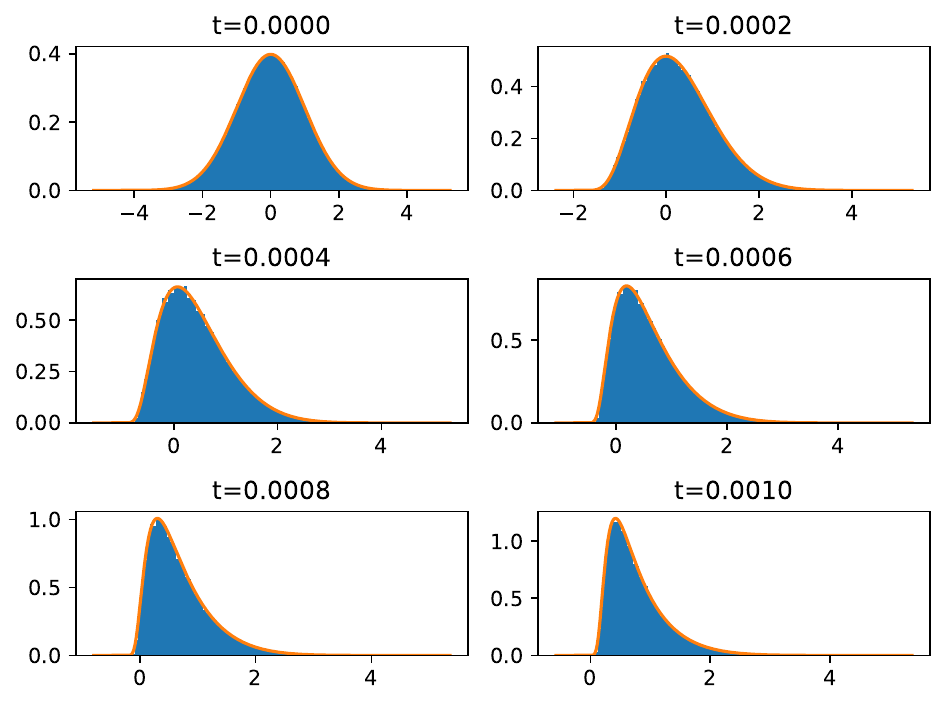}
         \caption{Density evolution} \label{fig:FPK_power6_density}
     \end{subfigure}
        \caption{Left: log-log plot of Fokker-Planck equation with a sixth order polynomial potential. The y-axis represents $\log_{10}$ error defined by \eqref{eq:error_def}. x-axis represents $\log_{10}(N)$. The bias terms $b_i$ are initialized based on Section \ref{nn struct} with $B=10$ for the dashed line and $B=4$ for the solid line. Red lines represent results when only the weights terms are updated. Black lines represent results when both weights and bias are updated. Middle: mapping comparison between $T(t,z)$ (using Eq.~\eqref{eq:trasnport_numerical}) and our computed solution $f(\theta_t,z)$. Right: density evolution of the Fokker-Planck equation with a sixth order polynomial potential. Orange curve represents the density $p(t,x)$ computed by a numerical solver. Blue rectangles represents the histogram of $10^6$ particles in 100 bins from $t=0$ to $t=10^{-3}$.  }
        \label{fig:FPK_6}
\end{figure}

\begin{figure}
\centering
     \begin{subfigure}[b]{0.32\textwidth}
         \centering
         \includegraphics[width=\textwidth]{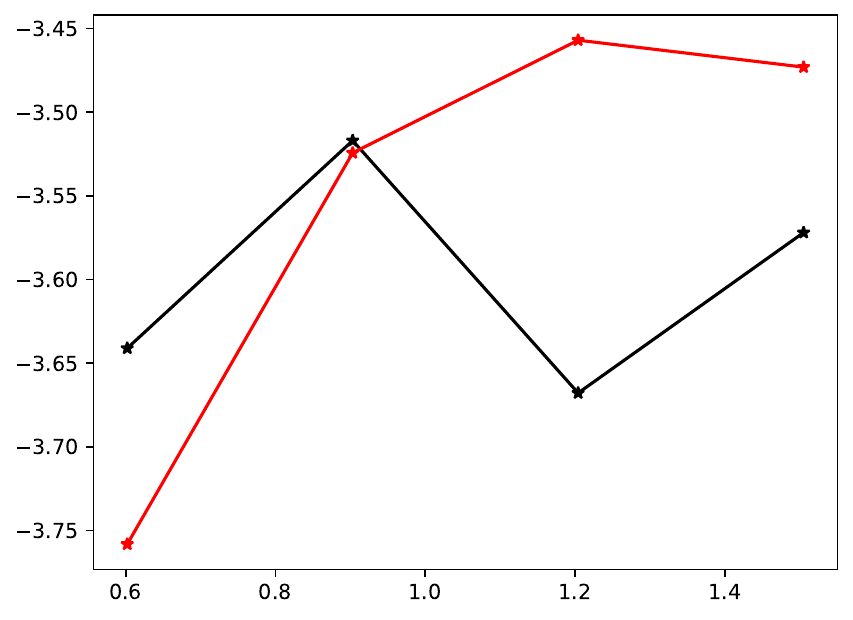}
         \caption{Error}\label{fig:porous_err}
     \end{subfigure}
     \centering
     \begin{subfigure}[b]{0.31\textwidth}
         \centering
         \includegraphics[width=\textwidth]{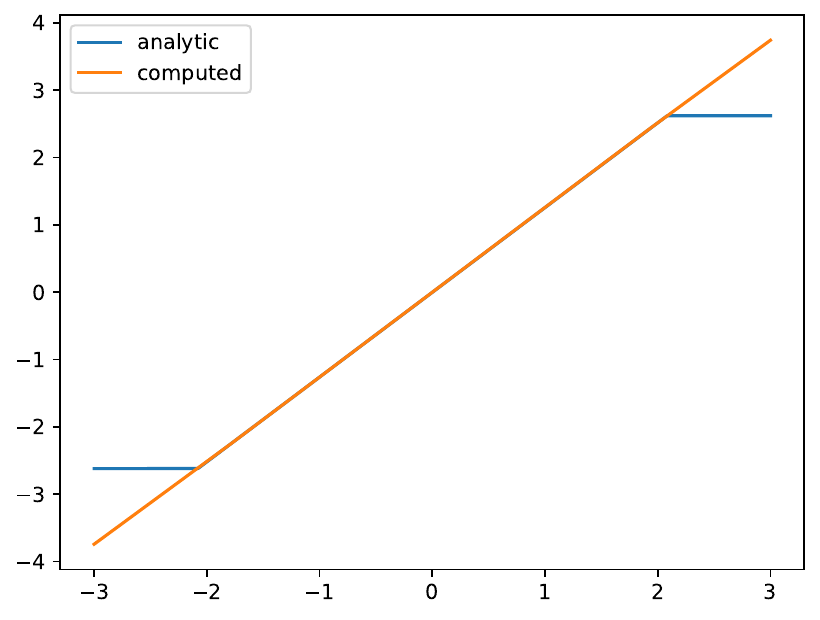}
         \caption{Mapping comparison}\label{fig:porous_map}
     \end{subfigure}
     \begin{subfigure}[b]{0.33\textwidth}
         \centering
         \includegraphics[width=\textwidth]{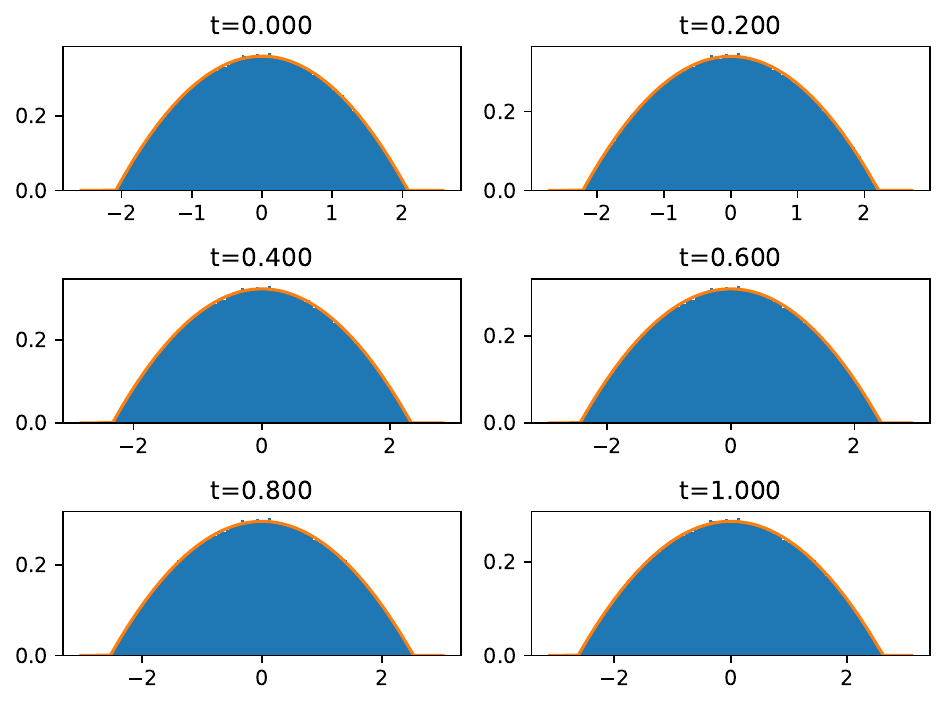}
         \caption{Density evolution}\label{fig:porous_density}
     \end{subfigure}
        \caption{Left: log-log plot of porous medium equation. The y-axis represents $\log_{10}$ error defined by \eqref{eq:error_def}. x-axis represents $\log_{10}(N)$. The bias terms $b_i$ are initialized based on Section \ref{nn struct} with $B=3^{2/3}$. Red lines represent results when only the weights terms are updated. Black lines represent results when both weights and bias are updated. Middle: mapping comparison between $T(t,z)$ (using Eq.~\eqref{eq:trasnport_numerical}) and our computed solution $f(\theta_t,z)$. Right: density evolution of the porous medium equation. Orange curve represents the density $p(t,x)$ given by Eq.~\eqref{eq:porous_medium}. Blue rectangles represent the histogram of $10^6$ particles in 100 bins from $t=0$ to $t=1$.}
        \label{fig:porous}
\end{figure}

\subsection{Porous Medium Equation}\label{sec:PM}
We consider Example \ref{ex6} with the functional $U(p(x)) = \frac{1}{m-1}p(x)^m$, $m>1$. This choice of $U$ yields the porous medium equation
\begin{equation}\label{eq:porous_medium}
    \partial_t p(t,x) = \Delta p(t,x)^m\,.
\end{equation}
It is known that Eq.~\eqref{eq:porous_medium} admits solutions given by the Barenblatt profile 
\begin{equation}\label{eq:barenblatt}
    p(t,x) = (t_0+t)^{-\alpha}\Big(C-\beta \norm{x}^2(t_0+t)^{-2\alpha/d}\Big)^{\frac{1}{m-1}}_+ \,,
\end{equation}
where $x\in \RR^d$, $\alpha=\frac{d}{d(m-1)+2}$, $\beta=\frac{(m-1)\alpha}{2dm}$, $t_0>0$ and $C$ is a normalization constant so that Eq.~\eqref{eq:barenblatt} integrates to 1 for all $t\geq0$. In this example, we consider the case when $m=2$. Then $\alpha=\frac{1}{3}$, $\beta=\frac{1}{12}$ and $C = \frac{3^{1/3}}{4}$. Eq.~\eqref{eq:barenblatt} also suggests that the support of the reference measure $p_0(x)=p(x,0)$ is bounded in $[-3^{2/3}t_0^{1/3},3^{2/3}t_0^{1/3}]$, which could help us with initializing the bias. More precisely, we cound initialize our $b_i$'s following Section \ref{nn struct} with $B=3^{2/3}t_0^{1/3}$. In our experiment, we set $t_0 =1$. We use $dt=10^{-3}$ and run for 1000 steps. We use $M=10^6$ particles sampled from $p(x,0)$ defined in Eq.~\eqref{eq:barenblatt} using importance sampling to approximate $\mathbb{E}_{z\sim p_{\mathrm{r}}}\Big[\nabla_\theta \hat U(\frac{p_{\mathrm{r}}(z)}{D_zf(\theta,z)})\Big]$, where $\hat{U}(p) = p$. Similar to the case of Fokker-Planck equation, special care needs to be taken when evaluating $\partial b_i \mathbb{E}_{z\sim p_{\mathrm{r}}}\Big[\hat U(\frac{p_{\mathrm{r}}(z)}{D_zf(\theta,z)})\Big]$. Using similar analysis from Section \ref{sec:FPK}, we have that 
\begin{equation}
    \partial b_i \mathbb{E}_{z\sim p_{\mathrm{r}}}\Big[\hat U(\frac{p_{\mathrm{r}}(z)}{D_zf(\theta,z)})\Big] = \begin{cases}
         \frac{p_{\mathrm{r}}(b_i)^2}{D_zf(\theta,b_i-\delta)}-\frac{p_{\mathrm{r}}(b_i)^2}{D_zf(\theta,b_i)},\quad 1\leq i \leq N\,. \\
        \frac{p_{\mathrm{r}}(b_i)^2}{D_zf(\theta,b_i)}-\frac{p_{\mathrm{r}}(b_i)^2}{D_zf(\theta,b_i+\delta)},\quad N+1\leq i \leq 2N\,.
    \end{cases}
\end{equation}
Our results are demonstrated in Fig.~\ref{fig:porous}. In Fig.~\ref{fig:porous_map}, \ref{fig:porous_density} we have $N=32$; both the bias and weights are updated. 

\begin{figure}
     \centering
     \begin{subfigure}[b]{0.4\textwidth}
         \centering
         \includegraphics[width=\textwidth]{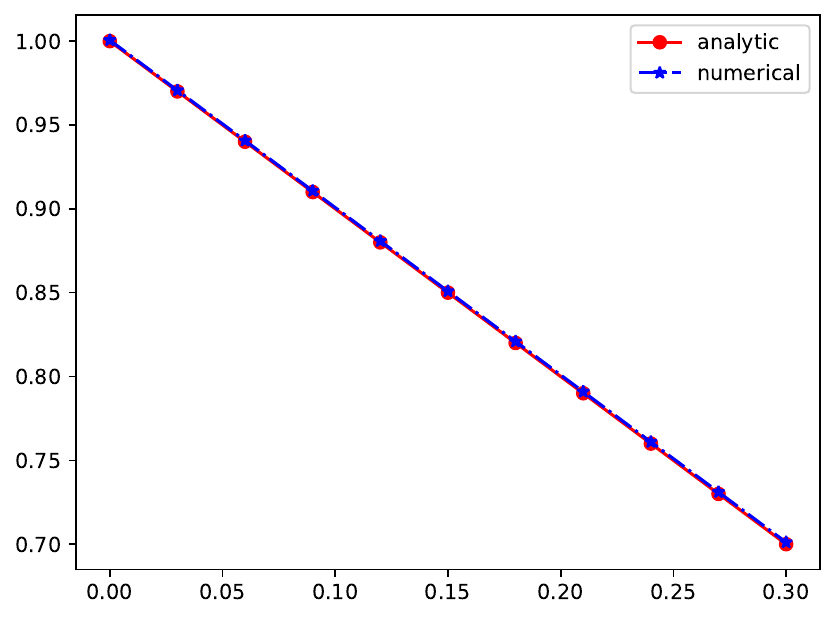}
         \caption{$\chi=1.5$}
     \end{subfigure}
     \begin{subfigure}[b]{0.4\textwidth}
         \centering
         \includegraphics[width=\textwidth]{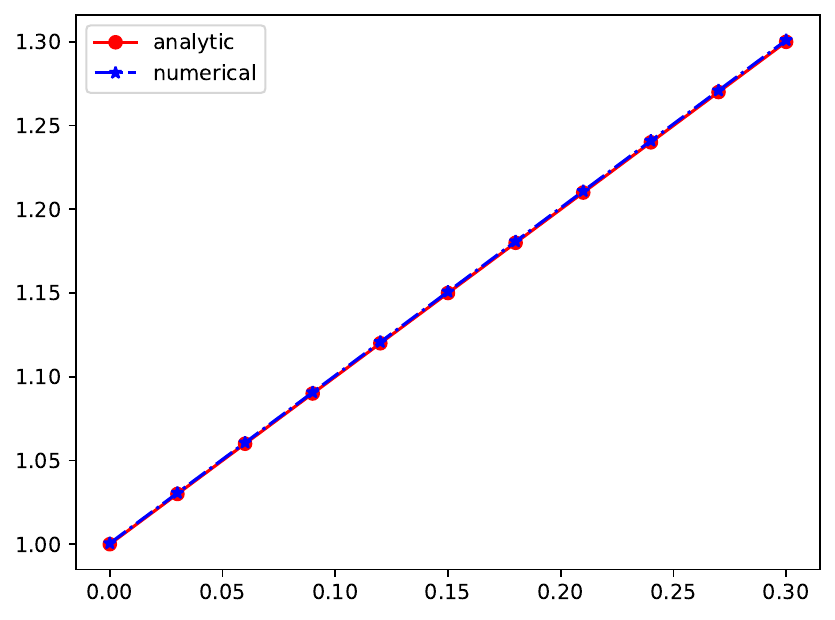}
         \caption{$\chi=0.5$}
     \end{subfigure}
        \caption{Second moment comparison between our numerical solution and analytic solution \eqref{eq:second_moment}. $x$-axis represents time.  }
        \label{fig:KS_2_moment}
\end{figure}

\begin{figure}
     \centering
     \begin{subfigure}[b]{0.4\textwidth}
         \centering
         \includegraphics[width=\textwidth]{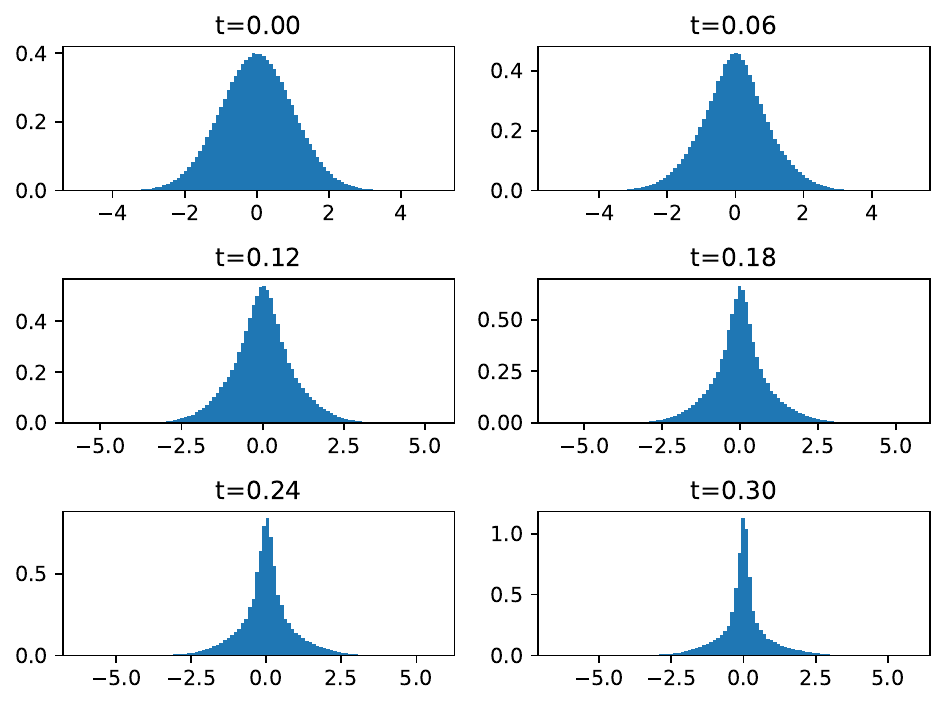}
         \caption{$\chi=1.5$}
     \end{subfigure}
     \begin{subfigure}[b]{0.4\textwidth}
         \centering
         \includegraphics[width=\textwidth]{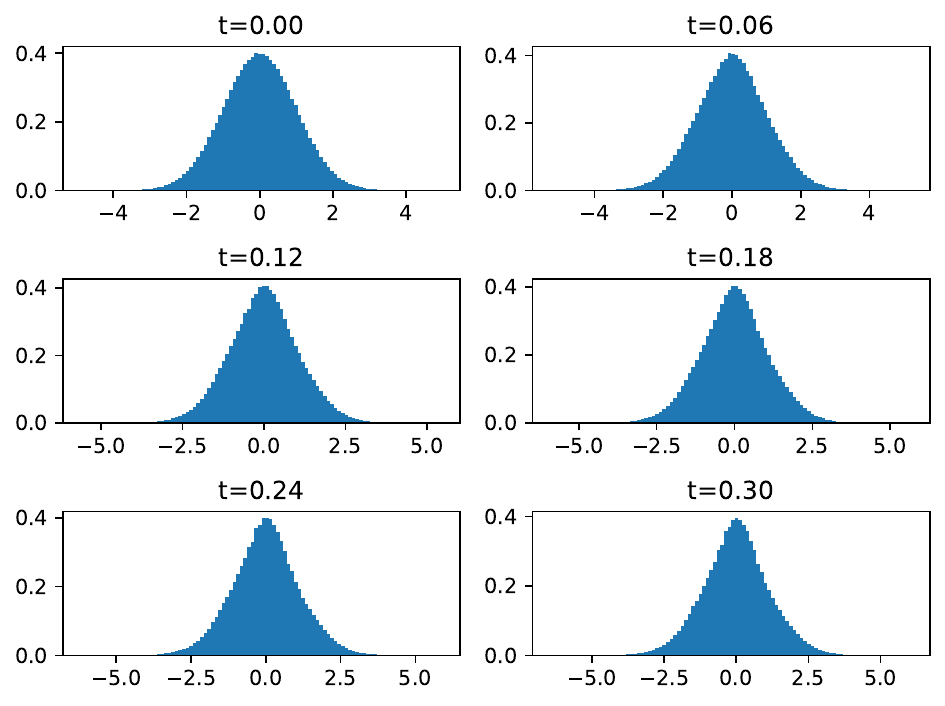}
         \caption{$\chi=0.5$}
     \end{subfigure}
        \caption{Density evolution of Keller-Segel equation with different $\chi$. Blue rectangles represent the histogram of $10^6$ particles in 100 bins from $t=0$ to $t=0.3$.  }
        \label{fig:KS_density}
\end{figure}
\begin{figure}
     \centering
     \begin{subfigure}[b]{0.4\textwidth}
         \centering
         \includegraphics[width=\textwidth]{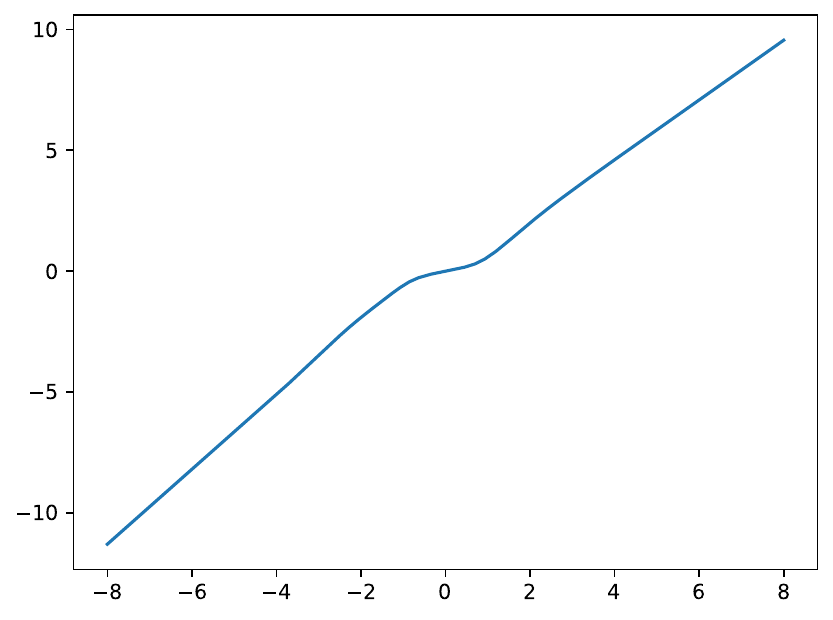}
         \caption{$\chi=1.5$}
     \end{subfigure}
     \begin{subfigure}[b]{0.4\textwidth}
         \centering
         \includegraphics[width=\textwidth]{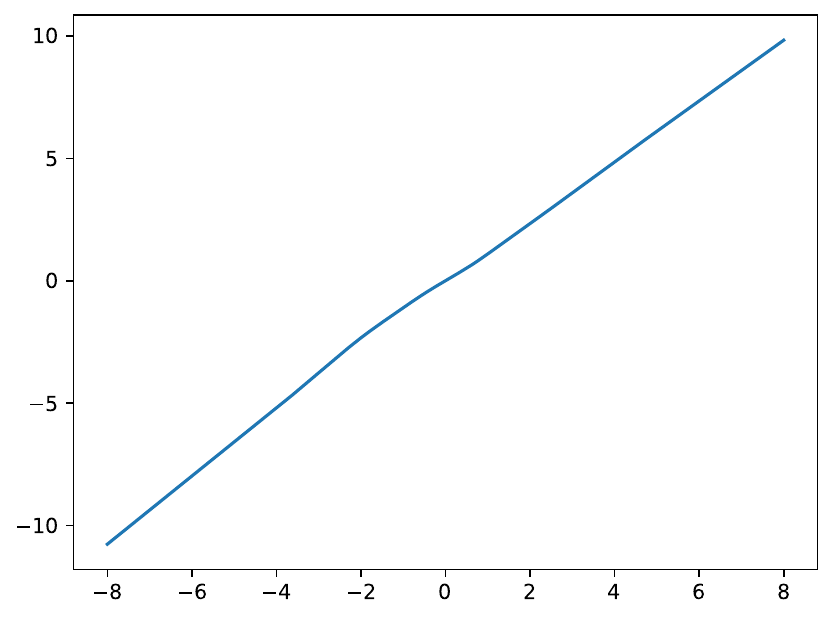}
         \caption{$\chi=0.5$}
     \end{subfigure}
        \caption{Lagrangian mapping of Keller-Segel equation with different $\chi$ at $t=0.3$. }
        \label{fig:KS_mapping}
\end{figure}
\subsection{Keller-Segel equation}
We consider the one-dimensional modified Keller-Segel equation, which is a combination of interaction energy in Example \ref{ex5} and  potential energy in Example \ref{ex6}: 
\begin{equation}
    \partial_t p(t,x) = \nabla \cdot \big( p(t,x) \nabla (U'(p) + W * p) \big),
\end{equation}
where $U(p) = p \log p$ and $W(x) = 2 \chi \log|x| $, $\chi>0$ is a constant. The second moment of $p(t,x)$ has an analytic form given by 
\begin{equation}\label{eq:second_moment}
    \mathbb E_{z\sim p(\cdot,t)} [z^2] = 2(1-\chi)t \, \mathbb E_{z\sim p(\cdot,0)} [z^2]\,.
\end{equation}
It is clear from Eq.~\eqref{eq:second_moment} that $\chi=1$ is a critical value. When $\chi>1$, the solution blows up as $t\to \infty$. So we consider two cases: $\chi=1.5$, and $\chi=0.5$. We present our results in Fig.~\ref{fig:KS_2_moment}, \ref{fig:KS_density} and \ref{fig:KS_mapping}. We used 2000 particles with a standard Gaussian initial distribution. We set $dt = 3\times10^{-4}$ and run for 1000 steps. The interaction term $W*p$ is evaluated using the 2000 particles with self-interaction excluded. We used a neural network with $N=32$ and $B=4$ following the setup and initialization in Section \ref{nn struct}. We update both the bias and weights terms in our experiment. 

\section{Discussion}
This paper analyzes the neural network projected dynamics for one-dimensional Wasserstein gradient flows of general energy functionals. For two-layer neural network functions with ReLU activations, we analyze the convergence and stability issues for the proposed numerical schemes from location parameter $b$ and scale parameter $a$. In numerical experiments, we demonstrate the second-order spatial domain accuracy as discussed in the numerical analysis. 

In future work, we shall study neural projected dynamics as a computational framework for building theoretical guaranteed machine learning numerical schemes. Various topics in this direction remain to be studied. First, we shall design neural network approximations to approximate the initial value of high-dimensional PDEs, which traditional PDE solvers cannot efficiently solve due to the curse of dimensionality. In particular, how can we understand the numerical accuracy of deep neural network functions in high dimensions when approximating PDEs?  
Next, we shall generalize the neural projected dynamics to dynamical systems for conservative-dissipative equations in statistical physics. The equation includes Hamiltonian structures induced from the conservative system and the related mean-field control problems. Furthermore, considering the closed relationship between the Wasserstein density manifold and sampling algorithms, we shall investigate sampling using the projected dynamics on neural parameter spaces and study their theoretical and statistical properties. We also consider the time-implicit (proximal-type) computations of the proposed algorithm \cite{lin2021wasserstein,li2019affine}, which could improve the performance and stability of the scheme.

\bibliographystyle{plain} 
\bibliography{Paper}

\end{document}